\pdfoutput=1
\documentclass[11pt,a4paper,twoside]{article}
\usepackage{amsmath,amssymb,amscd,amsthm,amsxtra,amsrefs}
\usepackage{mathtools}
\usepackage{mathrsfs}
\usepackage{thmtools}
\usepackage{textgreek}
\usepackage{authblk}
\usepackage{indentfirst}
\usepackage{fancyhdr}
\usepackage{float}
\usepackage{tikz-cd}
\usepackage{caption}
\usepackage[linktocpage=false]{hyperref}
\hypersetup{colorlinks,citecolor=purple,linkcolor=blue,urlcolor=cyan}

\theoremstyle{plain}
\newtheorem{proposition}{Proposition}[section]
\newtheorem{theorem}[proposition]{Theorem}
\newtheorem{lemma}[proposition]{Lemma}
\newtheorem{corollary}[proposition]{Corollary}
\newtheorem*{basiclemma}{Basic Lemma}

\theoremstyle{definition}
\newtheorem{definition}[proposition]{\emph{\textbf{Definition}}}

\theoremstyle{remark}
\newtheorem*{remark}{Remark}

\newcommand{\separate}{\medskip}
\newcommand{\tbf}{\textbf}
\newcommand{\tn}{\textnormal}
\newcommand{\noi}{\noindent}
\newcommand{\ph}{\phantom}

\newcommand{\op}{\operatorname}
\newcommand{\sse}{\subseteq}
\newcommand{\Equiv}{\Leftrightarrow}
\newcommand{\wt}{\widetilde}
\newcommand{\wh}{\widehat}
\newcommand{\mh}{^{\,\wh{}}}
\newcommand{\mrm}[1]{\mathrm{#1}}
\newcommand{\bb}[1]{\mathbb{#1}}
\newcommand{\Scr}[1]{\mathscr{#1}}
\newcommand{\re}[1]{#1_\mathrm{R}}
\newcommand{\im}[1]{#1_\mathrm{I}}
\newcommand{\abs}[1]{\vert#1\vert}
\newcommand{\Abs}[1]{\bigg\vert#1\bigg\vert}

\newcommand{\norm}[1]{{\Vert#1\Vert}}
\newcommand{\Norm}[1]{\bigg\Vert#1\bigg\Vert}
\newcommand{\conj}{\overline}
\newcommand{\wpln}{\bb{C}_\omega}
\newcommand{\sZ}{\bb{Z}^*}
\newcommand{\mZ}{\wt{ \bb{Z} }}
\newcommand{\rmd}{\mathrm{d}}
\newcommand{\rme}{\mathrm{e}}
\newcommand{\rmi}{\mathrm{i}}

\DeclareMathOperator{\sgn}{sgn}
\DeclareMathOperator{\imaginary}{Im}
\DeclareMathOperator{\real}{Re}
\DeclareMathOperator{\res}{Res}
\DeclareMathOperator*{\essup}{ess\!\sup}

\renewcommand{\=}{\ensuremath{\mathrel{\mathop:}=}}

\renewenvironment*{abstract}{\small\begin{quotation}{\noi\tbf{Abstract}.}}{\end{quotation}}
\providecommand{\keywords}[1]
{
  \small
  \textit{Keywords:\ }#1
}
\providecommand{\msc}[1]
{
  \small
  \textit{Mathematics Subject Classification:\ }#1
}

\makeatletter
\newcommand{\maxnorm}{\@ifstar\@maxnorms\@maxnorm}
\newcommand{\@maxnorms}[1]{\vert\mkern-1.5mu\vert\mkern-1.5mu\vert#1\vert\mkern-1.5mu\vert\mkern-1.5mu\vert}
\newcommand{\@maxnorm}[2][]{\mathopen{#1\vert\mkern-1.5mu#1\vert\mkern-1.5mu#1\vert}
#2\mathclose{#1\vert\mkern-1.5mu#1\vert\mkern-1.5mu#1\vert}}
\makeatother
\makeatletter
\renewcommand\section{\@startsection{section}{1}{0pt}{\baselineskip}{.1\baselineskip}{\centering\bfseries}}
\makeatother
\makeatletter
\renewcommand{\@seccntformat}[1]{\csname the#1\endcsname\csname adddot@#1\endcsname\ }
\newcommand{\adddot@section}{.}
\makeatother
\makeatletter
\renewcommand\subsection{\@startsection{subsection}{1}{0pt}{\baselineskip}{.1\baselineskip}{\centering\bfseries}}
\makeatother
\makeatletter
\def\dottednumberline#1{\hb@xt@\@tempdima{#1.\hfill}}
\renewcommand{\l@section}[2]{{\let\numberline\dottednumberline\@dottedtocline{2}{1em}{1em}{#1}{#2}}}
\renewcommand{\l@subsection}{\@dottedtocline{2}{1.5em}{1.8em}}
\makeatother

\numberwithin{equation}{subsection}
\renewcommand*{\theequation}{%
  \ifnum\value{subsection}=0 %
    \thesection
  \else
    \thesubsection
  \fi
  .\arabic{equation}%
}

\allowdisplaybreaks[3]

\begin{document}
\title{A semi-periodic initial-value problem for the Kadomtsev--Petviashvili II equation}
\author[1,\footnote{\href{mailto:kalpet@master.math.upatras.gr}{kalpet@master.math.upatras.gr}}]{P. Kalamvokas}
\author[1]{V.\,G. Papageorgiou}
\author[2]{A.\,S. Fokas}
\author[3]{L.-Y. Sung}
\affil[1]{Department of Mathematics, University of Patras, Greece}
\affil[2]{DAMTP, University of Cambridge, UK and Viterbi School of Engineering, University of Southern California, USA}
\affil[3]{Department of Mathematics, Louisiana Sate University, USA}
\date{March 19, 2023}
\maketitle
\thispagestyle{empty}
\begin{abstract}
\noi We investigate the Cauchy problem on the cylinder, namely the semi-periodic problem where there is periodicity in the $x$-direction and decay in the
$y$-direction, for the Kadomtsev--Petviashvili II equation by the inverse spectral transform method. For initial data with small $L^1$ and $L^2$ norms, 
assuming the zero mass constraint, this initial-value problem is reduced to a Riemann--Hilbert problem on the boundary of certain infinite strips with shift.
\end{abstract}
\keywords{Kadomtsev--Petviashvili equations, integrable PDEs, Lax pair, inverse spectral transform, Riemann--Hilbert problems}\\[5pt]
\msc{35Q53, 35Q15, 35Q35, 58F07}

\tableofcontents
\addtocontents{toc}{~\hfill\protect\textbf{Page}\vspace{5pt}}

\section[Introduction]{Introduction}\label{S:intro}
\medskip

The Kadomtsev--Petviashvili (\textsc{KP}) equation
\begin{equation}\label{Eq:kp}
( u_t + 6u u_x + u_{xxx} )_x = -3\sigma^2 u_{yy},
\end{equation}
where $\sigma^2 = \pm1$, was first introduced in 1970 by B.~Kadomtsev and V.~Petviashvili~\cite{KP70} as a model to study the evolution of long
ion-acoustic waves of small amplitude, propagating in plasma under the effect of long transverse perturbations and later derived as a model for surface
water waves \cite{FD75} and \cite{AS79} (see also \cite{SF85}). It may be thought of as a two spatial dimension analog of the Korteweg--de Vries
(KdV) equation
\begin{equation}\label{Eq:kdv}
u_t + 6u u_x + u_{xxx} = 0.
\end{equation}
The choice of the sign of $\sigma^2$ is critical with respect to the stability of solitons of the KdV equation subject to transverse perturbations
(in the $y$ direction). For $\sigma^2 = -1$, they are unstable, while for $\sigma^2 = 1$ they are stable.
The case $\sigma^2 = 1$ is known as the \textsc{KP}II equation and, in the context of fluid mechanics, appears in the study of long waves in shallow water 
under weak surface tension whereas the case $\sigma^2 = -1$ is called the \textsc{KP}I equation and can be employed to model water waves in thin films, 
where the very high surface tension dominates the gravitational force. The \textsc{KP} equation is one of the most notable integrable nonlinear evolution
PDEs in $2 + 1$ dimensions (i.e., two spatial and one temporal) and is solvable by use of the so called inverse scattering transform.

The Inverse scattering transform method can be viewed, as explained in \cite{AKNS74}, as a nonlinear analog of the Fourier transform and reduces the 
solution of the Cauchy problem to the solution of an inverse scattering problem for an associated linear eigenvalue equation. This method was discovered in 
1967 in the famous article~\cite{GGKM67}, as a way to solve the initial-value problem for the KdV equation with decaying initial data on the real line. The 
possibility of using the inverse scattering transform method for the \textsc{KP} equation follows from the existence of a Lax pair. Such a pair was discovered 
by Dryuma \cite{D74} and Zakharov and Shabat \cite{ZS74} independently. For the \textsc{KP}I equation, the possibility of implementing the inverse 
scattering transform method was suggested by Manakov \cite{M81} and Segur \cite {S82} and was implemented by Fokas and Ablowitz \cite{FA83}. This 
formulation was improved and corrected by several authors, as reviewed in \cite{F09}. The analysis of \textsc{KP}II was implemented by Ablowitz, Bar 
Yaacov and Fokas \cite{ABF83} using the so-called $\conj{\partial}$ formalism. This formalism was introduced earlier by Beals and Coifman 
in \cites{BC81,BC82} for the analysis of evolution PDEs in one space variable where the Riemann--Hilbert problem approach is not only adequate but also 
preferable. Rigorous aspects of the new methodology, often referred to as the inverse spectral transform (\textsc{IST}), were developed by several authors 
including Beals and Coifman \cites{BC84,BC85}. In particular, rigorous treatment of the \textsc{KP}II equation for the decaying in the plane problem was 
given by Wickerhauser in \cite{MV87} and by Fokas and Sung in \cite{FS92}.

The aim of this paper is to establish that another class of initial-value problems in $2 + 1$ dimensions can be incorporated in the above techniques of the 
\textsc{IST} scheme: those with initial data periodic in one spatial direction and decaying in the other. Associated with the \textsc{KP} equation there exist 
four such problems: \textsc{KP}I periodic in $x$, \textsc{KP}I periodic in $y$, \textsc{KP}II periodic in $x$, \textsc{KP}II periodic in $y$. In this work we 
consider the initial-value problem for the \textsc{KP}II equation, assuming that $u$ is a periodic function in the $x$ spatial variable, with period $2\ell > 0$, 
and decaying in the $y$ direction, i.e., we study the following Cauchy problem:
\begin{subequations}\label{Eq:cauchyproblem}\begin{alignat}{2}
( u_t + 6u u_x + u_{xxx} )_x &= -3u_{yy}, \label{Eq:cauchyproblem_a}\\[2pt]
u(x + 2\ell, y, t) &= u(x, y, t), \quad& &(x, y) \in \bb{R}^2, \ t \geq 0, \label{Eq:cauchyproblem_b}\\
u(x, y, 0) &= u_0(x, y), & &x \in [-\ell, \ell\,], \ y \in \bb{R}, \label{Eq:cauchyproblem_c}
\end{alignat}\end{subequations}
where $u(x, y, t) \to 0$ sufficiently rapidly as $\abs{y} \to \infty$ and $u_0(x, y)$ is a known function which belongs to some appropriate functional space, 
satisfying the \emph{zero mass} constraint, i.e.,
\begin{equation}\label{Eq:zeromass}
\int_{\!-\ell}^\ell u_0(x, y) \,\rmd x = 0, \quad \forall \:y \in \bb{R}.
\end{equation}
Modified accordingly, we believe that the method presented here can be used for solving the other three semi-periodic problems mentioned earlier.\newline
The zero mass constraint for the corresponding problem for $(x, y) \in \bb{R}^2$, arises when equation \eqref{Eq:kp} is put in evolution form, namely,
\begin{equation}\label{Eq:evolkp}
u_t + 6u u_x + u_{xxx} = -3\sigma^2 \partial_x^{-1} u_{yy},
\end{equation}
and a meaning of $\partial_x^{-1}$, corresponding to the initial data, has to be properly defined. The implications of this constraint was studied 
in \cite{AV91}. It was realized that if one chooses $\partial_x^{-1} = \int_{-\infty}^x$ or $\partial_x^{-1} = \int_x^\infty$, then the constraint
\begin{equation}\label{Eq:const}
\frac{\partial^2}{\partial y^2}\biggl(\int_{-\infty}^\infty u(x, y, t) \,\rmd x\biggr) = 0
\end{equation}
is required. However, given sufficiently decaying and smooth initial data, then
\begin{equation}\label{Eq:primit}
\partial_x^{-1} = \frac{1}{2}\biggl(\int_{-\infty}^x - \int_x^\infty\biggr),
\end{equation}
and no further constraints on the initial data appear. Interestingly, the authors found that even with the choice~\eqref{Eq:primit}, the condition 
\eqref{Eq:const} is eventually achieved. For more about the zero mass constraint see also \cite{MST07} and \cites{FS99,S99}.

Caudrey in \cite{C86} considered the Lax pair of the \textsc{KP} equation as a certain limit of a suitable $N \times N$ problem in $1 + 1$ dimensions after 
``discretising'' one of the spatial variables. Then, letting $N \to \infty$, Caudrey obtained formal results with initial data periodic in one direction and
decaying in the other. Here we obtain rigorous results by considering the semi-periodic problem ($x$-periodic) for \textsc{KP}II directly using the \textsc{IST}
method. We also mention \cite{GPDS09} for results on the Cauchy problem for a class of \textsc{KP}II equations with a general $x$-dispersion of order
$\geq 2$ including the classical \textsc{KP}II equation on the spatial domain $\bb{T}_x \times \bb{R}_y$ (periodicity in $x$) via PDE techniques. For a 
detailed review of results of the \textsc{KP} equations with \textsc{IST} and PDE methods we refer to the recent monograph \cite{KS22}.

The rigorous analysis carried throughout this paper follows the work of \cite{MV87}. In the rest of this section we give a brief description of our results.
The Lax pair associated with the \textsc{KP}II equation in our case is given by
\begin{subequations}\label{Eq:laxpair}\begin{align}
\op{L} &= -\partial_y + \partial_{xx} + u, \label{Eq:Lpart}\\
\op{M} &= 4\partial_{xxx} + 6u\partial_x + 3u_x + 3\int_{\!-\ell}^x u_y \,\rmd s + \emph{\text{\textalpha}}, \label{Eq:Mpart}
\end{align}\end{subequations}
i.e., the compatibility of the linear problems $\op{L}\psi = \lambda\psi$ and $\psi_t = \op{M}\psi$ with $\lambda_t = 0$, yields \textsc{KP}II. Here
$\emph{\text{\textalpha}}$ is an arbitrary constant, independent of $x$ and $y$. Later on it will acquire dependence on $z \in \bb{C}$. For small smooth 
functions $u \in L^1 \cap L^2( [-\ell, \ell\,] \times \bb{R} )$, the operator $\op{L}$ can be determined by the leading coefficients of asymptotically 
exponential functions in its kernel. More precisely, let $z \in \bb{C}$ such that $2\real z \neq \frac{\pi}{\ell}n$, for every nonzero integer $n$, and let
$\mu(x, y; z)$ be a bounded function such that the function $\psi(x, y; z) = \mu(x, y; z)\rme^{\rmi z x - z^2 y}$ is in the kernel of \eqref{Eq:Lpart}.
If $u(x, y)$ is small in $L^1 \cap L^2( [-\ell, \ell\,] \times \bb{R} )$, there exists a unique such $\mu$ with the asymptotic behaviour $\mu(x, y; z) \to 1$ as
$\abs{y} \to \infty$~(theorem \ref{Th:exist_unique}). Hence, $\psi$ is asymptotic to $\rme^{\rmi z x - z^2 y}$. This function $\mu$ is holomorphic for such 
complex numbers $z$~(theorem \ref{Th:holomorphicity}). If $u$ has first order partial derivatives in $L^1 \cap L^2( [-\ell, \ell\,] \times \bb{R} )$,
then $\mu$ satisfies a Riemann--Hilbert problem with a shift:
\begin{equation}\label{Eq:RHproblem}
\mu^+(x, y; z) - \mu^-(x, y; z) = F(z)\rme^{-\rmi( z + \conj{z} )x + (z^2 - \conj{z}^2)y} \mu^-( x, y; -\conj{z} ) \equiv \Scr{S}\mu,
\end{equation}
with
\begin{equation}
F(z) = \frac{ \sgn(-\real z) }{2\ell}\int_{\!-\infty}^\infty \int_{\!-\ell}^\ell
u(x, y)\mu^+(x, y; z)\rme^{\rmi( z + \conj{z} )x - (z^2 - \conj{z}^2)y} \,\rmd x\rmd y,
\end{equation}
where $\mu^\pm$ are the pointwise limits of $\mu$ from the left and from the right side of the lines $\real z = \frac{\pi}{\ell}n/2$
(theorem \ref{Th:zasympotics} and theorem \ref{Th:departurefromholomorphicity}). The function $F(z)$ determines the departure from holomorphicity of
$\mu$ across these lines. If $u$ has partial derivatives up to second order in $L^1 \cap L^2( [-\ell, \ell\,] \times \bb{R} )$, the spectral data $F(z)$ has 
enough decay to ensure the existence of a unique bounded solution $\mu$ to the Riemann--Hilbert problem~\eqref{Eq:RHproblem} with $\mu(x, y; z) \to 1$ 
as $\abs{\!\imaginary z} \to \infty$ (theorem \ref{Th:forwardscattering}). Hence, both $\mu$ and $\Scr{S}\mu$ are determined by $F(z)$.
Then, $u$ is determined by
\begin{equation}\label{Eq:eigenfunctions}
u(x, y) = \frac{1}{\pi}\partial_x \sum_{ \substack{n \in \bb{Z} \\ n \neq 0} } \int_{\real z = \frac{\pi}{\ell}n/2}
F(z)\rme^{-\rmi( z + \conj{z} )x + (z^2 - \conj{z}^2)y} \mu^-( x, y; -\conj{z} ) \,\rmd \imaginary z
\end{equation}
(theorem \ref{Th:inversescattering}). Therefore, knowledge of the spectral data $F(z)$ suffices to determine $u$.

The maps $u \mapsto F$ and $F \mapsto u$ might be called the forward and inverse spectral transforms, respectively. They behave like the Fourier 
transform and its inverse. If $u(x, y)$ has derivatives up to second order in $L^1 \cap L^2( [-\ell, \ell\,] \times \bb{R} )$, then $F(z)$ decays like
$( 1 + \abs{ (\real z, \real z \imaginary z) } )^{-2}$ (lemma \ref{L:rapiddeacy} and equation \eqref{Eq:datadecay}). On the other hand, if $F(z)$ decays like
$( 1 + \abs{ (\real z, \real z \imaginary z) } )^{-4}$, then $u(x, y)$ has derivatives up to second order in $L^2( [-\ell, \ell\,] \times \bb{R} )$ and a bounded 
Fourier transform (theorem \ref{Th:inversedecay}).

In order to complete the procedure of the \textsc{IST} and establish a solution to our problem, the time evolution of the spectral data $F(z, t)$ needs to be 
determined. Let $t > 0$ be thought of as time. If $u(x, y, t)$ evolves according to the \textsc{KP}II equation, then
$\frac{\rmd}{\rmd t}F(z, t) = -4\rmi(z^3 + \conj{z}^3)F(z, t)$ or $F(z, t) = F(z, 0)\rme^{-4\rmi(z^3 + \conj{z}^3)t}$ (lemma \ref{L:temporalevol}).
Since $z^3 + \conj{z}^3$ is real, one has $\abs{ F(z, t) } = \abs{ F(z, 0) }$ for all $z$ on the lines $\real z = \frac{\pi}{\ell}n/2$ and $t \geq 0$. By the 
forward and inverse spectral transform, there is a solution $u(x, y, t)$ for all time to the initial-value problem \eqref{Eq:cauchyproblem}:
if the initial value $u_0(x, y)$ is sufficiently small and has derivatives up to order eight in $L^1 \cap L^2( [-\ell, \ell\,] \times \bb{R} )$, then the initial
spectral data $F(z, 0)$ is known, hence the spectral data is known for all time. Thus, $\mu(x, y, t; z)$ is known for all time and finally $u(x, y, t)$ can be 
recovered via \eqref{Eq:eigenfunctions} (theorem \ref{Th:solution}).\newline
The following figure depicts the inverse spectral transform method.
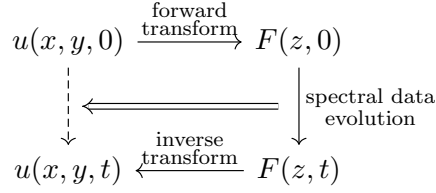
\begin{figure}[H]
\[
\begin{tikzcd}[row sep=large, column sep=large]
u(x, y, 0) \arrow{r}[above]{\substack{\text{forward} \\ \text{transform}}} \arrow[d, dashrightarrow, ""{name=D, right}]
 & F(z, 0) \arrow{d}[right, ""{name=U, left}]{\substack{\text{spectral data} \\ \text{evolution}}}\\
u(x, y, t)
 & \arrow{l}[above]{\substack{\text{inverse} \\ \text{transform}}} F(z, t)
\arrow[Rightarrow, from=U, to=D]
\end{tikzcd}
\]
\caption{Inverse spectral transform scheme\label{fig:istscheme}}
\end{figure}

For clarity of the exposition, let us collect here the notation that will frequently be used throughout this paper.
For a complex number $z$, $\re{z}$ and $\im{z}$ denote its real and imaginary part respectively, thus, $z = \re{z} + \rmi\im{z}$. $\bb{N}$ denotes the
set of natural numbers, i.e., $\bb{N} = \{1, 2, \dotsc\}$, and $\bb{Z}$ the set of integers while $\bb{N}_0 = \bb{N} \cup \{0\}$ and
$\sZ = \bb{Z} \setminus \!\{0\}$. $\Omega = \{ (x, y) \colon \!-\ell \leq x \leq \ell, \ y \in \bb{R} \}$, $\Scr{C} = \bb{Z} \times \bb{R}$ and
$\wpln = \{z \in \bb{C} \colon 2\re{z} \notin \omega\sZ\}$, where $\omega = \frac{\pi}{\ell}$. We write $E$ for the set $\Omega \times \wpln$ and
$\Scr{C}^*$ for the set $\sZ \times \bb{R}$ . For $1 \leq p \leq \infty$, we define the sets $L^p( \Scr{C} )$ by
\begin{equation}
L^p( \Scr{C} ) \= \{f \colon \Scr{C} \to \bb{C} \mid f \text{ is measurable and } \norm{f}_{ L^p( \Scr{C} ) } < \infty\},
\end{equation}
where
\begin{equation}
\norm{f}_{ L^p( \Scr{C} ) } \= \Bigg(\sum_{m = -\infty}^\infty \int_{\!-\infty}^\infty \,\abs{ f(m, \xi) }^p \,\rmd\xi\Bigg)^\frac{1}{p},
\quad 1 \leq p < \infty,
\end{equation}
and
\begin{equation}
\norm{f}_{ L^\infty( \Scr{C} ) } \= \essup_{ (m, \xi) \in \Scr{C} } \abs{ f(m, \xi) }.
\end{equation}
It is easy to establish that $( L^p( \Scr{C} ), \,\norm{\cdot}_{  L^p( \Scr{C} ) } )$ are Banach spaces. For a set $A$, other than
$\Scr{C}$ (or $\Scr{C}^*$), the standard notation for the $p$-norm of the Lebesgue spaces $L^p(A)$, will be used, i.e., $\norm{\cdot}_p$. $C(\Omega)$
is the set of continuous (complex) functions on $\Omega$ and $H(\wpln)$ is the set of holomorphic functions on $\wpln$.

The identity operator on a normed space is denoted by $\op{Id}$. If $\op{A}$ and $\op{B}$ are two operators, we denote with $[ \op{A}, \op{B} ]$ their
commutator, namely
\begin{equation}
[ \op{A}, \op{B} ] = \op{A}\op{B} - \op{B}\op{A},
\end{equation}
whenever this make sense. For a (linear) bounded operator $\op{T}$ between two normed spaces, we write $\norm{ \op{T} }_\mrm{op}$ for its norm.

If $f(x, y)$ is a function which is $2\ell$-periodic in $x$ and decaying in $y$ (i.e., $f$ and $\partial_y^n f$ tend to zero as $\abs{y} \to \infty$, for all 
$n \in \bb{N}$), define its Fourier transform by the integral
\begin{equation}
\wh{f}(m, \xi) \= \frac{1}{2\ell}\int_{\!-\infty}^\infty \int_{\!-\ell}^\ell f(x, y)\rme^{-\rmi\omega m x - \rmi\xi y} \,\rmd x\rmd y, \quad (m, \xi) \in \Scr{C},
\end{equation}
and the inverse transform by
\begin{equation}
f(x, y) = \big[ \wh{f}(m, \xi) \big]\spcheck\!(x, y) \= \frac{1}{2\pi}\sum_{m = -\infty}^\infty \int_{\!-\infty}^\infty
\wh{f}(m, \xi)\rme^{\rmi\omega m x + \rmi\xi y} \,\rmd\xi.
\end{equation}
This definition implies that $f \mapsto \wh{f}$ has norm less than $\omega$ as a map from $L^1(\Omega)$ to $L^\infty( \Scr{C} )$ and that when
$f \in L^1(\Omega) \cap L^2(\Omega)$, $\norm{\wh{f}\,}_{ L^2( \Scr{C} ) } = \sqrt{\omega}\norm{f}_2$. We will also use the notation
$\Scr{F}( f(x) )(k)$ for the Fourier transform of some function $f(x)$ in the real line namely
\[
\Scr{F}( f(x) )(k) \= \frac{1}{2\pi}\int_{\!-\infty}^\infty f(x)\rme^{-\rmi k x} \,\rmd x.
\]

The notation $\partial_v^n$ will be used to denote the $n$-th order partial derivative operator with respect to the variable $v$, thus,
$\partial_v^n = \partial^n\!/\partial v^n$, where $n$ is a \text{non-neg}\-ative integer. If $\alpha = (\alpha_1, \alpha_2)$ is a multi-index, then
\begin{equation}
\partial^\alpha = \partial_x^{\alpha_1} \partial_y^{\alpha_2}
\end{equation}
denotes a differential operator of order $\abs{\alpha} = \alpha_1 + \alpha_2$, with $x \in [-\ell, \ell\,]$ and $y \in \bb{R}$. If $\abs{\alpha} = 0$, then
$\partial^\alpha f(x, y) = f(x, y)$. For a point $s = (s_1, s_2)$ in $\bb{R}^2$ its norm is given by
\begin{equation}
\abs{s} = \sqrt{s_1^2 + s_2^2}.
\end{equation}
For a multi-index $\alpha = (\alpha_1, \alpha_2)$, we denote by $s^\alpha$ the monomial $s_1^{\alpha_1} s_2^{\alpha_2}$, which has degree
$\abs{\alpha}$ .

A sum with a prime next to it will mean summation over all integers except zero, i.e.,
\begin{equation}
\sideset{}{'}\sum_{m = -\infty}^\infty = \sideset{}{'}\sum_{ m \in \bb{Z} } = \sum_{m \in \sZ}.
\end{equation}

The Bachmann--Landau notation will be used in the normal way: given two functions $f$, $g$ defined on a set of real numbers, we write
\begin{equation}
f(x) = O( g(x) ),
\end{equation}
if $\abs{ f(x) } \leq M\abs{ g(x) }$, for some positive number $M$ and
\begin{equation}
f(x) = o( g(x) ),
\end{equation}
if $f(x)/g(x) \to 0$ as $x \to \infty$.

For brevity we will suppress the $t$-dependence for the functions until section \ref{S:time_evolution}.
\section[The Direct Problem]{The Direct Problem}\label{S:direct}
\smallskip
\subsection[Bounded Eigenfunctions of the Perturbed Heat Operator]{Bounded Eigenfunctions of the Perturbed Heat Operator}\label{s:bdd_eigen}
\separate

Consider the spectral problem for the operator $\op{L}$, i.e., the spectral equation
\begin{equation}\label{Eq:spectral}
\op{L}\psi = -\psi_y + \psi_{xx} + u\psi = \lambda\psi.
\end{equation}
We wish to reconstruct the potential $u$ through the spectral data of $\op{L}$. Using the transformation $\psi \to \psi \rme^{-\lambda y}$, the spectral
variable $\lambda$ is eliminated from equation \eqref{Eq:spectral}, hence we arrive at the equation
\begin{equation}\label{Eq:heat}
-\psi_y + \psi_{xx} + u\psi = 0,
\end{equation}
which is the well known one-dimensional heat equation, perturbed by $u$. If $u$ decays, then a class of solutions to this equation may be
specified---functions that are asymptotic to an exponential solution to the unperturbed equation. More precisely, introduce the \emph{Jost}
function $\mu$ defined by
\begin{equation}\label{Eq:Jost}
\psi(x, y; z) = \mu(x, y; z)\rme^{\rmi z x - z^2 y}; \qquad \lim_{\abs{y} \to \infty} \mu(x, y; z) = 1, \text{ for each } z,
\end{equation}
where $z$ is a complex variable and $\psi$ is a solution to equation \eqref{Eq:heat}. This function satisfies the boundary-value problem
\begin{subequations}\label{Eq:bvp}\begin{align}
&-\mu_y + \mu_{xx} + 2\rmi z\mu_x + u\mu = 0, \label{Eq:bvp_a}\\
&\text{ for each } z, \ \lim_{\abs{y} \to \infty} \mu(x, y; z) = 1.
\end{align}\end{subequations}
If we know $\mu$, we can determine $u$. Hence, we must show that for a given $u$ there is a unique function $\mu$ which solves \eqref{Eq:bvp}. 
Introduce the shifted derivatives
\begin{equation}\label{Eq:shifted}
D_1 = \partial_x + \rmi z, \ D_2 = \partial_y - z^2, \quad z \in \bb{C}.
\end{equation}
By ``completing the square'', rewrite equation \eqref{Eq:bvp_a} in operator form as a polynomial in $D_1, D_2$, i.e.,
\begin{equation}\label{Eq:shiftedeq}
[ (\partial_x + \rmi z)^2 - (\partial_y - z^2) ]\mu \equiv P( \partial + w(z) )\mu = -u\mu,
\end{equation}
where $\partial = (\partial_x, \partial_y)$, $w(z) = (\rmi z, -z^2)$ and $P(a, b) = a^2 - b$. Since $P( \partial + w(z) )$ annihilates the constant function 1, 
this equation can be written as
\begin{equation}
P( \partial + w(z) )(\mu - 1) = -u\mu.
\end{equation}
Applying the Fourier transform (with $t$ and $z$ considered parameters) yields
\begin{equation}\label{Eq:Transformed}
[ (\rmi\omega m + \rmi z)^2 - (\rmi\xi - z^2) ]( \wh{\mu - 1} )(m, \xi; z) = -\wh{u\mu}(m, \xi; z).
\end{equation}
Let us introduce some convenient notation:
\begin{equation}
q = (m, \xi) \in \Scr{C}, \quad P_z(m, \xi) = -P( \rmi(\omega m, \xi) + w(z) ).
\end{equation}
Then, $P_z(m, \xi) = (\omega m)^2 + 2\omega m z + \rmi\xi$, hence \eqref{Eq:Transformed} reads
\begin{equation}\label{Eq:bvp_atransformed}
P_z(m, \xi)( \wh{\mu - 1} )(m, \xi; z) = \wh{u\mu}(m, \xi; z).
\end{equation}
For each $z \in \bb{C}$, there are two distinct roots $(m, \xi)$ of $P_z(m, \xi)$, namely $(0, 0)$ and
\begin{equation}\label{Eq:nontrivialzero}
r_0(z) \equiv \Big( \frac{ -2\re{z} }{\omega}, 4\re{z}\im{z} \Big) = -\Big( \frac{ z + \bar{z} }{\omega}, \,\rmi(z^2 - \bar{z}^2) \Big).
\end{equation}
Using the linearity of the Fourier transform, we can write equation \eqref{Eq:bvp_atransformed} in the form
\begin{equation}\label{Eq:contractionform}
( \wh{\mu - 1} )(m, \xi; z) = \frac{ \wh{u}(m, \xi) }{ P_z(m, \xi) } + \frac{ ( \wh{u} \ast \wh{\mu - 1} )(m, \xi; z) }{ 2\pi P_z(m, \xi) },
\end{equation}
where $\ast$ denotes convolution in the $m$ and $\xi$ variables.

The study of this equation (equivalently of equation~\eqref{Eq:bvp_a}) is based on the following \tbf{basic} lemma which will be the main tool of our analysis 
and will be extensively used in the proofs of the theorems to follow.
\begin{basiclemma}\hypertarget{L:basiclemma}
Let $f \in L^2( \Scr{C} ) \cap L^\infty( \Scr{C} )$ such that $f(0, \xi) = 0$ for all $\xi \in \bb{R}$. Then,
\begin{equation}
(m, \xi) \mapsto \frac{ f(m, \xi) }{ P_z(m, \xi) } \in L^1 ( \Scr{C} ),
\end{equation}
and
\begin{equation}\label{Eq:basiclone}
\Norm{ \frac{f}{P_z} }_{ L^1 ( \Scr{C} ) } \equiv \sum_{m = -\infty}^\infty \int_{\!-\infty}^\infty \,\Abs{ \frac{ f(m, \xi) }{ P_z(m, \xi) } } \,\rmd\xi
\leq C\max\{ \norm{f}_{ L^2 ( \Scr{C} ) }, \norm{f}_{ L^\infty ( \Scr{C} ) } \},
\end{equation}
uniformly in $z \in \wpln$, where
\begin{equation}\label{Eq:constant}
C = \frac{4\pi^2}{3\omega^2} + \frac{\pi}{\omega}\sqrt{ \frac{\pi}{3} }.
\end{equation}
\end{basiclemma}
\begin{proof}
Since $f(0, \xi) = 0$ for every real number $\xi$, it is sufficient to show that
\[
\sideset{}{'}\sum_{m = -\infty}^\infty \int_{\!-\infty}^\infty \,\Abs{ \frac{ f(m, \xi) }{ P_z(m, \xi) } } \,\rmd\xi \leq
C\max\{ \norm{f}_{ L^2 ( \Scr{C} ) }, \norm{f}_{ L^\infty ( \Scr{C} ) } \}.
\]
Let $z \in \wpln$. Then, for $(m, \xi) \in \Scr{C}^*$,
\begin{align*}
P_z(m, \xi) &= (\omega m)^2 + 2\omega m\re{z} + \rmi( \xi + 2\omega m\im{z} )\\
&= ( \omega m + \re{z} )^2 - \re{z}^2 +\rmi( \xi + 2\omega m\im{z} ).
\end{align*}
Thus,
\begin{equation}
\sideset{}{'}\sum_{m = -\infty}^\infty \int_{\!-\infty}^\infty \,\Abs{ \frac{ f(m, \xi) }{ P_z(m, \xi) } } \,\rmd\xi = \sum_{m \in \mZ} \int_{\!-\infty}^\infty
\frac{ \abs{ f( \frac{ m - \re{z} }{\omega}, \xi - 2( m - \re{z} )\im{z} ) } }{ \abs{m^2 - \re{z}^2 + \rmi\xi} } \,\rmd\xi,
\end{equation}
where $\mZ = \omega\sZ + \re{z}$. Split the above integral into one with $\abs{\xi} < 1$ and its complement with $\abs{\xi} \geq 1$. Now,
\begin{align*}
\sum_{m \in \mZ} \int_{\!-1}^1 \frac{1}{ \abs{m^2 - \re{z}^2 + \rmi\xi} } \,\rmd\xi &= \sum_{m \in \mZ} \int_{\!-1}^1
\frac{1}{ \big[ (m^2 - \re{z}^2)^2 + \xi^2 \big]^\frac{1}{2} } \,\rmd\xi\\
&\leq \sum_{m \in \mZ} \int_{\!-1}^1 \frac{1}{ \abs{m^2 - \re{z}^2} } \,\rmd\xi\\
&= 2\sum_{m \in \mZ} \frac{1}{ \abs{m^2 - \re{z}^2} },
\end{align*}
and
\begin{align*}
\sum_{m \in \mZ} \int_{\abs{\xi} \geq 1} \frac{1}{\abs{m^2 - \re{z}^2 + \rmi\xi}^2} \,\rmd\xi
&= \sum_{m \in \mZ} \int_{\abs{\xi} \geq 1} \frac{1}{ (m^2 - \re{z}^2)^2 + \xi^2 } \,\rmd\xi\\
&= 2\sum_{m \in \mZ} \int_1^\infty \frac{1}{ (m^2 - \re{z}^2)^2 + \xi^2} \,\rmd\xi\\
&= 2\sum_{m \in \mZ} \int_1^\infty \frac{ \abs{m^2 - \re{z}^2}^{-1} }{1 + ( \xi/\abs{m^2 - \re{z}^2} )^2} \,\rmd( \xi/\abs{m^2 - \re{z}^2} )\\
&= \frac{\pi}{2}\sum_{m \in \mZ} \frac{1}{ \abs{m^2 - \re{z}^2} }.
\end{align*}
Therefore, the estimation of the sum
\[
\sum_{m \in \mZ} \frac{1}{ \abs{m^2 - \re{z}^2} },
\]
will lead to the desired result.

Since $m = \omega k + \re{z}$ and $\re{z} \neq \frac{\omega}{2}2k = \omega k$ for every nonzero integer $k$, we see that $m$ is not zero.
Thus, we can write
\[
\sum_{m \in \mZ} \frac{1}{ \abs{m^2 - \re{z}^2} } = \sum_{m \in \mZ^+} \frac{1}{ \abs{m^2 - \re{z}^2} } +
\sum_{m \in \mZ^-} \frac{1}{ \abs{m^2 - \re{z}^2} } = 2\sum_{m \in \mZ^+} \frac{1}{ \abs{m^2 - \re{z}^2} }.
\]
Now, for every pair of distinct, positive real numbers $a$ and $b$, the following inequality holds true:
\begin{equation}\label{Eq:basicinequality}
(a - b)^4 < (a^2 - b^2)^2,
\end{equation}
for $0 < 4a b$, so we can conclude that $(a - b)^2 < (a + b)^2$ and the inequality follows.
Suppose $\re{z} \neq 0$. Replacing $a$ and $b$ with $m$ and $\abs{ \re{z} }$ and taking square roots we arrive at
\[
( m - \abs{ \re{z} } )^2 < \abs{m^2 - \re{z}^2}.
\]
Thus,
\[
\sum_{m \in \mZ^+} \frac{1}{ \abs{m^2 - \re{z}^2} } < \sum_{m \in \mZ^+} \frac{1}{ ( m - \abs{ \re{z} } )^2 }.
\]
When $\re{z} > 0$,
\[
\sum_{m \in \mZ^+} \frac{1}{ ( m - \abs{ \re{z} } )^2 } = \frac{1}{\omega^2}\sum_{ \substack{m \in \sZ \\ \omega m + \re{z} > 0} } \frac{1}{m^2}
= \frac{1}{\omega^2}\sum_{ \substack{m > -\frac{ \re{z} }{\omega} \\ m \neq 0} } \frac{1}{m^2} < \frac{2}{\omega^2}\sum_{m = 1}^\infty \frac{1}{m^2}.
\]
Assume now that $\re{z} < 0$. Then, again
\begin{align*}
\sum_{m \in \mZ^+} \frac{1}{ ( m - \abs{ \re{z} } )^2 } &= \sum_{m \in \mZ^+} \frac{1}{ ( m + \re{z} )^2 } =
\sum_{m \in \mZ^-} \frac{1}{ ( m - \re{z} )^2 } \\
&= \frac{1}{\omega^2}\sum_{ \substack{m \in \sZ \\ \omega m + \re{z} < 0} } \frac{1}{m^2} =
\frac{1}{\omega^2}\sum_{ \substack{m < -\frac{ \re{z} }{\omega} \\ m \neq 0} } \frac{1}{m^2} \\
&< \frac{2}{\omega^2}\sum_{m = 1}^\infty \frac{1}{m^2}.
\end{align*}
Putting the parts together yields the estimate
\begin{equation}
\sum_{m \in \mZ} \frac{1}{ \abs{m^2 - \re{z}^2} } < \frac{4}{\omega^2}\sum_{m = 1}^\infty \frac{1}{m^2}.
\end{equation}
Observe that this inequality also holds when $\re{z} = 0$ for
\begin{equation*}
\sum_{m \in \mZ} \frac{1}{ \abs{m^2 - \re{z}^2} }  = \sum_{m \in \omega\sZ} \frac{1}{m^2} =
\frac{1}{\omega^2}\sideset{}{'}\sum_{m = -\infty}^\infty \frac{1}{m^2}
= \frac{2}{\omega^2}\sum_{m = 1}^\infty \frac{1}{m^2} < \frac{4}{\omega^2}\sum_{m = 1}^\infty \frac{1}{m^2}.
\end{equation*}

Returning to the original problem, we find that
\begin{align}
\sum_{m \in \mZ} \int_{\!-1}^1 \frac{ \abs{ f( \frac{ m - \re{z} }{\omega}, \xi - 2( m - \re{z} )\im{z} ) } }{ \abs{m^2 - \re{z}^2 + \rmi\xi} } \,\rmd\xi
&\leq \sum_{m \in \mZ} \int_{\!-1}^1 \frac{ \norm{f}_{ L^\infty( \Scr{C} ) } }{ \abs{m^2 - \re{z}^2 + \rmi\xi} } \,\rmd\xi \nonumber\\
&< \norm{f}_{ L^\infty( \Scr{C} ) } \frac{8}{\omega^2}\sum_{m = 1}^\infty \frac{1}{m^2},
\end{align}
and
\begin{align}
\sum_{m \in \mZ} \int_{\abs{\xi} \geq 1}
\frac{ \abs{ f( \frac{ m - \re{z} }{\omega}, \xi - 2( m - \re{z} )\im{z} ) } }{ \abs{m^2 - \re{z}^2 + \rmi\xi} } \,\rmd\xi
&\leq \Bigg(\sum_{m \in \mZ} \int_{\abs{\xi} \geq 1}
\frac{\norm{f}_{ L^2( \Scr{C} ) }^2}{\abs{m^2 - \re{z}^2 + \rmi\xi}^2} \,\rmd\xi\Bigg)^\frac{1}{2} \nonumber\\
&< \norm{f}_{ L^2( \Scr{C} ) } \Bigg( \frac{2\pi}{\omega^2}\sum_{m = 1}^\infty \frac{1}{m^2} \Bigg)^\frac{1}{2}.
\end{align}
The combination of these two inequalities yields inequality \eqref{Eq:basiclone} which concludes the proof.
\end{proof}
It is now possible to prove the basic theorem concerning the existence and uniqueness of the solution of equation \eqref{Eq:bvp_a}.
\begin{theorem}\label{Th:exist_unique}
Suppose the function $u(x, y)$  belongs to both $L^1(\Omega)$ and $L^2(\Omega)$ and is small in the sense that
\begin{equation}\label{Eq:smallnorm}
\max\{\omega\norm{u}_1, \sqrt{\omega}\norm{u}_2\} < \frac{2\pi}{C},
\end{equation}
where the constant $C$ is defined by~\eqref{Eq:constant}. Then, there is a unique, bounded solution $\mu(x, y; z)$ to the boundary-value problem
\begin{subequations}\label{Eq:BVP}\begin{align}
&(-\partial_y + \partial_x^2 +2\rmi z\partial_x)\mu + u\mu = 0, \label{Eq:BVP_a}\\
&\text{ for each } z \in \wpln, \ \lim_{\abs{y} \to \infty} \mu(x, y; z) = 1,
\end{align}\end{subequations}
such that $\wh{\mu - 1} \in L^1( \Scr{C} )$, for every $z \in \wpln$.
\end{theorem}
\begin{proof}
Consider the map $f \mapsto (\wh{u} \ast f)/2\pi P_z$. This map is bounded from $L^1( \Scr{C} )$ to $L^1( \Scr{C} )$, uniformly in $z \in \wpln$, and has 
norm less than one. Indeed, since $u \in L^1(\Omega) \cap L^2(\Omega)$, it follows that $\wh{u} \in L^2( \Scr{C} ) \cap L^\infty( \Scr{C} )$. Thus,
if $f \in L^1( \Scr{C} )$, then
\[
\norm{\wh{u} \ast f}_{ L^2( \Scr{C} ) } \leq \norm{ \wh{u} }_{ L^2( \Scr{C} ) } \norm{f}_{ L^1( \Scr{C} ) } =
\sqrt{\omega}\norm{u}_2 \norm{f}_{ L^1 ( \Scr{C} ) },
\]
and
\[
\norm{\wh{u} \ast f}_{ L^\infty ( \Scr{C} ) } \leq \norm{ \wh{u} }_{ L^\infty ( \Scr{C} ) } \norm{f}_{ L^1 ( \Scr{C} ) } <
\omega\norm{u}_1 \norm{f}_{ L^1 ( \Scr{C} ) }.
\]
Furthermore, $(\wh{u} \ast f)(0, \xi) = 0$ for every $\xi \in \bb{R}$ because $\wh{u}(0, \xi) = 0$, as a consequence of the zero mass constraint.
Hence, the function $\wh{u} \ast f$ satisfies the assumptions of the \hyperlink{L:basiclemma}{Basic Lemma} thus,
\begin{align*}
\Norm{ \frac{\wh{u} \ast f}{2\pi P_z} }_{ L^1( \Scr{C} ) } &\leq
\frac{C}{2\pi}\max\{ \norm{\wh{u} \ast f}_{ L^2( \Scr{C} ) }, \norm{\wh{u} \ast f}_{ L^\infty( \Scr{C} ) } \}\\
&< \frac{C}{2\pi}\max\{\sqrt{\omega}\norm{u}_2, \,\omega\norm{u}_1\}\norm{f}_{ L^1( \Scr{C} ) },
\end{align*}
uniformly in $z \in \wpln$. By assumption, $\frac{C}{2\pi}\max \{\omega\norm{u}_1, \sqrt{\omega}\norm{u}_2\} < 1$.

Applying Banach's fixed-point theorem in $L^1( \Scr{C} )$, equation \eqref{Eq:contractionform} has a unique solution $(\wh{\mu - 1})(m, \xi; z)$ for each 
$z \in \wpln$. Its inverse Fourier transform $\mu(x, y; z)$ solves equation \eqref{Eq:BVP_a}. Furthermore, $\mu \in L^\infty(E)$: for every $z \in \wpln$
\[
\abs{ (\mu - 1)(x, y; z) } = \big\lvert [ \wh{\mu - 1} ]\spcheck\!(x, y; z) \big\rvert \leq
\frac{1}{2\pi}\big\lVert (\wh{\mu - 1})(\cdot, \cdot; z) \big\rVert_{ L^1( \Scr{C} ) }.
\]
But
\begin{align*}
\big\lVert (\wh{\mu - 1})(\cdot, \cdot; z) \big\rVert_{ L^1( \Scr{C} ) } &\leq \Norm{ \frac{ \wh{u} }{P_z} }_{ L^1( \Scr{C} ) }
+ \Norm{ \frac{ \wh{u} \ast ( \wh{\mu - 1} )(\cdot, \cdot; z) }{2\pi P_z} }_{ L^1( \Scr{C} ) }\\
&< C\maxnorm{u} + \frac{C}{2\pi}\maxnorm{u} \big\lVert ( \wh{\mu - 1} )(\cdot, \cdot; z) \big\rVert_{ L^1( \Scr{C} ) },
\end{align*}
and so
\begin{equation}\label{Eq:lonenorm}
\big\lVert ( \wh{\mu - 1} )(\cdot, \cdot; z) \big\rVert_{ L^1( \Scr{C} ) } < \frac{ C\maxnorm{u} }{ 1 - \frac{1}{2\pi}C\maxnorm{u} },
\end{equation}
where we set $\maxnorm{\cdot} \equiv \max\{\omega\norm{\cdot}_1, \sqrt{\omega}\norm{\cdot}_2\}$.

Equation \eqref{Eq:contractionform} can be written in the form
\begin{equation}\label{Eq:eigenfunction}
\mu(x, y; z) = 1 + \frac{1}{2\pi}\sideset{}{'}\sum_{m = -\infty}^\infty \int_{-\infty}^\infty
\frac{ \wh{u\mu}(m, \xi; z) }{ P_z(m, \xi) }\rme^{\rmi\omega m x + \rmi\xi y} \,\rmd\xi,
\end{equation}
for $(x, y) \in \Omega$, $z \in \wpln$. Since
\begin{equation}
\frac{ \wh{u\mu}(m, \xi; z) }{ P_z(m, \xi) } \in L^1( \Scr{C} ),
\end{equation}
the Riemann--Lebesgue lemma implies that $\mu(x, y; z) \to 1$ as $\abs{y} \to \infty$ for each $z \in \wpln$.
\end{proof}
More precise asymptotics of $\mu(x, y; z)$ as $\abs{y} \to \infty$ is given by the following proposition.
\begin{proposition}\label{Pr:asymptotics}
Suppose $\mu(x, y; z)$ is a solution to equation \eqref{Eq:eigenfunction} in $L^\infty(E)$ and that
\begin{equation}\label{Eq:decaycondition}
( 1 + \abs{y} )\abs{ u(x, y) } \in L^1(\Omega) \cap L^2(\Omega).
\end{equation}
Then,
\begin{equation}
\mu(x, y; z) = 1 + o\bigg( \frac{1}{ \abs{y} } \bigg), \text { as } \abs{y} \to \infty.
\end{equation}
\end{proposition}
\begin{proof}
To see why this is true, let us calculate $\partial_\xi( \wh{u\mu}(m, \xi; z)/P_z(m, \xi) )$:
\begin{align*}
\partial_\xi\bigg( \frac{ \wh{u\mu}(m, \xi; z) }{ P_z(m, \xi) } \bigg) &=
\frac{ \partial_\xi( \wh{u\mu}(m, \xi; z) )P_z(m, \xi) - \rmi\wh{u\mu}(m, \xi; z) }{ {P_z}^2(m, \xi) }\\
&= \frac{ \partial_\xi\wh{u\mu}(m, \xi; z) }{ P_z(m, \xi) } - \rmi\frac{ \wh{u\mu}(m, \xi; z) }{ {P_z}^2(m, \xi) }\\
&= \frac{ [-\rmi y u\mu]\mh(m, \xi; z) }{ P_z(m, \xi) } - \rmi\frac{ \wh{u\mu}(m, \xi; z) }{ {P_z}^2(m, \xi) }.
\end{align*}
Since $\mu$ is bounded, \eqref{Eq:decaycondition} guarantees that $[-\rmi y u\mu]\mh$ belongs to $L^2( \Scr{C} )$ and $L^\infty( \Scr{C} )$. Hence,
by the \hyperlink{L:basiclemma}{Basic Lemma}, $[-\rmi y u\mu]\mh/P_z \in L^1( \Scr{C} )$. Now a simple calculation shows that
\[
\sideset{}{'}\sum_{m = -\infty}^\infty \int_{\!-\infty}^\infty \frac{1}{ {P_z}^2(m, \xi) } \,\rmd\xi =
\frac{\pi}{2}\sum_{m \in \mZ} \frac{1}{ \abs{m^2 - \re{z}^2} },
\]
and the sum in the right hand side converges. By the boundedness of $\mu$, it follows that $\wh{u\mu} \in L^2( \Scr{C} ) \cap L^\infty( \Scr{C} )$.
Thus, $\wh{u\mu}/{P_z}^2$ is also a member of $L^1( \Scr{C} )$. Therefore, by the Riemann--Lebesgue lemma
\[
\bigg[ \partial_\xi\bigg( \frac{ \wh{u\mu}(m, \xi; z) }{ P_z(m, \xi) } \bigg) \bigg]\spcheck\!(x, y; z) \to 0, \text{ as } \abs{y} \to \infty,
\]
and the result follows from the identity
\[
\bigg[ \partial_\xi\bigg( \frac{ \wh{u\mu}(m, \xi; z) }{ P_z(m, \xi) } \bigg) \bigg]\spcheck =
-\rmi y\bigg[ \frac{ \wh{u\mu}(m, \xi; z) }{ P_z(m, \xi) } \bigg]\spcheck.\qedhere
\]
\end{proof}
The eigenfunction $\mu$ has several regularity properties. To derive them we will use equation \eqref{Eq:eigenfunction}. An immediate result is that
$\mu(x, y; z)$ belongs to $C(\Omega)$ for every $z \in \wpln$.
\begin{proposition}\label{Pr:continuity}
Suppose $\mu(x, y; z)$ is a solution to equation \eqref{Eq:eigenfunction}. Then, for every $z \in \wpln$, $\mu$ is continuous in $y$ for all
$x \in [-\ell, \ell\,]$ and continuous in $x$ for all $y \in \bb{R}$.
\end{proposition}
\begin{proof}
Let $x \in [-\ell, \ell\,]$ and $ \{y_n\}_{n = 1}^\infty$ be a sequence of real numbers, converging to $y_0$. Define the sequence of functions
$\{f_n\}_{n = 1}^\infty$,
\[
f_n(x, m, \xi; z) = \frac{ \wh{u\mu}(m, \xi; z) }{ P_z(m, \xi) }\rme^{\rmi\omega m x + \rmi\xi y_n}.
\]
For every $n \in \bb{N}$ the function $\rme^{\rmi\omega m x + \rmi\xi y_n}$ is continuous, hence measurable. Thus, the functions $f_n(x, m, \xi; z)$, 
being the product of measurable functions, are measurable for all $n$. From the continuity of the exponential function, it follows that the sequence
$\{f_n\}$ converges pointwise to the function
\[
\frac{ \wh{u\mu}(m, \xi; z) }{ P_z(m, \xi) }\rme^{\rmi\omega m x + \rmi\xi y_0}.
\]
Moreover,
\[
\abs{ f_n(x, m, \xi; z) } = \Abs{ \frac{ \wh{u\mu}(m, \xi; z) }{ P_z(m, \xi) } }.
\]
But since the function $\abs{\wh{u\mu}/P_z}$ belongs to $L^1( \Scr{C} )$, an application of Lebesgue's dominated convergence theorem yields the 
following:
\begin{align*}
\lim_{n \to \infty} \mu(x, y_n; z) &= 1 + \frac{1}{2\pi}\lim_{n \to \infty} \sideset{}{'}\sum_{m = -\infty}^\infty \int_{-\infty}^\infty
f_n(x, m, \xi; z) \,\rmd\xi\\
&= 1 + \frac{1}{2\pi}\sideset{}{'}\sum_{m = -\infty}^\infty \int_{-\infty}^\infty \lim_{n \to \infty} f_n(x, m, \xi; z) \,\rmd\xi\\
&= 1 + \frac{1}{2\pi}\sideset{}{'}\sum_{m = -\infty}^\infty \int_{-\infty}^\infty
\frac{ \wh{u\mu}(m, \xi; z) }{ P_z(m, \xi) }\rme^{\rmi\omega m x + \rmi\xi y_0} \,\rmd\xi\\
&= \mu(x, y_0; z).
\end{align*}

Using similar arguments, we can easily see that $\mu$ is also continuous with respect to $x$ for every $y \in \bb{R}$, hence the result follows.
\end{proof}
The following proposition provides us with an alternative way of writing equation \eqref{Eq:eigenfunction}, which will proven to be quite useful, in particular 
proving analytic properties of $\mu$ with respect to $z$.
\begin{proposition}\label{Pr:Neumannseries}
Suppose $u(x, y)$ satisfies condition \eqref{Eq:smallnorm}. Then, $\mu$ admits the representation in Neumann series
\begin{equation}\label{Eq:Neumannseries}
\mu(x, y; z) = \sum_{n = 0}^\infty (\Scr{N}_u^n 1)(x, y; z),
\end{equation}
where the operator $\Scr{N}_u$ is defined for every function $h \in L^\infty(\Omega)$ by
\begin{equation}\label{Eq:Neumann}
(\Scr{N}_u h)(x, y; z) \= \frac{1}{2\pi}\sideset{}{'}\sum_{m = -\infty}^\infty \int_{-\infty}^\infty
\frac{ \wh{u h}(m, \xi) }{ P_z(m, \xi) }\rme^{\rmi\omega m x + \rmi\xi y} \,\rmd\xi.
\end{equation}
\end{proposition}
\begin{proof}
Let $h \in  L^\infty(\Omega)$. Then, $u h \in L^1(\Omega) \cap L^2(\Omega)$, since $u \in L^1(\Omega) \cap L^2(\Omega)$, hence
$\wh{u h} \in L^2( \Scr{C} ) \cap L^\infty( \Scr{C} )$. Thus, from the \hyperlink{L:basiclemma}{Basic Lemma},
\[
\sideset{}{'}\sum_{m = -\infty}^\infty \int_{-\infty}^\infty \frac{ \wh{u h}(m, \xi) }{ P_z(m, \xi) } \,\rmd\xi \leq
C\max\{ \norm{ \wh{u h} }_{ L^2( \Scr{C} ) }, \norm{ \wh{u h} }_{ L^\infty( \Scr{C} ) } \},
\]
uniformly in $z \in \wpln$. This yields,
\begin{align*}
\abs{ (\Scr{N}_u h)(x, y; z) } &\leq \frac{1}{2\pi}C\max\{ \norm{ \wh{u h} }_{ L^2( \Scr{C} ) }, \norm{ \wh{u h} }_{ L^\infty( \Scr{C} ) } \}\\
&\leq \frac{1}{2\pi}C\max\{\omega\norm{u}_1, \sqrt{\omega}\norm{u}_2\}\norm{h}_\infty.
\end{align*}
Therefore, for all $z \in \wpln$, the operator $\Scr{N}_u \colon L^\infty(\Omega) \to L^\infty(\Omega)$ is bounded with norm less than one. Write 
equation \eqref{Eq:eigenfunction} as $\mu = 1 + \Scr{N}_u \mu$. Then, $(\op{Id} - \Scr{N}_u)\mu = 1$, and since $\norm{\Scr{N}_u}_\mrm{op} < 1$
we are allowed to write
\begin{equation}
\mu = (\op{Id} - \Scr{N}_u)^{-1} 1 = \sum_{n = 0}^\infty \Scr{N}_u^n 1.\qedhere
\end{equation}
\end{proof}
\begin{remark}
The Neumann series \eqref{Eq:Neumannseries} converges uniformly for $z \in \wpln$. This can be deduced by an application of the Weierstrass M-test:
for $n \in \bb{N}$,
\begin{equation}\label{Eq:Mtest}
\abs{ (\Scr{N}_u^n 1)(x, y; z) } \leq \bigg( \frac{1}{2\pi}C\maxnorm{u} \bigg)^n,
\end{equation}
while the geometric series
\begin{equation}
\sum_{n = 0}^\infty \bigg( \frac{1}{2\pi}C\maxnorm{u} \bigg)^n,
\end{equation}
converges since $C\maxnorm{u}/2\pi < 1$.
\end{remark}
A consequence of the representation \eqref{Eq:Neumannseries}, is that $\mu(x, y; z)$ is a holomorphic function with respect to $z \in \wpln$.
\begin{theorem}\label{Th:holomorphicity}
Suppose $u(x, y)$ belongs to $L^1(\Omega) \cap L^2(\Omega)$ such that $\maxnorm{u}$ is small. Then, for every $(x, y) \in \Omega$,
$\mu(x, y; \cdot) \in H(\wpln)$.
\end{theorem}
\begin{proof}
We will show that each function $(\Scr{N}_u^n 1)(x, y; z)$ in the Neumann series \eqref{Eq:Neumannseries} of $\mu(x, y; z)$ is a holomorphic function 
with respect to $z \in \wpln$. Since the series converges uniformly in $\wpln$, it defines a holomorphic function there.

We will use induction. Let $z_0 \in \wpln$. For $(m, \xi) \in \Scr{C}^*$, the function $1/P_z(m, \xi)$ is holomorphic in $z \in \wpln$. By
the \hyperlink{L:basiclemma}{Basic Lemma},
\[
\sideset{}{'}\sum_{ m \in \bb{Z} } \int_{\!-\infty}^\infty \frac{ \wh{u}(m, \xi) }{ P_z(m, \xi) }\rme^{\rmi\omega m x + \rmi\xi y} \,\rmd\xi
\]
converges absolutely and uniformly in $\wpln$. Hence, dominated convergence yields
\begin{align*}
\lim_{z \to z_0} (\Scr{N}_u 1)(x, y; z) &= \frac{1}{2\pi}\sideset{}{'}\sum_{ m \in \bb{Z} } \int_{\!-\infty}^\infty \lim_{z \to z_0}
\frac{ \wh{u}(m, \xi) }{ P_z(m, \xi) }\rme^{\rmi\omega m x + \rmi\xi y} \,\rmd\xi\\
&= \frac{1}{2\pi}\sideset{}{'}\sum_{ m \in \bb{Z} } \int_{\!-\infty}^\infty
\frac{ \wh{u}(m, \xi) }{ P_{z_0}(m, \xi) }\rme^{\rmi\omega m x + \rmi\xi y} \,\rmd\xi,
\end{align*}
which shows that $\Scr{N}_u 1$ is continuous at $z_0$ and therefore, continuous in $\wpln$. Now let $T$ be a triangle in $\wpln$. By Fubini's theorem,
\[
\int_T (\Scr{N}_u 1)(x, y; z) \,\rmd z = \frac{1}{2\pi}\sideset{}{'}\sum_{ m \in \bb{Z} } \int_{\!-\infty}^\infty \wh{u}(m, \xi)
\bigg(\int_T \frac{1}{ P_z(m, \xi) } \,\rmd z\bigg)\rme^{\rmi\omega m x + \rmi\xi y} \,\rmd\xi = 0.
\]
Applying Morera's theorem, we conclude that $(\Scr{N}_u 1)(x, y; \cdot) \in H(\wpln)$.

Suppose now that $(\Scr{N}_u^k 1)(x, y; \cdot) \in H(\wpln)$ for all $k < n$. Since $\Scr{N}_u^{n -1} 1$ is bounded for every $z \in \wpln$, and $u$ 
belongs to $L^1(\Omega) \cap L^2(\Omega)$, it follows that for every $z \in \wpln$,
$[u\Scr{N}_u^{n - 1} 1]\mh \in L^2( \Scr{C} ) \cap L^\infty( \Scr{C} )$. Therefore, \hyperlink{L:basiclemma}{Basic Lemma} implies that
\[
\sideset{}{'}\sum_{ m \in \bb{Z} } \int_{\!-\infty}^\infty
\frac{ [u\Scr{N}_u^{n - 1} 1]\mh(m, \xi; z) }{ P_z(m, \xi) }\rme^{\rmi\omega m x + \rmi\xi y} \,\rmd\xi
\]
converges absolutely and uniformly for all $z \in \wpln$. For each $(m, \xi) \in \Scr{C}^*$, the function $[u\Scr{N}_u^{n - 1} 1]\mh(m, \xi; z)$ is
holomorphic in $z \in \wpln$: by definition
\[
[u\Scr{N}_u^{n - 1} 1]\mh(m, \xi; z) = \frac{1}{2\ell}\int_{\!-\infty}^\infty \int_{\!-\ell}^\ell
u(x, y)(\Scr{N}_u^{n - 1} 1)(x, y; z)\rme^{-\rmi\omega m x - \rmi\xi y} \,\rmd x\rmd y.
\]
By the induction hypothesis, $\Scr{N}_u^{n - 1} 1$ is holomorphic. Also, as a consequence of inequality \eqref{Eq:Mtest}, for every $z \in \wpln$
\[
\int_{\!-\infty}^\infty \int_{\!-\ell}^\ell \abs{ u(x, y)(\Scr{N}_u^{n - 1} 1)(x, y; z) } \,\rmd x\rmd y \leq
\Big( \frac{ C\maxnorm{u} }{2\pi} \Big)^{n - 1} \norm{u}_1 < \infty.
\]
Thus, from the continuity of $\Scr{N}_u^{n - 1} 1$ and dominated convergence, we get that\\$[u\Scr{N}_u^{n - 1} 1]\mh$ is continuous and, changing the 
order of integration via Fubini's theorem,
\[
\int_T [u\Scr{N}_u^{n - 1} 1]\mh(m, \xi; z) \,\rmd z = 0.
\]

But now, continuity of the function $[u\Scr{N}_u^{n - 1} 1]\mh(m, \xi; z)/P_z(m, \xi)$ and dominated convergence implies the continuity of
$(\Scr{N}_u^n 1)(x, y; z)$ in $\wpln$ and as the function\\
$[u\Scr{N}_u^{n - 1} 1]\mh(m, \xi; z)/P_z(m, \xi)$ is holomorphic in $z \in \wpln$,
yet another application of Fubini's theorem gives
\begin{align*}
\int_T (\Scr{N}^n_u 1)(x, y; z) \,\rmd z &= \int_T \Bigg(\frac{1}{2\pi}\sideset{}{'}\sum_{ m \in \bb{Z} } \int_{\!-\infty}^\infty
\frac{ [u\Scr{N}_u^{n - 1} 1]\mh(m, \xi; z) }{ P_z(m, \xi) }\rme^{\rmi\omega m x + \rmi\xi y} \,\rmd\xi\Bigg) \rmd z\\
&= \frac{1}{2\pi}\sideset{}{'}\sum_{ m \in \bb{Z} } \int_{\!-\infty}^\infty
\bigg(\int_T \frac{ [u\Scr{N}_u^{n - 1} 1]\mh(m, \xi; z) }{ P_z(m, \xi) } \,\rmd z\bigg)\rme^{\rmi\omega m x + \rmi\xi y} \,\rmd\xi\\
&= 0,
\end{align*}
for every triangle $T$ in $\wpln$. A final application of Morera's theorem, shows that $(\Scr{N}_u^n 1)(x, y; \cdot) \in H(\wpln)$, thus,
the theorem is proved.
\end{proof}
\subsection[Sectionally Holomorphicity of the Eigenfunctions and formulation of a Riemann--Hilbert problem]{Sectionally Holomorphicity of the Eigenfunctions and formulation of a
Riemann--Hilbert problem}\label{s:jumps}
\separate

The function $\wh{u\mu}/P_z$ is discontinuous whenever the real part of $z$ belongs to $\frac{\omega}{2}\sZ$. Consider the vertical lines
$\real z = \frac{\omega}{2}n$, $n \in \sZ$. When $z$ lies on such a line, the integral corresponding to $m = -n$ in equation \eqref{Eq:eigenfunction} is 
singular. As it turns out, this singularity is integrable if we assume some smoothness for the potential $u$. If $u$ belongs in some suitable Sobolev space, the 
function $\mu$ will have a limit from the left and from the right of each one of these lines. Call these limits $\mu^-$ and $\mu^+$ respectively, and let
$\op{J}\mu$ be the jump of $\mu$ across such a line, i.e., $\op{J}\mu = \mu^+ - \mu^-$. Now, $\mu$ satisfies the analytic family of differential
equations $P( \partial + w(z) )\mu = -u\mu$ in the parameter $z \in \bb{C}$. Hence,
\begin{equation}
P( \partial + w(z) )\op{J}\mu = -u\op{J}\mu.
\end{equation}
From the existence and uniqueness theorem for this equation, we have $\op{J}\mu = \Scr{S}\mu$ for some linear operator $\Scr{S}$. Call the map
$u \mapsto \Scr{S}$ the forward spectral transform. Knowledge of this map amounts to knowing $\mu$ and consequently $u$. It remains to calculate these 
jumps. We begin by establishing an important lemma.
\begin{lemma}\label{L:rapiddeacy}
Suppose $\maxnorm{u} < 2\pi/C$ and that for some multi-index $\alpha = (\alpha_1, \alpha_2)$, $u$ is smooth to order $\abs{\alpha}$ and
$\abs{\partial^{\alpha'} u} \in L^1(\Omega) \cap L^2(\Omega)$, for all multi-indices $\alpha'$ such that $\abs{\alpha'} \leq \abs{\alpha}$. Then, the 
function $\mu$ is smooth to order $\abs{\alpha}$ and
\begin{equation}\label{Eq:bigo}
\abs{ \wh{u\mu}(m, \xi; z) } = O\bigg( \frac{1}{ 1 + \abs{\omega m}^{\alpha_1} + \abs{\xi}^{\alpha_2} } \bigg),
\end{equation}
for $(m, \xi; z) \in \Scr{C} \times \wpln$.
\end{lemma}
\begin{proof}
Let $z \in \wpln$. To establish the smoothness of $\mu$ we use induction. Suppose $\partial^{\alpha'} \mu \in L^\infty(E)$ for all $\alpha' < \alpha$. Then,
\[
\partial^\alpha \mu = \partial^\alpha(\op{Id} - \Scr{N}_ u)^{-1} 1 = (\op{Id} - \Scr{N}_u)^{-1} [\partial^\alpha, \Scr{N}_u]\mu,
\]
since
\begin{align*}
(\op{Id} - \Scr{N}_u)\partial^\alpha \mu &= \partial^\alpha \mu - \Scr{N}_u \partial^\alpha \mu\\
&= \partial^\alpha(1 +  \Scr{N}_u \mu) - \Scr{N}_u \partial^\alpha \mu\\
&= \partial^\alpha \Scr{N}_u \mu - \Scr{N}_u \partial^\alpha \mu.
\end{align*}
By Leibniz' rule, 
\begin{align*}
[\partial^\alpha, \Scr{N}_u]\mu &= \partial^\alpha \Scr{N}_u \mu - \Scr{N}_u \partial^\alpha \mu\\
&= \sum_{\alpha' \leq \alpha} \binom{\alpha}{\alpha'}(\partial^{\alpha - \alpha'} \Scr{N}_u)(\partial^{\alpha'} \mu) - \Scr{N}_u \partial^\alpha \mu\\
&= \sum_{\alpha' < \alpha} \binom{\alpha}{\alpha'}(\partial^{\alpha - \alpha'} \Scr{N}_u)(\partial^{\alpha'} \mu).
\end{align*}
 
Each operator $\partial^{\alpha - \alpha'} \Scr{N}_u$ is bounded from $L^\infty(\Omega)$ to $L^\infty(\Omega)$ for every $z \in \wpln$:
if $h \in L^\infty(\Omega)$, then
\begin{align*}
\partial^{\alpha - \alpha'} \bigg( \frac{ \wh{u h}(m, \xi) }{ P_z(m, \xi) }\rme^{\rmi\omega m x + \rmi\xi y} \bigg) &=
\frac{ \wh{u h}(m, \xi) }{ P_z(m, \xi) }(\rmi\omega m, \rmi\xi)^{\alpha - \alpha'} \rme^{\rmi\omega m x + \rmi\xi y}\\
&= \frac{ [\partial^{\alpha - \alpha'} u h]\mh(m, \xi) }{ P_z(m, \xi) }\rme^{\rmi\omega m x + \rmi\xi y}.
\end{align*}
But because of the smoothness of $u$, $[\partial^{\alpha - \alpha'} u h]\mh \in L^2( \Scr{C} ) \cap L^\infty( \Scr{C} )$. The rest follows by an application 
of dominated convergence and the \hyperlink{L:basiclemma}{Basic Lemma}.

Meanwhile, each term $\partial^{\alpha'} \mu$ belongs to $L^\infty(E)$ by the induction hypothesis, and the operator $(\op{Id} - \Scr{N}_u)^{-1}$ is 
bounded on $L^\infty(\Omega)$ for every $z \in \wpln$ because $\maxnorm{u} < 2\pi/C$. Thus, $\partial^\alpha \mu$ is bounded.

Now,
\begin{align*}
( \abs{\omega m}^{\alpha_1} + \abs{\xi}^{\alpha_2} )\abs{ \wh{u\mu}(m, \xi; z) } &= \abs{ (\rmi\omega m)^{\alpha_1}\wh{u\mu}(m, \xi; z) }
+ \abs{ (\rmi\xi)^{\alpha_2}\wh{u\mu}(m, \xi; z) }\\
&= \abs{ [\partial_x^{\alpha_1} u\mu]\mh(m, \xi; z) } + \abs{ [\partial_y^{\alpha_2} u\mu]\mh(m, \xi; z) }\\
&\leq \norm{\partial_x^{\alpha_1} u\mu}_1 + \norm{\partial_y^{\alpha_2} u\mu}_1.
\end{align*}
Since we also have that $\abs{ \wh{u\mu}(m, \xi; z) } \leq \norm{u\mu}_1$, equation \eqref{Eq:bigo} follows, hence the lemma is proved.
\end{proof}
\begin{theorem}\label{Th:zasympotics}
Suppose $\max\{\omega\norm{u}_1, \sqrt{\omega}\norm{u}_2\}$ is small, so that $\mu(x, y; z)$ is the unique solution to
\[
P( \partial + w(z) )\mu + u\mu = 0, \quad \mu \in L^\infty(E), \ \lim_{\abs{y} \to \infty} \mu(x, y; z) = 1.
\]
If in addition $y u(x, y) \in L^1(\Omega)$ and
\begin{equation}\label{Eq:weakderivatives}
u_x, u_y \in L^1(\Omega) \cap L^2(\Omega),
\end{equation}
then $\mu$ has pointwise one-sided limits at the lines $\real z= \frac{\omega}{2}n$, $n \in \sZ$, and
\begin{equation}
\lim_{ \substack{\abs{ \im{z} } \to \infty \\ \re{z} = const.} } \mu(x, y; z) = 1.
\end{equation}
\end{theorem}
\begin{proof}
Letting
\begin{equation}\label{Eq:Fouriercoefficients}
\mu_m(y; z) = \frac{1}{2\pi}\int_{\!-\infty}^\infty \frac{ \wh{u\mu}(m, \xi; z) }{ P_z(m, \xi) }\rme^{\rmi\xi y} \,\rmd\xi,
\end{equation}
we can convert $\mu(x, y; z)$ in the form of a Fourier series
\begin{equation}\label{Eq:Fourierseries}
\mu(x, y; z) = 1 + \sideset{}{'}\sum_{m = -\infty}^\infty \mu_m(y; z)\rme^{\rmi\omega m x}.
\end{equation}
Since $u_x$, $u_y \in L^1(\Omega) \cap L^2(\Omega)$, lemma \ref{L:rapiddeacy} shows that
\begin{equation}\label{Eq:firstorderbound}
\abs{ \wh{u\mu}(m, \xi; z) } = O\bigg( \frac{1}{ 1 + \abs{\omega m} + \abs{\xi} } \bigg).
\end{equation}
Hence,
\[
\abs{ \mu_m(y; z) } \leq \frac{1}{2\pi}\int_{\!-\infty}^\infty \frac{ \abs{ \wh{u\mu}(m, \xi; z) } }{ \abs{ P_z(m, \xi) } } \,\rmd\xi
\leq \frac{c}{2\pi}\int_{\!-\infty}^\infty \frac{1}{ 1 + \abs{\omega m} + \abs{\xi} }\frac{1}{ \abs{ P_z(m, \xi) } } \,\rmd\xi,
\]
for some positive, real constant $c$. Thus,
{\allowdisplaybreaks
\begin{align*}
\abs{ \mu_m(y; z) } &< \frac{c}{2\pi}\int_{\!-\infty}^\infty \frac{1}{ ( 1 + \abs{\xi} )\abs{ P_z(m, \xi) } } \,\rmd\xi\\
&\leq \frac{c}{2\pi}\bigg(\int_{\!-\infty}^\infty \frac{1}{ ( 1 + \abs{\xi} )^\frac{3}{2} } \,\rmd\xi\bigg)^\frac{2}{3}
\bigg(\int_{\!-\infty}^\infty \frac{1}{\abs{ P_z(m, \xi) }^3} \,\rmd\xi\bigg)^\frac{1}{3}\\
&= \frac{c}{2\pi}\bigg(2\int_0^\infty \frac{1}{ (1 + \xi)^\frac{3}{2} } \,\rmd\xi\bigg)^\frac{2}{3} \bigg(\int_{\!-\infty}^\infty
\frac{1}{\abs{ P_z(m, \xi) }^3} \,\rmd\xi\bigg)^\frac{1}{3}\\
&= \frac{c}{2\pi}4^\frac{2}{3}
\bigg(\int_{\!-\infty}^\infty
\frac{1}{ [ ( ( \omega m + \re{z} )^2 - \re{z}^2 )^2 + ( \xi + 2\omega m \im{z} )^2 ]^\frac{3}{2} } \,\rmd\xi\bigg)^\frac{1}{3}\\
&= \frac{c}{2\pi}4^\frac{2}{3} \bigg(\int_{\!-\infty}^\infty
\frac{1}{ [ ( ( \omega m + \re{z} )^2 - \re{z}^2 )^2 + \xi^2 ]^\frac{3}{2} } \,\rmd\xi\bigg)^\frac{1}{3}\\
&= \frac{c}{2\pi}4^\frac{2}{3} \abs{ ( \omega m + \re{z} )^2 - \re{z}^2 }^{ -\frac{2}{3} } \bigg(\int_{\!-\infty}^\infty
\frac{1}{ (1 + v^2)^\frac{3}{2} } \,\rmd v\bigg)^\frac{1}{3}\\
&= \frac{c}{2\pi}4^\frac{2}{3} 2^\frac{1}{3} \frac{1}{ \abs{ ( \omega m + \re{z} )^2 - \re{z}^2 }^\frac{2}{3} }.
\end{align*}}
Now
\[
\sideset{}{'}\sum_{m = -\infty}^\infty \frac{1}{ \abs{ ( \omega m + \re{z} )^2 - \re{z}^2 }^\frac{2}{3} } =
\sum_{m \in \mZ} \frac{1}{ \abs{m^2 - \re{z}^2}^\frac{2}{3} } = 2\sum_{m \in \mZ^+} \frac{1}{ \abs{m^2 - \re{z}^2}^\frac{2}{3} }.
\]
But inequality \eqref{Eq:basicinequality} yields
\[
\frac{1}{ \abs{m^2 - \re{z}^2}^\frac{2}{3} } < \frac{1}{ ( m - \abs{ \re{z} } )^\frac{4}{3} }.
\]
Using the same arguments as we did in the corresponding part of the proof of the \hyperlink{L:basiclemma}{Basic Lemma},
\[
\sum_{m \in \mZ^+} \frac{1}{ ( m - \abs{ \re{z} } )^\frac{4}{3} } < \frac{2}{ \omega^\frac{4}{3} }\sum_{m = 1}^\infty \frac{1}{ m^\frac{4}{3} },
\]
for every $z \in \wpln$. Therefore, the series
\[
\sideset{}{'}\sum_{m = -\infty}^\infty \mu_m(y; z)\rme^{\rmi\omega m x}
\]
converges uniformly in $\wpln$.

Let $m$ be a nonzero integer. Denote by $\mu^+$ and $\mu^-$ the non-tangential limits of $\mu$ from the right and from the left of the line
\begin{equation}
L_m \= \{ \zeta \in \bb{C} \colon \re{\zeta} = -\tfrac{\omega}{2}m, \ \im{\zeta} \in \bb{R} \}
\end{equation}
respectively. By the uniform convergence of the series in \eqref{Eq:Fourierseries}, in order to establish the existence of the limits $\mu^\pm$, it is enough
to show that these limits exists for the function $\mu_m(y; z)$ for all $y \in \bb{R}$. Write $\mu_m(y; z)$ in the form
\[
\mu_m(y; z) = \frac{1}{2\pi \rmi}\int_{\!-\infty}^\infty \frac{ \wh{u\mu}(m, \xi; z)\rme^{\rmi\xi y} }{ \xi - \rmi\omega m(\omega m + 2z) } \,\rmd\xi =
\frac{1}{2\pi \rmi}\int_{\!-\infty}^\infty \frac{ \wh{u\mu}(m, \xi; z)\rme^{\rmi\xi y} }{ \xi - p_0(z) } \,\rmd\xi,
\]
where $p_0(z) =  i\omega m(\omega m + 2z)$. This is a Cauchy type integral. Hence, to show the existence of the limit of $\mu_m(y; z)$ as $z$
approaches the line $L_m$ from the sides along any non-tangential path, or equivalently as $p_0$ approaches the real axes from the upper and from the 
lower half-planes, it suffices to show that for every $y \in \bb{R}$, the function $\wh{u\mu}(m, \xi; z)\rme^{\rmi\xi y}$ is H\"{o}lder continuous for all finite
$\xi$, tends to a definite limit $\wh{u\mu}(m, \infty; z)\rme^{\rmi\infty y}$ as $\abs{\xi} \to \infty$, and that for large $\xi$, the inequality
\begin{equation}\label{Eq:Holderinfinity}
\wh{u\mu}(m, \xi; z)\rme^{\rmi\xi y} - \wh{u\mu}(m, \infty; z)\rme^{\rmi\infty y} \leq \frac{M}{\abs{\xi}^\kappa},
\end{equation}
holds for some positive, real constants $M$ and $\kappa$.
For $\xi_1$, $\xi_2 \in \bb{R}$ we have the following:
\begin{align*}
&\abs{ \wh{u\mu}(m, \xi_1; z)\rme^{\rmi\xi_1 y} - \wh{u\mu}(m, \xi_2; z)\rme^{\rmi\xi_2 y} }\\
&\leq \frac{1}{2\ell}\int_{\!-\infty}^\infty \int_{\!-\ell}^\ell
\abs{ u(x, y')\mu(x, y'; z) }\abs{ \rme^{ \rmi\xi_1(y - y') } - \rme^{ \rmi\xi_2(y - y') } } \,\rmd x\rmd y'\\
&\leq \frac{1}{2\ell}\int_{\!-\infty}^\infty \int_{\!-\ell}^\ell \abs{ u(x, y') } \,\norm{\mu}_\infty \abs{y - y'}\abs{\xi_1 - \xi_2} \,\rmd x\rmd y'\\
&\leq \frac{1}{2\ell}\norm{\mu}_\infty \bigg(\abs{y}\norm{u}_1 +  \int_{-\infty}^\infty \int_{\!-\ell}^\ell
\abs{ u(x, y') }\abs{y'} \,\rmd x\rmd y'\bigg)\abs{\xi_1 - \xi_2},
\end{align*}
hence $\wh{u\mu}(m, \xi; z)\rme^{\rmi\xi y}$ is indeed H\"{o}lder continuous for all finite $\xi$. Furthermore, from \eqref{Eq:firstorderbound} there exists 
a real number $c > 0$ such that $\abs{ \wh{u\mu}(m, \xi; z) } \leq c/\abs{\xi}$. Hence, the limit $\wh{u\mu}(m, \infty; z)\rme^{\rmi\infty y}$ is definite
(in fact $\wh{u\mu}(m, \infty; z)\rme^{\rmi\infty y} = 0$) and the inequality \eqref{Eq:Holderinfinity} is satisfied (with $M = c$ and $\kappa = 1$).

Finally, fix $\lambda$ and suppose $z = \lambda + \rmi\im{z}$ is a complex number with $\lambda \notin \tfrac{\omega}{2}\sZ$. By symmetry, we can 
assume that $\lambda > 0$. Using once again the uniform convergence of the Fourier series \eqref{Eq:Fourierseries} on $\wpln$,
\[
\lim_{\abs{ \im{z} } \to \infty} (\mu( x, y; \lambda + \rmi\im{z} ) - 1) = \sideset{}{'}\sum_{m = -\infty}^\infty
\lim_{\abs{ \im{z} } \to \infty} \mu_m( y; \lambda + \rmi\im{z} )\rme^{\rmi\omega m x}.
\]
Split the sum as follows:
\begin{align*}
\sideset{}{'}\sum_{m = -\infty}^\infty \lim_{\abs{ \im{z} } \to \infty} \mu_m( y; \lambda + \rmi\im{z} )\rme^{\rmi\omega m x} &=
\sum_{ \substack{ m > 0 \\ \substack{\text{or} \\ \omega m + 2\lambda < 0} } } \lim_{\abs{ \im{z} } \to \infty}
\mu_m( y; \lambda + \rmi\im{z} )\rme^{\rmi\omega m x}\\
&{}+ \sum_{-2\lambda < \omega m < 0} \lim_{\abs{ \im{z} } \to \infty} \mu_m( y; \lambda + \rmi\im{z} )\rme^{\rmi\omega m x}.
\end{align*}
For $m > 0$ or $\omega m + 2\lambda < 0$, we have $(\omega m)^2 + 2\omega m\lambda > 0$. Hence,
\begin{align*}
\mu_m(y; z) &= \frac{1}{2\pi}\int_{\!-\infty}^\infty
\frac{ \wh{u\mu}(m, \xi; z) }{ (\omega m)^2 + 2\omega m\lambda + \rmi( \xi + 2\omega m\im{z} ) }\rme^{\rmi\xi y} \,\rmd\xi\\
&= \frac{1}{2\pi}\int_{\!-\infty}^\infty \wh{u\mu}(m, \xi; z)\rme^{\rmi\xi y}
\bigg(\int_{\!-\infty}^0 \rme^{ ( (\omega m)^2 + 2\omega m\lambda + \rmi( \xi + 2\omega m\im{z} ) )\tau } \,\rmd\tau\bigg)\rmd\xi\\
&= \frac{1}{2\pi}\int_{\!-\infty}^0 \rme^{ ( (\omega m)^2 + 2\omega m\lambda )\tau }
\bigg(\int_{\!-\infty}^\infty \wh{u\mu}(m, \xi; z)\rme^{ \rmi\xi(y + \tau) } \,\rmd\xi\bigg)\rme^{\rmi2\omega m\im{z}\tau} \,\rmd\tau\\
&= \int_{\!-\infty}^0 f_{m, y, \lambda}(\tau)\rme^{\rmi2\omega m\im{z}\tau} \,\rmd\tau,
\end{align*}
where
\begin{align*}
f_{m, y, \lambda}(\tau) &= \rme^{ ( (\omega m)^2 + 2\omega m\lambda )\tau }
\frac{1}{2\pi}\int_{\!-\infty}^\infty \wh{u\mu}(m, \xi; z)\rme^{ \rmi\xi(y + \tau) } \,\rmd\xi\\
&= \rme^{ ( (\omega m)^2 + 2\omega m\lambda )\tau }\frac{1}{2\ell}\int_{\!-\ell}^\ell u(x, y + \tau)\mu(x, y + \tau; z)\rme^{-\rmi\omega m x} \,\rmd x.
\end{align*}
But since $\rme^{ ( (\omega m)^2 + 2\omega m\lambda )\tau } < 1$ for $\tau < 0$,
\begin{align*}
\int_{\!-\infty}^0 \abs{ f_{m, y, \lambda}(\tau) } \,\rmd\tau &< \frac{1}{2\ell}\int_{\!-\infty}^0 \int_{\!-\ell}^\ell
\abs{ u(x, y + \tau) }\abs{ \mu(x, y + \tau; z) } \,\rmd x\rmd\tau\\
&\leq \frac{1}{2\ell}\norm{\mu}_\infty \int_{\!-\infty}^0 \int_{\!-\ell}^\ell \abs{ u(x, y + \tau) } \,\rmd x\rmd\tau\\
&\leq \frac{1}{2\ell}\norm{\mu}_\infty \norm{u}_1.
\end{align*}
Thus,
\[
\mu_m(y; z) = \int_{\!-\infty}^0 f_{m, y, \lambda}(\tau)\rme^{\rmi2\omega m\im{z}\tau} \,\rmd\tau,
\]
with $f_{m, y, \lambda}(\tau) \in L^1(-\infty, 0)$. Hence, from the Riemann--Lebesgue lemma,
\[
\lim_{\abs{ \im{z} } \to \infty} \mu_m(y; z) = 0.
\]

Now, for $-2\lambda < \omega m < 0$, $(\omega m)^2 + 2\omega m\lambda$ is negative, and so
\[
\mu_m(y; z) = \int_0^\infty f_{m, y, \lambda}(\tau)\rme^{\rmi2\omega m\im{z}\tau} \,\rmd\tau,
\]
where this time
\[
f_{m, y, \lambda}(\tau) = -\rme^{ ( (\omega m)^2 + 2\omega m\lambda )\tau } \frac{1}{2\ell}\int_{\!-\ell}^\ell
u(x, y + \tau)\mu(x, y + \tau; z)\rme^{-\rmi\omega m x} \,\rmd x,
\]
and $f_{m, y, \lambda}(\tau) \in L^1(0, \infty)$. Hence, $\lim\limits_{\abs{ \im{z} } \to \infty} \mu_m(y; z) = 0$ in this case as well.
Therefore, $\lim\limits_{\abs{ \im{z} } \to \infty} \mu_m( y; \lambda + i\im{z} ) = 0$ for every $m \in \sZ$ and consequently
\[
\lim_{\abs{ \im{z} } \to \infty} (\mu( x, y; \lambda + i\im{z} ) - 1) = 0,
\]
thus, concluding the proof of the theorem.
\end{proof}
Theorems \ref{Th:holomorphicity} and \ref{Th:zasympotics} show that $\mu(x, y; z)$ is a sectionally holomorphic function (with respect to $z$) in the strips
\begin{gather}
S_n \= \{ z \in \bb{C} \colon \tfrac{\omega}{2}n < \re{z} < \tfrac{\omega}{2}(n + 1), \ \im{z} \in \bb{R} \}, \label{Eq:stripsplus}\\
S_{-n} \= \{ z \in \bb{C} \colon -\tfrac{\omega}{2}(n + 1) < \re{z} < -\tfrac{\omega}{2}n, \ \im{z} \in \bb{R} \}, \label{Eq:stripsminus}\\
\intertext{for $n \in \bb{N}$, and}
S_0 \= \{ z \in \bb{C} \colon \abs{ \re{z} } < \tfrac{\omega}{2}, \ \im{z} \in \bb{R} \}. \label{Eq:zerostrip}
\end{gather}
\begin{figure}[H]
\centering
\begin{tikzpicture}
\draw[->](-5,0) -- (5,0) node[below]{$\real z$};
\draw[->](0,-2) -- (0,2) node[right]{$\imaginary z$};
\node[below] at (-0.2, 0){$0$};
\node[fill=white] at (0, 1){$S_0$};
\draw[dashed,red] (1,2)--(1,-2);
\node at (1.5, 1){$S_1$};
\node[below] at (0.8,0){$\frac{\omega}{2}$};
\draw[fill,red](1,0) circle[radius=1pt];
\draw[dashed,red] (2,2)--(2,-2);
\node at (2.5, 1){$S_2$};
\node[below] at (1.8,0) {$\omega$};
\draw[fill,red](2,0) circle[radius=1pt];
\draw[dashed,red] (3,2)--(3,-2);
\node[below] at (2.8,0) {$\frac{3\omega}{2}$};
\draw[fill,red](3,0) circle[radius=1pt];
\draw[dashed,red] (-1,2)--(-1,-2);
\node at (-1.5, 1){$S_{-1}$};
\node[below] at (-1.3,0) {$-\frac{\omega}{2}$};
\draw[fill,red](-1,0) circle[radius=1pt];
\draw[dashed,red] (-2,2)--(-2,-2);
\node at (-2.5, 1){$S_{-2}$};
\node[below] at (-2.3,0) {$-\omega$};
\draw[fill,red](-2,0) circle[radius=1pt];
\draw[dashed,red] (-3,2)--(-3,-2);
\node[below] at (-3.4,0) {$-\frac{3\omega}{2}$};
\draw[fill,red](-3,0) circle[radius=1pt];
\draw[dotted,very thick] (3.5,1)--(3.8,1);
\draw[dotted,very thick] (-3.8,1)--(-3.5,1);
\end{tikzpicture}
\caption{Strips of holomorphicity\label{fig:strips}}
\end{figure}
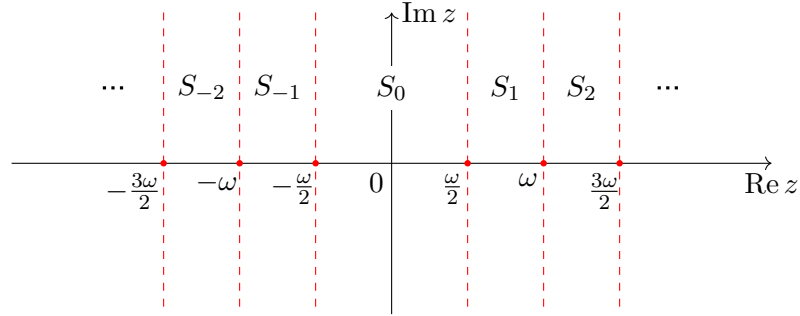
To calculate the jump of $\mu$ across the lines
\begin{equation}\label{Eq:jumpcontours}
L_n \= \{ z \in \bb{C} \colon \re{z} = -\tfrac{\omega}{2}n, \ \im{z} \in \bb{R} \}, \quad n \in \sZ,
\end{equation}
it is convenient to calculate the integrals in the equation \eqref{Eq:eigenfunction} for each integer $m \in \sZ$.
\begin{proposition}
Suppose $\maxnorm{u} < 2\pi/C$. Then, for $z \in \wpln$, the function $\mu(x, y; z)$ satisfies the integral equation
\begin{equation}\label{Eq:classicrepresentation}
\mu(x, y; z) = 
\begin{cases}
1 + [\op{g}_{ u, \mrm{r} } \mu](x, y; z), & \re{z} > 0\\[2pt]
1 + [\op{g}_{ u, \mrm{l} } \mu](x, y; z), & \re{z} < 0,
\end{cases}
\end{equation}
where
\begin{equation}\label{Eq:groperator}
\begin{split}
[\op{g}_{ u, \mrm{r} } h](x, y; z) &\= \frac{1}{2\ell}\Bigg(\sum_{ \substack{ m > 0 \\ \substack{ \tn{or} \\ \omega m < -2\re{z} } } }
\int_{\!-\infty}^y \int_{\!-\ell}^\ell - \sum_{-2\re{z} < \omega m < 0} \int_y^\infty \int_{\!-\ell}^\ell\Bigg)\\
&\quad\ph{1 + \frac{1}{2\ell}} u(x', y')h(x', y')\rme^{ \rmi\omega m(x - x') - \omega m(\omega m + 2z)(y - y') } \,\rmd x'\rmd y',
\end{split}
\end{equation}
and
\begin{equation}\label{Eq:gloperator}
\begin{split}
[\op{g}_{ u, \mrm{l} } h](x, y; z) &\= \frac{1}{2\ell}\Bigg(\sum_{ \substack{ m < 0 \\ \substack{ \tn{or} \\ \omega m > -2\re{z} } } }
\int_{\!-\infty}^y \int_{\!-\ell}^\ell - \sum_{ 0 < \omega m < -2\re{z} } \int_y^\infty \int_{\!-\ell}^\ell\Bigg)\\
&\quad\ph{1 + \frac{1}{2\ell}} u(x', y')h(x', y')\rme^{ \rmi\omega m(x - x') - \omega m(\omega m + 2z)(y - y') } \,\rmd x'\rmd y',
\end{split}
\end{equation}
for every function $h \in L^\infty(\Omega)$.
\end{proposition}
\begin{proof}
Fix a number $z$ in $\wpln$ and suppose $\re{z} > 0$. We have
\begin{align*}
\mu(x, y; z) - 1 &= \frac{1}{2\pi}\sideset{}{'}\sum_{m = -\infty}^\infty \int_{\!-\infty}^\infty \frac{1}{ P_z(m, \xi) }\\
\begin{split}
&\qquad \bigg(\frac{1}{2\ell}\int_{\!-\infty}^\infty \int_{\!-\ell}^\ell
u(x', y')\mu(x', y'; z)\rme^{ \rmi\omega m(x - x') + \rmi\xi(y - y') } \,\rmd x'\rmd y'\bigg)\rmd\xi
\end{split}\\ 
&= \frac{1}{2\ell}\sideset{}{'}\sum_{m = -\infty}^\infty \frac{1}{2\pi i}\int_{\!-\infty}^\infty \frac{1}{ \xi - \rmi\omega m(\omega m + 2z) }\\
\begin{split}
&\qquad \bigg(\int_{\!-\infty}^y \int_{\!-\ell}^\ell u(x', y')\mu(x', y'; z)\rme^{ \rmi\omega m(x - x') + \rmi\xi(y - y') } \,\rmd x'\rmd y'\bigg)\rmd\xi
\end{split}\\
&\qquad + \frac{1}{2\ell}\sideset{}{'}\sum_{m = -\infty}^\infty \frac{1}{2\pi i}\int_{\!-\infty}^\infty \frac{1}{\xi - \rmi\omega m(\omega m + 2z) }\\
\begin{split}
&\qquad \bigg(\int_y^\infty \int_{\!-\ell}^\ell u(x', y')\mu(x', y'; z)\rme^{ \rmi\omega m(x - x') + \rmi\xi(y - y') } \,\rmd x'\rmd y'\bigg)\rmd\xi.
\end{split}
\end{align*}

Let $p_0 = \rmi\omega m(\omega m + 2z)$ and $s = \xi + \rmi\tau$ be a complex number in the upper-half plane. The integral
\[
\int_{\!-\infty}^y \int_{\!-\ell}^\ell u(x', y')\mu(x', y'; z)\rme^{ \rmi\omega m(x - x') + \rmi s(y - y') } \,\rmd x'\rmd y'
\]
converges absolutely since
\begin{align*}
\int_{\!-\infty}^y \int_{\!-\ell}^\ell \abs{ u(x', y') }\abs{ \mu(x', y'; z) }\rme^{ -\tau(y - y') } \,\rmd x'\rmd y' &<
\norm{\mu}_\infty \int_{\!-\infty}^y \int_{\!-\ell}^\ell \abs{ u(x', y') } \,\rmd x'\rmd y'\\
&\leq \norm{\mu}_\infty \norm{u}_1.
\end{align*}
Thus, it defines a holomorphic function with respect to $s$; apply Fubini's and Morera's theorems. Hence, the function
\begin{align*}
f(s; m, x, y) &= \frac{1}{s - p_0}g(s; m, x, y)\\ 
&\equiv \frac{1}{s - p_0}\int_{\!-\infty}^y \int_{\!-\ell}^\ell u(x', y')\mu(x', y'; z)\rme^{ \rmi\omega m(x - x') + \rmi s(y - y') } \,\rmd x'\rmd y'
\end{align*}
is holomorphic, if $\imaginary p_0 < 0$, and meromorphic with a simple pole at the point $p_0$, if $\imaginary p_0 > 0$. Therefore, by the residue theorem
\[
\frac{1}{2\pi i}\int_\gamma f(s; m, x, y) \,ds = 
\begin{cases}
0, & \imaginary p_0 < 0\\
\res(f, s = p_0), & \imaginary p_0 > 0,
\end{cases}
\]
where the curve $\gamma = [-R, R] + C_R$, and $C_R$ is the semi-circle in the upper-half plane, centred at the origin with radius $R$, such that
$R > \abs{p_0}$. Now,
\begin{align*}
\Abs{\int_{C_R} f(s; m, x, y) \,\rmd s} &\leq \frac{\pi R}{ R - \abs{p_0} }\max_{s \in C_R} \abs{ g(s; m, x, y) }\\
&= \frac{\pi R}{ R - \abs{p_0} }\max_{ \theta \in [0, \pi] } \abs{ g(R\rme^{\rmi\theta}; m, x, y) }.
\end{align*}
But
\[
\abs{ g(R\rme^{\rmi\theta}; m, x, y) } \leq \norm{\mu}_\infty \int_{\!-\infty}^y \int_{\!-\ell}^\ell
\abs{ u(x', y') }\rme^{ -R\sin\theta(y - y') } \,\rmd x'\rmd y',
\]
and since $\sin\theta > 0$, absolute and dominated convergence implies
\[
\lim_{R \to \infty} \int_{\!-\infty}^y \int_{\!-\ell}^\ell \abs{ u(x', y') }\rme^{ -R\sin\theta(y - y') } \,\rmd x'\rmd y' = 0,
\]
hence $\max_{ \theta \in [0, \pi] } \abs{ g(R\rme^{\rmi\theta}; m, x, y) } \to 0$ as $R$ tends to $\infty$. Thus, writing
\[
\frac{1}{2\pi \rmi}\int_\gamma f(s; m, x, y) \,\rmd s = \frac{1}{2\pi \rmi}\int_{\!-R}^R f(\xi; m, x, y) \,\rmd\xi +
\frac{1}{2\pi \rmi}\int_{C_R} f(s; m, x, y) \,\rmd s
\]
and taking the limit as $R \to \infty$, yields
\[
\frac{1}{2\pi \rmi}\int_{\!-\infty}^\infty f(\xi; m, x, y) \,\rmd\xi = 
\begin{cases}
0, & \imaginary p_0 < 0 \\
g(p_0; m, x, y), & \imaginary p_0 > 0.
\end{cases}
\]
Similarly, choosing $\tilde{\gamma} = \wt{C}_R + [-R, R]$ with positive orientation, where $\wt{C}_R$ is the semi-circle, situated in the lower-half plane,
centred at the origin with radius $R$ such that $R > \abs{p_0}$, we find
\[
\frac{1}{2\pi i}\int_{\!-\infty}^\infty \tilde{f}(\xi; m, x, y) \,\rmd\xi = 
\begin{cases}
0, & \imaginary p_0 > 0 \\
-\tilde{g}(p_0; m, x, y), & \imaginary p_0 < 0,
\end{cases}
\]
where
\begin{align*}
\tilde{f}(s; m, x, y) &= \frac{1}{s - p_0}\tilde{g}(s; m, x, y) \\ 
&\equiv \frac{1}{s - p_0}\int_y^\infty \int_{\!-\ell}^\ell u(x', y')\mu(x', y'; z)\rme^{ \rmi\omega m(x - x') + \rmi s(y - y') } \,\rmd x'\rmd y'.
\end{align*}
Now, $\imaginary p_0 = \imaginary( \rmi\omega m(\omega m + 2z) ) = \real( \omega m(\omega m + 2z) )$. Thus,
\[
\imaginary p_0 > 0 \Equiv m > 0 \text{ or } m < -\tfrac{ 2\re{z} }{\omega},
\]
and
\[
\imaginary p_0 < 0 \Equiv -\tfrac{ 2\re{z} }{\omega} < m < 0.
\]
Putting all this together,
\begin{multline*}
\mu(x, y; z) = 1 + \frac{1}{2\ell}\Bigg(\sum_{ \substack{ m > 0 \\ \substack{\text{or} \\ m < -2\re{z}/\omega} } }
\int_{\!-\infty}^y \int_{\!-\ell}^\ell - \sum_{-2\re{z}/\omega < m < 0} \int_y^\infty \int_{\!-\ell}^\ell\Bigg)\\
u(x', y')\mu(x', y'; z)\rme^{ \rmi\omega m(x - x') - \omega m(\omega m + 2z)(y - y') } \,\rmd x'\rmd y'.
\end{multline*}

In the case where $\re{z} < 0$, we have
\[
\imaginary p_0 > 0 \Equiv m < 0 \text{ or } m > -\tfrac{ 2\re{z} }{\omega},
\]
and
\[
\imaginary p_0 < 0 \Equiv 0 < m < -\tfrac{ 2\re{z} }{\omega}.
\]
Thus, repeating the previous arguments,
\begin{equation*}
\begin{split}
\mu(x, y; z) &= 1 + \frac{1}{2\ell}\Bigg(\sum_{ \substack{ m < 0 \\ \substack{\text{or} \\ m > -2\re{z}/\omega} } }
\int_{\!-\infty}^y \int_{\!-\ell}^\ell - \sum_{0 < m < -2\re{z}/\omega} \int_y^\infty \int_{\!-\ell}^\ell\Bigg)\\
&\qquad\ph{1 + \frac{1}{2\ell}(} u(x', y')\mu(x', y'; z)\rme^{ \rmi\omega m(x - x') - \omega m(\omega m + 2z)(y - y') } \,\rmd x'\rmd y'.\qedhere
\end{split}
\end{equation*}
\end{proof}

Since $\mu = 1 + \Scr{N}_u\mu$, equation \eqref{Eq:classicrepresentation} implies that, for all $z \in \wpln$ and $h \in L^\infty(\Omega)$,
\[
(\Scr{N}_u h)(x, y; z) = 
\begin{cases}
[\op{g}_{ u, \mrm{r} } h](x, y; z), & \re{z} > 0\\[2pt]
[\op{g}_{ u, \mrm{l} } h](x, y; z), & \re{z} < 0.
\end{cases}
\]
Furthermore, condition \eqref{Eq:zeromass} aloows us to rewrite equations \eqref{Eq:groperator} and \eqref{Eq:gloperator} as
\begin{equation}
\begin{split}
[\op{g}_{ u, \mrm{r} } h](x, y; z) &\= \frac{1}{2\ell}\Bigg(\sum_{ \substack{ m > 0 \\ \substack{ \tn{or} \\ \omega m < -2\re{z} } } }
\int_{\!-\infty}^y \int_{\!-\ell}^\ell - \sum_{-2\re{z} < \omega m \leq 0} \int_y^\infty \int_{\!-\ell}^\ell\Bigg)\\
&\quad\ph{1 + \frac{1}{2\ell}} u(x', y')h(x', y')\rme^{ \rmi\omega m(x - x') - \omega m(\omega m + 2z)(y - y') } \,\rmd x'\rmd y',
\end{split}
\end{equation}
and
\begin{equation}
\begin{split}
[\op{g}_{ u, \mrm{l} } h](x, y; z) &\= \frac{1}{2\ell}\Bigg( \sum_{ \substack{ m \leq 0 \\ \substack{ \tn{or} \\ \omega m > -2\re{z} } } }
\int_{\!-\infty}^y \int_{\!-\ell}^\ell - \sum_{ 0 < \omega m < -2\re{z} } \int_y^\infty \int_{\!-\ell}^\ell \Bigg)\\
&\quad\ph{1 + \frac{1}{2\ell}}u(x', y')h(x', y')\rme^{ \rmi\omega m(x - x') - \omega m(\omega m + 2z)(y - y') } \,\rmd x'\rmd y'.
\end{split}
\end{equation}
We are now ready to calculate $\op{J}\mu$ across the contours $L_n$ and derive the spectral data.
\begin{theorem}\label{Th:departurefromholomorphicity}
Suppose the potential $u(x, y)$ is small and regular, such that $\maxnorm{u} < 2\pi/C, \,y u(x, y) \in L^1(\Omega)$ and $u_x$, $u_y \in L^1(\Omega) \cap L^2(\Omega)$. Then,
\begin{equation}\label{Eq:RiemannHilbert}
\op{J}\mu(x, y; z) \equiv \mu^+(x, y; z) - \mu^-(x, y; z) = F(z)\rme^{-\rmi( z + \bar{z} )x + (z^2 - \bar{z}^2)y} \mu^-( x, y; -\bar{z} ),
\end{equation}
for $z$ on the contour $L_n$, $n \in \sZ$, where
\begin{equation}\label{Eq:scatteringdata}
F(z) \= -\frac{ \sgn( \re{z} ) }{2\ell}\int_{\!-\infty}^\infty \int_{\!-\ell}^\ell
u(x, y)\mu^+(x, y; z)\rme^{ \rmi( z + \bar{z} )x - (z^2 - \bar{z}^2)y } \,\rmd x\rmd y,
\end{equation}
defines the \emph{spectral data}.
\end{theorem}
\begin{proof}
Fix a positive integer $n$ and let $z = \tfrac{\omega}{2}n + \rmi\im{z}$ (for $n$ negative integer the analysis is similar and thus, omitted). Then,
{\allowdisplaybreaks
\begin{align}\label{Eq:rightlimit}
\mu^+(x, y; z) &= 1 +  [\op{g}_{ u, \mrm{r} } \mu]^+(x, y; z)\nonumber\\ 
&= 1 + \frac{1}{2\ell}\Bigg(\sum_{ \substack{ m > 0 \\ \substack{\tn{or} \\ \omega m < -\omega n} } } \int_{\!-\infty}^y \int_{\!-\ell}^\ell -
\sum_{-\omega n \leq \omega m \leq 0} \int_y^\infty \int_{\!-\ell}^\ell\Bigg)\nonumber\\
&\qquad\ph{1 + \frac{1}{2\ell}} u(x', y')\mu^+(x', y'; z)\rme^{ \rmi\omega m(x - x') - ( (\omega m + z)^2 - z^2 )(y - y') } \,\rmd x'\rmd y'\nonumber\\
&= 1 + \frac{1}{2\ell}\Bigg(\sum_{ \substack{ k > z \\ \substack{ \tn{or} \\ k < -\bar{z} } } } \int_{\!-\infty}^y \int_{\!-\ell}^\ell -
\sum_{-\bar{z} \leq k \leq z} \int_y^\infty \int_{\!-\ell}^\ell\Bigg)\nonumber\\
&\qquad\ph{1 + \frac{1}{2\ell}} u(x', y')\mu^+(x', y'; z)\rme^{ \rmi(k - z)(x - x') - (k^{2} - z^2)(y - y') } \,\rmd x'\rmd y'\nonumber\\
&= 1 + \frac{1}{2\ell}\Bigg(\sum_{k \in U_{ z, \mrm{r} }^+}
\int_{\!-\infty}^y \int_{\!-\ell}^\ell - \sum_{k \in V_{ z, \mrm{r} }^+} \int_y^\infty \int_{\!-\ell}^\ell\Bigg)\nonumber\\
&\qquad\ph{1 + \frac{1}{2\ell}} u(x', y')\mu^+(x', y'; z)\rme^{ \rmi(k - z)(x - x') - (k^{2} - z^2)(y - y') } \,\rmd x'\rmd y'\nonumber\\
&\equiv 1 +  [\op{g}_{ u, \mrm{r} }^+ \mu^+](x, y; z),
\end{align}}
where
\begin{subequations}\label{Eq:+sets}\begin{align}
U_{ z, \mrm{r} }^+ &= (\omega\bb{N} + z) \cup ( -\omega\bb{N} - \bar{z} ) \label{Eq:U+},\\
V_{ z, \mrm{r} }^+ &= \{-\bar{z}, -\bar{z} + \omega, \dotsc, z - \omega, z\} \label{Eq:V+},
\end{align}\end{subequations}
and
\begin{align}\label{Eq:leftlimit}
\mu_-(x, y; z) &= 1 + [\op{g}_{ u, \mrm{r} } \mu]^-(x, y; z)\nonumber\\
&= 1 + \frac{1}{2\ell}\Bigg(\sum_{ \substack{ m > 0 \\ \substack{\tn{or} \\ \omega m \leq -\omega n} } } \int_{\!-\infty}^y \int_{\!-\ell}^\ell -
\sum_{-\omega n < \omega m \leq 0} \int_y^\infty \int_{\!-\ell}^\ell\Bigg)\nonumber\\
&\qquad\ph{1 + \frac{1}{2\ell}} u(x', y')\mu^-(x', y'; z)\rme^{ \rmi\omega m(x - x') - ( (\omega m + z)^2 - z^2 )(y - y') } \,\rmd x'\rmd y'\nonumber\\
&= 1 + \frac{1}{2\ell}\Bigg(\sum_{ \substack{ k > z \\ \substack{ \tn{or} \\ k \leq -\bar{z} } } } \int_{\!-\infty}^y \int_{\!-\ell}^\ell -
\sum_{-\bar{z} < k \leq z} \int_y^\infty \int_{\!-\ell}^\ell\Bigg)\nonumber\\
&\qquad\ph{1 + \frac{1}{2\ell}} u(x', y')\mu^-(x', y'; z)\rme^{ \rmi(k - z)(x - x') - (k^{2} - z^2)(y - y') } \,\rmd x'\rmd y'\nonumber\\
&= 1 + \frac{1}{2\ell}\Bigg(\sum_{k \in U_{ z, \mrm{r} }^-}
\int_{\!-\infty}^y \int_{\!-\ell}^\ell - \sum_{k \in V_{ z, \mrm{r} }^-} \int_y^\infty \int_{\!-\ell}^\ell\Bigg)\nonumber\\
&\qquad\ph{1 + \frac{1}{2\ell}} u(x', y')\mu^-(x', y'; z)\rme^{ \rmi(k - z)(x - x') - (k^{2} - z^2)(y - y') } \,\rmd x'\rmd y'\nonumber\\
&\equiv 1 +  [\op{g}_{ u, \mrm{r} }^- \mu^-](x, y; z),
\end{align}
with
\begin{subequations}\label{Eq:-sets}\begin{align}
U_{ z, \mrm{r} }^- &= (\omega\bb{N} + z) \cup ( -\omega\bb{N}_0 - \bar{z} ), \label{Eq:U-}\\
V_{ z, \mrm{r} }^- &= \{-\bar{z} + \omega, \dotsc, z - \omega, z\} \label{Eq:V-}.
\end{align}\end{subequations}
Thus,
\begin{align*}
\op{J}\mu &= \op{g}_{ u, \mrm{r} }^+ \mu^+ - \op{g}_{ u, \mrm{r} }^- \mu^-\\
&= \op{g}_{ u, \mrm{r} }^+ \mu^+ - \op{g}_{ u, \mrm{r} }^- \mu^+ + \op{g}_{ u, \mrm{r} }^- \mu^+ - \op{g}_{ u, \mrm{r} }^- \mu^-\\
&= (\op{g}_{ u, \mrm{r} }^+ - \op{g}_{ u, \mrm{r} }^-)\mu^+ + \op{g}_{ u, \mrm{r} }^- \op{J}\mu.
\end{align*}
But from equations \eqref{Eq:rightlimit}--\eqref{Eq:-sets},
\[
\begin{split}
(\op{g}_{ u, \mrm{r} }^+ - \op{g}_{ u, \mrm{r} }^-)\mu^+ = -\frac{1}{2\ell}\int_{\!-\infty}^\infty \int_{\!-\ell}^\ell &u(x', y')\mu^+(x', y'; z)\\
&\, \rme^{ \rmi(-\bar{z} - z)(x - x') - (\bar{z}^2 - z^2)(y - y') } \,\rmd x'\rmd y'.
\end{split}
\]
Recognizing $F(z)$ from its definition and setting
\[
d(x, y; z) = -\rmi( z + \bar{z} )x + (z^2 - \bar{z}^2)y,
\]
we obtain
\[
\op{J}\mu(x, y; z) = F(z)\rme^{ d(x, y; z) } + \op{g}_{ u, \mrm{r} }^- \op{J}\mu(x, y; z),
\]
or equivalently,
\[
(\op{Id} - \op{g}_{ u, \mrm{r} }^-)\op{J}\mu(x, y; z) = F(z)\rme^{ d(x, y; z) }.
\]
Now, $\op{g}_{ u, \mrm{r} }^-$ is the limit of the operator $\op{g}_{ u, \mrm{r} }$ as $\zeta$ approaches $z$ from the left of the line $L_n$ where
$\zeta$ is situated in the strip $S_{n - 1}$. But $\op{g}_{ u, \mrm{r} } = \Scr{N}_u$. Since $\Scr{N}_u$ is bounded with norm less than one, so is its limit. 
Thus, $\op{g}_{ u, \mrm{r} }^-$ is bounded and has norm less than one. Hence,
\[
\op{J}\mu(x, y; z) = (\op{Id} - \op{g}_{ u, \mrm{r} }^-)^{-1} F(z)\rme^{ d(x, y; z) } = F(z)(\op{Id} - \op{g}_{ u, \mrm{r} }^-)^{-1} \rme^{ d(x, y; z) },
\]
since $\op{g}_{ u, \mrm{r} }^-$ commutes with multiplication by functions of $z$ alone. It remains to calculate
$(\op{Id} - \op{g}_{ u, \mrm{r} }^-)^{-1} \rme^{ d(x, y; z) }$. Call this function $\nu(x, y; z)$. It is bounded because the exponential is bounded and
$\op{Id} - \op{g}_{ u, \mrm{r} }^-$ is invertible on $L^\infty(\Omega)$. It satisfies the equation
\[
\nu(x, y; z) = \rme^{ d(x, y; z) } + \op{g}_{ u, \mrm{r} }^- \nu(x, y; z).
\]
But then, $\wt{\nu}(x, y; z) = \nu(x, y; z)\rme^{ -d(x, y; z) }$ satisfies
\begin{align*}
\wt{\nu}(x, y; z) &= 1 + \rme^{ -d(x, y; z) } \op{g}_{ u, \mrm{r} }^- \nu(x, y; z)\\
&= 1 + \frac{1}{2\ell}\Bigg(\sum_{k \in U_{ z, \mrm{r} }^-}
\int_{\!-\infty}^y \int_{\!-\ell}^\ell - \sum_{k \in V_{ z, \mrm{r} }^-} \int_y^\infty \int_{\!-\ell}^\ell\Bigg)\\
&\qquad\ph{1 + \frac{1}{2\ell}} u(x', y')\wt{\nu}(x', y'; z)\rme^{ \rmi( k + \bar{z} )(x - x') - (k^{2} - \bar{z}^2)(y - y') } \,\rmd x'\rmd y'.
\end{align*}
This equation has a unique solution, which has already been named $\mu^-( x, y; -\bar{z} )$. Indeed, the number $-\bar{z}$ is located on the line
$L_{-n}$ since
\[
\real( -\bar{z} ) = \real(-z) = -\re{z} = -\frac{\omega}{2}n.
\]
Therefore,
\begin{align}
\mu_-( x, y; -\bar{z} ) &= 1 + [\op{g}_{ u, \mrm{l} } \mu]^-( x, y; -\bar{z} )\nonumber\\
&= 1 + \frac{1}{2\ell}\Bigg(\sum_{ \substack{ m \leq 0 \\ \substack{\tn{or} \\ \omega m > \omega n} } } \int_{\!-\infty}^y \int_{\!-\ell}^\ell -
\sum_{0 < \omega m \leq \omega n} \int_y^\infty \int_{\!-\ell}^\ell\Bigg)\nonumber\\
&\quad\ph{1 + \frac{1}{2\ell}} u(x', y')\mu^-( x', y'; -\bar{z} )
\rme^{ \rmi\omega m(x - x') - ( ( \omega m  -\bar{z} )^2 - \bar{z}^2 )(y - y') } \,\rmd x'\rmd y'\nonumber\\
&= 1 + \frac{1}{2\ell}\Bigg(\sum_{ \substack{ k \leq -\bar{z} \\ \substack{\tn{or} \\ k > z} } } \int_{\!-\infty}^y \int_{\!-\ell}^\ell -
\sum_{-\bar{z} < k \leq z} \int_y^\infty \int_{\!-\ell}^\ell\Bigg)\nonumber\\
&\quad\ph{1 + \frac{1}{2\ell}} u(x', y')\mu^-( x', y'; -\bar{z} )\rme^{ \rmi( k + \bar{z} )(x - x') - (k^{2} - \bar{z}^2)(y - y') } \,\rmd x'\rmd y'\nonumber\\
&= 1 + \frac{1}{2\ell}\Bigg(\sum_{k \in U_{ z, \mrm{r} }^-}
\int_{\!-\infty}^y \int_{\!-\ell}^\ell - \sum_{k \in V_{ z, \mrm{r} }^-} \int_y^\infty \int_{\!-\ell}^\ell\Bigg)\nonumber\\
&\quad\ph{1 + \frac{1}{2\ell}} u(x', y')\mu^-( x, y; -\bar{z} )\rme^{ \rmi( k + \bar{z} )(x - x') - (k^{2} - \bar{z}^2)(y - y') } \,\rmd x'\rmd y'.
\end{align}
Thus, $\wt{\nu}(x, y; z) = \mu^-( x, y; -\bar{z} )$ and so $\nu(x, y; z) = \mu^-( x, y; -\bar{z} )\rme^{ d(x, y; z) }$, which yields
\[
\op{J}\mu(x, y; z) = F(z)\mu^-( x, y; -\bar{z} )\rme^{ d(x, y; z) }.\qedhere
\]
\end{proof}
One easy consequence of the definition of the spectral data is the following estimate on $F(z)$.
\begin{proposition}
The function $F(z)$ is bounded. More precisely,
\begin{equation}
\abs{ F(z) } < \frac{\omega\norm{u}_1}{ 1 - \tfrac{C}{2\pi}\maxnorm{u} }, \quad \forall z \in \bb{C} \setminus \wpln.
\end{equation}
\end{proposition}
\begin{proof}
Let $z \in \bb{C} \setminus \wpln$. Then,
\begin{align*}
F(z) &= -\frac{ \sgn( \re{z} ) }{2\ell}\int_{\!-\infty}^\infty \int_{\!-\ell}^\ell
u(x, y)\mu^+(x, y; z)\rme^{ \rmi( z + \bar{z} )x - (z^2 - \bar{z}^2)y } \,\rmd x\rmd y\\
&= -\frac{ \sgn( \re{z} ) }{2\ell}\int_{\!-\infty}^\infty \int_{\!-\ell}^\ell u(x, y)\mu^+(x, y; z)\rme^{\rmi2\re{z}x - \rmi4\re{z}\im{z}y} \,\rmd x\rmd y \\
&= -\sgn( \re{z} )\wh{u\mu^+}(-\tfrac{ 2\re{z} }{\omega}, 4\re{z}\im{z}; z)\\
&= -\sgn( \re{z} )\wh{u\mu^+}(r_0(z); z).
\end{align*}
Consequently,
\[
\abs{ F(z) } = \abs{ \wh{u\mu^+}(r_0(z); z) } \leq \norm{ \wh{u\mu^+} }_\infty < \omega\norm{u\mu^+}_1 \leq
\omega\norm{u}_1 \norm{\mu^+}_\infty \leq \omega\norm{u}_1 \norm{\mu}_\infty.
\]
But $\mu = (\op{Id} - \Scr{N}_u)^{-1} 1$ and $\norm{\Scr{N}_u 1}_\infty \leq \tfrac{C}{2\pi}\maxnorm{u}$. Thus,
$\norm{\mu}_\infty \leq ( 1 - \tfrac{C}{2\pi}\maxnorm{u} )^{-1}$ and the proposition is proved.
\end{proof}
\begin{definition}
The \emph{spectral data} associated to a small potential $u(x, y)$ of the perturbed heat operator is the function defined by
\begin{equation}
F(z) = -\sgn( \re{z} )\wh{u\mu^+}(r_0(z); z), \quad z \in \bb{C} \setminus \wpln.
\end{equation}
Abusing notation, the bounded linear map $\Scr{S}$ determined by $F(z)$ shall also be called spectral data. Here
\begin{equation}\label{Eq:scatteringoperator}
(\Scr{S}\mu)(x, y; z) = F(z)\rme^{ \rmi r_0(z) \cdot (\omega x, y) } \mu^-( x, y; -\bar{z} ).
\end{equation}
\end{definition}
The function $F(z)$ behaves much like the Fourier transform of $u$. If $u$ is smooth, then $F$ has rapid decay in some directions. In particular, if $u$
is small, $y u \in L^1(\Omega)$ and
\[
\abs{\partial^{\alpha'} u} \in L^1(\Omega) \cap L^2(\Omega), \;\forall \:\abs{\alpha'} \leq \abs{\alpha},
\]
for some multi-index $\alpha = (\alpha_1, \alpha_2)$, then
\begin{equation}\label{Eq:datadecay}
\abs{ F(\tfrac{\omega}{2}n, \tau) } = O\bigg( \frac{1}{ 1 + \abs{\omega n}^{\alpha_1} + \abs{2\omega n\tau}^{\alpha_2} } \bigg),
\end{equation}
for $n \in \sZ$ and $\tau \in \bb{R}$. This follows immediately from the definition of $F$ and lemma~\ref{L:rapiddeacy}.
\begin{corollary}\label{C:squareintegrability}
Suppose that $\maxnorm{u}$ is small, $y u \in L^1(\Omega)$ and that $u(x, y)$ has two continuous derivatives in $L^1(\Omega) \cap L^2(\Omega)$. 
Then,
\begin{equation}\label{Eq:infinitybound}
\sup_{n \in \sZ} \abs{n}\sup_{ \tau \in \bb{R} } \abs{ F(\tfrac{\omega}{2}n, \tau) }< \infty,
\end{equation}
and
\begin{equation}\label{Eq:infinitytwobound}
\int_{\!-\infty}^\infty \abs{ F(\tfrac{\omega}{2}n, \tau) }^2 \,\rmd\tau = O\bigg( \frac{1}{n^4} \bigg),
\end{equation}
for all $n \in \sZ$.
\end{corollary}
\begin{proof}
Let $n \in \sZ$. From equation \eqref{Eq:datadecay},
\[
\abs{ F(\tfrac{\omega}{2}n, \tau) } \leq \frac{c}{1 + \abs{\omega n}^2 + \abs{2\omega n\tau}^2} < \frac{c}{\omega^2 n^2},
\]
for some positive, real constant $c$. Hence, $\abs{n}\sup_{ \tau \in \bb{R} } \abs{ F(\tfrac{\omega}{2}n, \tau) } < c \,\omega^{-2} \abs{n}^{-1}$
which yields
$\sup_{n \in \sZ} \abs{n}\sup_{ \tau \in \bb{R} } \abs{ F(\tfrac{\omega}{2}n, \tau) } < c \,\omega^{-2}$. Furthermore,
\[
\abs{ F(\tfrac{\omega}{2}n, \tau) } < \frac{c}{\abs{\omega n}^2 + \abs{2\omega n\tau}^2} = \frac{c}{ (\omega n)^2 (1 + (2\tau)^2) },
\]
thus,
\begin{align*}
\int_{\!-\infty}^\infty \abs{ F(\tfrac{\omega}{2}n, \tau) }^2 \,\rmd\tau &< \int_{\!-\infty}^\infty \frac{c^2}{ (\omega n)^4 (1 + (2\tau)^2)^2 } \,\rmd\tau\\
&= \frac{1}{n^4}\frac{c^2}{2\omega^4}\int_{\!-\infty}^\infty \frac{1}{ (1 + v^2)^2 } \,\rmd v < \infty.\qedhere
\end{align*}
\end{proof}
Suppose that $u(x, y)$ is a small potential in the sense of the preceding corollary. Then, the heat operator perturbed by $u$ has the associated spectral 
operator $\Scr{S}$, and its unique solution $\mu$ satisfies $\mu^+ - \mu^- = \Scr{S}\mu$. This is a Riemann--Hilbert problem on the infinite contours 
$L_n$, $n \in \sZ$. Since $\mu$ has the asymptotic behaviour $\mu \to 1$ as $\abs{ \im{z} } \to \infty$, the general solution to this problem is given by the 
Fredholm integral equation
\begin{equation}\label{Eq:Cauchyintegral}
\mu(x, y; z) = 1 + \frac{1}{2\pi \rmi}\sideset{}{'}\sum_{n = -\infty}^\infty \int_{L_n} \frac{ (\Scr{S}\mu)(x, y; \zeta) }{\zeta - z} \,\rmd\zeta.
\end{equation}
This equation is understood as the limit of the solution of the same Riemann--Hilbert problem on the contour $L = \sum_{\abs{n} \leq k} L_n$, when
$k \to \infty$ (see \cite{G66}). It can be shown that equation \eqref{Eq:Cauchyintegral} has a unique solution. Let $L^2( \abs{\!\real z} )$ denote the space 
of all measurable functions $f(z)$ on the jump contours such that
\begin{equation}
\norm{f}_{ L^2( \abs{\!\real z} ) }^2 \= \sideset{}{'}\sum_{n = -\infty}^\infty \int_{L_n} \,\abs{ f(z) }^2 \,\abs{\!\real z} \,\rmd z < \infty,
\end{equation}
and $\Lambda = \bigcup\limits_{n \in \sZ} L_n$. We have the following important result.
\begin{proposition}\label{Pr:important}
Suppose $\Scr{S} \colon L^\infty(E) \to L^2(\Scr{C}^*) \cap L^\infty(\Scr{C}^*)$ with
\[
(\Scr{S}f)(x, y; z) = F(z)\rme^{ \rmi r_0(z) \cdot (\omega x, y) } f^-( x, y; -\bar{z} ).
\]
If the function $F(z)$ is sufficiently small in $L^2( \abs{\!\real z} ) \cap L^\infty(\Lambda)$, then $\mathcal{C}\Scr{S}$ is a contraction of $L^\infty(E)$. 
Especially equation \eqref{Eq:Cauchyintegral} has a unique solution.
\end{proposition}
\begin{proof}Let $\mathcal{C}\Scr{S}\mu(x, y; z)$ denote the second term in the right hand side of the equation~\eqref{Eq:Cauchyintegral} where
$\mathcal{C}$ is the operator
\begin{equation}
\mathcal{C}f(z) = \frac{1}{2\pi \rmi}\sideset{}{'}\sum_{n = -\infty}^\infty \int_{L_n} \frac{ f(\zeta) }{\zeta - z} \,\rmd\zeta,
\end{equation}
which is defined at least for Schwartz functions on $\Scr{C}^*$. Choosing the param\-etrization
$\zeta = -\tfrac{\omega}{2}n - \rmi\tfrac{\tau}{2\omega n} \equiv \zeta(n, \tau)$ for the contour $L_n$, equation \eqref{Eq:Cauchyintegral} is
transformed into
\begin{equation}
\mu(x, y; z) = 1 + \frac{1}{2\pi}\sideset{}{'}\sum_{n = -\infty}^\infty \int_{\!-\infty}^\infty
\frac{ F \circ \zeta(n, \tau)\mu^-( x, y; -\bar{\zeta}(n, \tau) ) }{ P_z(n, \tau) }\rme^{\rmi\omega n x + \rmi\tau y} \,\rmd\tau.
\end{equation}
Taking absolute values, immediately shows
\[
\abs{ \mathcal{C}\Scr{S}\mu(x, y; z) } \leq \norm{\mu}_\infty \frac{1}{2\pi}\sideset{}{'}\sum_{n = -\infty}^\infty \int_{\!-\infty}^\infty
\Abs{ \frac{ F \circ \zeta(n, \tau) }{ P_z(n, \tau) } } \,\rmd\tau.
\]
Now, observe that the function $z \mapsto r_0(z)$ (on $\bb{C} \setminus \wpln$), defined in equation \eqref{Eq:nontrivialzero}, and
$(n, \tau) \mapsto \zeta(n, \tau)$ are inverses of each other. Thus,
\[
\abs{ F \circ \zeta(n, \tau) } = \abs{ \wh{u\mu_+}( n, \tau; \zeta(n, \tau) ) } < \omega\norm{u}_1 \norm{\mu}_\infty,
\]
which shows that $F \circ \zeta \in L^\infty(\Scr{C}^*)$. Also, from \eqref{Eq:datadecay}
\[
\abs{ F \circ \zeta(n, \tau) }^2 = \abs{ F( -\tfrac{\omega}{2}n, -\tfrac{\tau}{2\omega n} ) }^2 = O\bigg( \frac{1}{ (\omega^2 n^2 + \tau^2)^2 } \bigg).
\]
A simple calculation yields $1/(\omega^2 n^2 + \tau^2)^2 \in L^1(\Scr{C}^*)$, thus, $F \circ \zeta \in L^2(\Scr{C}^*)$. Hence, as in 
the \hyperlink{L:basiclemma}{Basic Lemma}
\[
\sideset{}{'}\sum_{n = -\infty}^\infty \int_{\!-\infty}^\infty \Abs{ \frac{ F \circ \zeta(n, \tau) }{ P_z(n, \tau) } } \,\rmd\tau \leq C\max
\{ \norm{F \circ \zeta}_{ L^2(\Scr{C}^*) }, \,\norm{F \circ \zeta}_{ L^\infty(\Scr{C}^*) } \}.
\]
It is easy to see that $\norm{F \circ \zeta}_{ L^\infty(\Scr{C}^*) } = \norm{F}_{ L^\infty(\Lambda) }$ and
$\norm{F \circ \zeta}_{ L^2(\Scr{C}^*) }^2 = 4\norm{F}_{ L^2( \abs{\!\real z} ) }^2$. Therefore, if we have
$C\max\{ \norm{F \circ \zeta}_{ L^2(\Scr{C}^*) }, \,\norm{F \circ \zeta}_{ L^\infty(\Scr{C}^*) } \} < 2\pi$ or equivalently
$C\max\{ 2\norm{F}_{ L^2( \abs{\!\real z} ) }, \,\norm{F}_{ L^\infty(\Lambda) } \} < 2\pi$, then $\mathcal{C}\Scr{S}$ is a contraction of $L^\infty(E)$, 
thus, equation \eqref{Eq:Cauchyintegral} has a unique bounded solution.
\end{proof}

We consider again equation \eqref{Eq:Cauchyintegral}, i.e., $\mu = 1 + \mathcal{C}\Scr{S}\mu$, and apply the operator $P(\partial + w)$. Then, 
\[
-u\mu = P(\partial + w)\mu = P(\partial + w)1 + P(\partial + w)\mathcal{C}\Scr{S}\mu = P(\partial + w)\mathcal{C}\Scr{S}\mu.
\]
Since the expression in the definition of $\mathcal{C}\Scr{S}\mu$ converges absolutely, $\mathcal{C}\Scr{S}\mu$ may be differentiated
in the parameters $x$ and $y$ and the following formulas are easily derived.
\begin{lemma}\label{L:commutators}\hfill
\begin{enumerate}
\item $[ P(\partial + w), \Scr{S} ] = 0$,
\item $\displaystyle[ P(\partial + w), \mathcal{C} ]f(x, y; z) = -\frac{1}{\pi}\partial_x \sideset{}{'}\sum_{n = -\infty}^\infty \int_{L_n}
f(x, y; \zeta) \,\rmd\zeta$. This is independent of $z$.
\end{enumerate}
\end{lemma}
As a consequence of this lemma,
\begin{align*}
-u\mu &= \mathcal{C}\Scr{S}(P(\partial + w)\mu) + [ P(\partial + w), \mathcal{C} ]\Scr{S}\mu\\
&= -u\mathcal{C}\Scr{S}\mu - \frac{1}{\pi}\partial_x \sideset{}{'}\sum_{n = -\infty}^\infty \int_{L_n} \Scr{S}\mu \,\rmd\zeta\\
&= -u(\mu - 1) - \frac{1}{\pi}\partial_x \sideset{}{'}\sum_{n = -\infty}^\infty \int_{L_n} \Scr{S}\mu \,\rmd\zeta,
\end{align*}
hence
\begin{equation}\label{Eq:potenial}
u(x, y) = \frac{1}{\pi}\partial_x \sideset{}{'}\sum_{n = -\infty}^\infty \int_{L_n} (\Scr{S}\mu)(x, y; \zeta) \,\rmd\zeta.
\end{equation}
The proof of the main result of this section is obtained by combining the results of theorems~\ref{Th:exist_unique}, \ref{Th:holomorphicity}, 
\ref{Th:zasympotics}, \ref{Th:departurefromholomorphicity} and proposition \ref{Pr:important}.
\begin{theorem}[The Forward Spectral Theorem]\label{Th:forwardscattering}
Suppose that $u(x, y)$ is small in $L^1(\Omega) \cap L^2(\Omega)$ with $y u(x, y) \in L^1(\Omega)$ and that
$\partial^\alpha u \in L^1(\Omega) \cap L^2(\Omega)$ for all $\abs{\alpha} \leq 2$. Then, the unique solution $\mu(x, y; z)$ to the equation 
\[
(-\partial_y + \partial_x^2 + 2i z\partial_x + u)\mu(x, y; z) = 0, \qquad \lim_{\abs{y} \to \infty} \mu(x, y; z) = 1,
\]
is also the unique solution to the equation
\[
\mu(x, y; z) = 1 + \frac{1}{2\pi \rmi}\sideset{}{'}\sum_{n = -\infty}^\infty \int_{L_n} \frac{ (\Scr{S}\mu)(x, y; \zeta) }{\zeta - z} \,\rmd\zeta,
\]
where $\Scr{S}$ is the spectral data associated to u by the heat operator, defined in equation \eqref{Eq:scatteringoperator} and $u$ can be found via 
equation \eqref{Eq:potenial}.
\end{theorem}
We finish this section proving a property of the Jost function $\mu$, crucial to the inverse problem.
\begin{theorem}\label{Th:squareintegrability}
Suppose $\maxnorm{u} < 2\pi/C$ and that $\abs{\partial^\alpha u} \in L^1(\Omega) \cap L^2(\Omega)$ for all $\abs{\alpha} \leq 3$. Then, the function
$\mu(x, y; z) - 1$ is holomorphic with respect to $z \in \wpln$ and
\begin{equation}\label{Eq:meanintegrability}
\sup_{\re{z} \notin \frac{\omega}{2}\sZ} \bigg(\int_{\!-\infty}^\infty \abs{\mu( x, y; \re{z} + \rmi\im{z} ) - 1}^2 \,\rmd\im{z}\bigg)^\frac{1}{2} < \infty,
\end{equation}
for all $(x, y) \in \Omega$.
\end{theorem}
\begin{proof}
Fix a $z$ in $\wpln$. A straightforward algebraic manipulation yields
\begin{equation*}
\frac{1}{ P_z(m, \xi) } = \frac{1}{2\omega m z}\bigg( 1 - \frac{ (\omega m)^2 + \rmi\xi }{ P_z(m, \xi) } \bigg) =
\frac{1}{2z}\bigg( \frac{1}{\omega m} - \frac{\omega m}{ P_z(m, \xi) } - \frac{\rmi\xi}{ \omega m P_z(m, \xi) } \bigg).
\end{equation*}
Hence,
\begin{align*}
\frac{ \wh{u\mu}(m, \xi; z) }{ P_z(m, \xi) } &=
\frac{1}{2z}\bigg( \frac{ \wh{u\mu}(m, \xi; z) }{\omega m} - \frac{ \omega m\wh{u\mu}(m, \xi; z) }{ P_z(m, \xi) }
- \frac{ \rmi\xi\wh{u\mu}(m, \xi; z) }{ \omega m P_z(m, \xi) } \bigg)\\
&= \frac{1}{2z}\bigg( \frac{ \wh{u\mu}(m, \xi; z) }{\omega m} + \rmi\frac{ \wh{\partial_x u\mu}(m, \xi; z) }{ P_z(m, \xi) }
- \frac{ \wh{\partial_y u\mu}(m, \xi; z) }{ \omega m P_z(m, \xi) } \bigg).
\end{align*}
From lemma \ref{L:rapiddeacy} we have
\[
\abs{ \wh{u\mu}(m, \xi; z) } \leq \frac{c_1}{1 + \abs{\omega m}^2 + \abs{\xi}^2}
< \frac{c_1}{ (\omega m)^2 \big(1 + \big( \frac{\xi}{ \omega\abs{m} } \big)^2\big) },
\]
for some positive, real constant $c_1$. Thus,
\[
\frac{ \abs{ \wh{u\mu}(m, \xi; z) } }{ \omega\abs{m} } < c_1\frac{1}{ (\omega m)^2 }
\frac{1}{ \omega\abs{m}\big(1 + \big( \frac{\xi}{ \omega\abs{m} } \big)^2\big) } \in L^1(\Scr{C}^*).
\]
Another application of lemma \ref{L:rapiddeacy} shows that
\[
\abs{ \wh{\partial_x u\mu}(m, \xi; z) } = O\bigg( \frac{1}{1 + \abs{\omega m}^2 + \abs{\xi}^2} \bigg),
\]
from which it follows that $\wh{\partial_x u\mu}(m, \xi; z) \in L^2(\Scr{C}^*) \cap L^\infty(\Scr{C}^*)$. Thus, as a consequence of the 
\hyperlink{L:basiclemma}{Basic Lemma}, $\wh{\partial_x u\mu}(m, \xi; z)/P_z(m, \xi) \in L^1(\Scr{C}^*)$. Finally, by lemma \ref{L:rapiddeacy} again, 
there exists a positive, real constant $c_2$ such that
\[
\abs{ \wh{\partial_y u\mu}(m, \xi; z) } \leq \frac{c_2}{1 + \abs{\omega m}^2 + \abs{\xi}^2}.
\]
But then, an application of H\"{o}lder's inequality gives
\begin{align*}
\Norm{ \frac{\wh{\partial_y u\mu}(m, \xi; z) }{ \omega m P_z(m, \xi) } }_{ L^1(\Scr{C}^*) } &\leq
\Norm{ \frac{c_2}{1 + \abs{\omega m}^2 + \abs{\xi}^2}\frac{1}{ \omega m P_z(m, \xi) } }_{ L^1(\Scr{C}^*) }\\
&\leq \Norm{ \frac{c_2}{ \omega m(1 + \abs{\omega m}^2 + \abs{\xi}^2) } }_{ L^2(\Scr{C}^*) } \Norm{ \frac{1}{ P_z(m, \xi) } }_{ L^2(\Scr{C}^*) }\\
&< \infty.
\end{align*}
Therefore,
\[
I \equiv \sideset{}{'}\sum_{m = -\infty}^\infty \int_{\!-\infty}^\infty
\,\Abs{ \frac{ \wh{u\mu}(m, \xi; z) }{\omega m} + \rmi\frac{ \wh{\partial_x u\mu}(m, \xi; z) }{ P_z(m, \xi) }
- \frac{ \wh{\partial_y u\mu}(m, \xi; z) }{ \omega m P_z(m, \xi) } }\,\rmd\xi < \infty.
\]

Hence,
\[
\abs{\mu(x, y ; z) - 1}^2 \leq \frac{1}{ (2\pi)^2 }\Bigg(\sideset{}{'}\sum_{m = -\infty}^\infty \int_{\!-\infty}^\infty
\,\Abs{ \frac{ \wh{u\mu}(m, \xi; z) }{ P_z(m, \xi) } } \,\rmd\xi\Bigg)^2 = \frac{1}{ (2\pi)^2 }\frac{1}{4\abs{z}^2}I^2 = \frac{c_3}{\abs{z}^2},
\]
with $c_3$ a real constant. For $\re{z} \neq 0$, $1/\abs{z}^2$ is integrable over the real line (with respect to $\imaginary z$) and
\[
\int_{\!-\infty}^\infty \frac{1}{\abs{z}^2} \,\rmd\imaginary z = \int_{\!-\infty}^\infty \frac{1}{\abs{ \re{z} + \rmi\im{z} }^2} \,\rmd\im{z} =
\int_{\!-\infty}^\infty \frac{1}{\re{z}^2 + \im{z}^2} \,\rmd\im{z} = \frac{\pi}{ \abs{ \re{z} } }.
\]
Thus,
\[
\int_{\!-\infty}^\infty \abs{\mu( x, y; \re{z} + \rmi\im{z} ) - 1}^2 \,\rmd\im{z} < \frac{c_3 \pi}{ \abs{ \re{z} } } < \infty.
\]
When $\re{z} = 0$ the conclusion follows from Fatou's lemma: let $\{ { \re{z} }_n \}$ be a sequence of real numbers not in $\frac{\omega}{2}\bb{Z}$ such 
that ${ \re{z} }_n \to 0$. Then,
\begin{align*}
\int_{\!-\infty}^\infty \abs{\mu( x, y; \rmi\im{z} ) - 1}^2 \,\rmd\im{z} &\leq \liminf_{n \to \infty} \int_{\!-\infty}^\infty
\abs{\mu( x, y; { \re{z} }_n + \rmi\im{z} ) - 1}^2 \,\rmd\im{z}\\
&< \liminf_{n \to \infty} \frac{c_3 \pi}{ \abs{ { \re{z} }_n } } = \infty.\qedhere
\end{align*}
\end{proof}
\section[The Inverse Problem]{The Inverse Problem}\label{S:inverse}
\smallskip
\subsection[An appropriate Space for the Inverse Problem]{An appropriate Space for the Inverse Problem}\label{s:Hspace}
\separate

Any small spectral operator $\Scr{S}$ of the form in proposition \ref{Pr:important} determines a unique solution $\mu(x, y; z)$ to $\mu = 1 + \mathcal{C}\Scr{S}\mu$ which also solves the 
equation~\eqref{Eq:bvp} with some potential $u(x, y)$. The operator $\Scr{S}$ is defined for $L^\infty(E)$ functions, holomorphic in $\wpln$ and for which the one-sided limits at the
lines $L_n, n \in \sZ$ exist. An appropriate space for the class of functions having the properties as in proposition \ref{Pr:important} is the Hardy-like space defined as follows:
\begin{equation}
\Scr{H}_\omega \= \{f \in H(\wpln) \colon \norm{f}_\omega < \infty\},
\end{equation}
where
\begin{equation}
\norm{f}_\omega \= \sup_{\re{z} \notin \frac{\omega}{2}\sZ} \bigg(\int_{\!-\infty}^\infty \abs{ f( \re{z} + \rmi\im{z} ) }^2 \,\rmd\im{z}\bigg)^\frac{1}{2} =
\sup_{\re{z} \notin \frac{\omega}{2}\sZ} \norm{ f_{ \re{z} } }_2,
\end{equation}
with $f_{ \re{z} }( \im{z} ) = f( \re{z} + \rmi\im{z} )$. The pair $(\Scr{H}_\omega, \,\norm{\cdot}_\omega)$ forms a normed vector space: let $f$ and $g$ in $\Scr{H}_\omega$ and
$\lambda$ a complex number. Evidently, the function $f + \lambda g$ is holomorphic in $\wpln$. Furthermore, if $\re{z} \notin \frac{\omega}{2}\sZ$, by Minkowski's inequality
\begin{align*}
\int_{\!-\infty}^\infty \abs{ (f + \lambda g)( \re{z} + \rmi\im{z} ) }^2 \,\rmd\im{z} &= \int_{\!-\infty}^\infty \abs{ f( \re{z} + \rmi\im{z} ) + \lambda g( \re{z} + \rmi\im{z} ) }^2 \,\rmd\im{z}\\
&= \norm{ f_{ \re{z} } + \lambda g_{ \re{z} } }_2^2\\
&\leq (\norm{ f_{ \re{z} } }_2 + \abs{\lambda}\norm{ g_{ \re{z} } }_2)^2.
\end{align*}
Thus,
\[
\bigg(\int_{\!-\infty}^\infty \abs{ (f + \lambda g)( \re{z} + \rmi\im{z} ) }^2 \,\rmd\im{z}\bigg)^\frac{1}{2} \leq \norm{f}_\omega + \abs{\lambda}\norm{g}_\omega,
\]
which shows that $f + \lambda g$ belongs to $\Scr{H}_\omega$ and $\norm{f + \lambda g}_\omega \leq \norm{f}_\omega + \abs{\lambda}\norm{g}_\omega$. Moreover
$\norm{f}_\omega = 0$ if and only if $f = 0$.
\begin{lemma}
Let $f \in \Scr{H}_\omega$ and $S_{a, b} = \{z \in \wpln \mid a < \re{z} < b, \,\im{z} \in \bb{R}\}$ where either $a = \tfrac{\omega}{2}n, \,b = \tfrac{\omega}{2}(n + 1)$ or
$a = -\tfrac{\omega}{2}(n + 1), \,b = -\tfrac{\omega}{2}n$ or $a = -b = -\tfrac{\omega}{2}$ with $n \in \bb{N}$, i.e., $S_{a, b}$ a strip in $\wpln$. Then, for $z \in S_{a, b}$
\[
\abs{ f( \re{z} + \rmi\im{z} ) } \leq \sqrt{ \frac{2}{\pi} }\norm{f}_\omega ( \min\{ \abs{\re{z} - a}, \,\abs{\re{z} - b} \} )^{ -\frac{1}{2} }.
\]
Furthermore, if $K$ is a compact subset of $S_{a, b}$ then, $f$ is bounded in $K$.
\end{lemma}
\begin{proof}
Let $z \in S_{a, b}$ and consider the closed disc $D(z, r)$ centred at $z$ with radius $r = \min\{ \abs{\re{z} - a}, \,\abs{\re{z} - b} \}$. Then, $D(z, r) \sse
\overline{S}_{\re{z} - r, \re{z} + r} \sse S_{a, b}$: if $\sigma + \rmi\tau \in D(z, r)$, then $\abs{\re{z} - \sigma} \leq r$ and $\abs{\im{z} - \tau} \leq r$, thus
$z \in \overline{S}_{\re{z} - r, \re{z} + r}$ and because $a < \re{z} - r$ and $\re{z} + r < b$, we also have $\overline{S}_{\re{z} - r, \re{z} + r} \sse S_{a, b}$. Since the function
$\abs{ f(z) }^2$ is subharmonic in $S_{a, b}$,
\begin{align*}
\abs{ f( \re{z} + \rmi\im{z} ) }^2 &\leq \frac{1}{\pi r^2}\iint_{ D(z, r) } \abs{ f(z) }^2 \,\rmd z\\
&\leq \frac{1}{\pi r^2}\iint_{ \overline{S}_{\re{z} - r, \re{z} + r} } \abs{ f(z) }^2 \,\rmd z\\
&= \frac{1}{\pi r^2}\int_{\re{z} - r}^{\re{z} + r} \int_{\!-\infty}^\infty \abs{ f( \re{z} + \rmi\im{z} ) }^2 \,\rmd\im{z}\rmd\re{z}\\
&\leq \frac{1}{\pi r^2}\int_{\re{z} - r}^{\re{z} + r} \norm{f}_\omega^2 \,\rmd\re{z} = \frac{2}{\pi}\norm{f}_\omega^2 \frac{1}{r}.
\end{align*}

If $z \in K \sse S_{a, b}$, and $K$ is a compact set, there exist some real numbers $M_1$, $M_2$ such that $a < M_1 \leq \re{z} \leq M_2 < b$, hence
$\re{z} - a \geq M_1 - a$ and $b - \re{z} \geq b - M_2$. Thus,
\[
\abs{ f(z) } \leq \sqrt{ \frac{2}{\pi} }\norm{f}_\omega ( \min\{ \abs{M_1 - a}, \,\abs{M_2 - b} \} )^{ -\frac{1}{2} },
\]
showing that $f$ is bounded in $K$.
\end{proof}
\begin{proposition}
$(\Scr{H}_\omega, \,\norm{\cdot}_\omega)$ is a Banach space.
\end{proposition}
\begin{proof}
Suppose $\{f_m\}_{m = 1}^\infty$ is a Cauchy sequence in $\Scr{H}_\omega$ and $K$ is a compact subset of $\wpln$. If $z \in K$, then $M_1 \leq \re{z} \leq M_2$ for some real
numbers $M_1$ and $M_2$. Since $K \sse \wpln$, there exist nonzero integers $n_1$, $n_2$ such that $\tfrac{\omega}{2}n_1 < M_1$ and $M_2 < \tfrac{\omega}{2}n_2$. Thus,
there exist numbers $a, \,b$ such that $a < M_1 \leq \re{z} \leq M_2 < b$. But then, by the previous lemma
\[
\abs{ f_m(z) - f_l(z) } \leq \sqrt{ \frac{2}{\pi} }\norm{f_m - f_l}_\omega ( \min\{ \abs{M_1 - a}, \,\abs{M_2 - b} \} )^{ -\frac{1}{2} },
\]
which shows that $\{f_m\}$ is uniformly Cauchy in compact subsets of $\wpln$ and thus, converges uniformly on compact subsets of $\wpln$ to some function $f \in H(\wpln)$.
Now, given $\epsilon > 0$, there exists $N \in \bb{N}$ such that $\norm{f_m - f_N}_\omega < \tfrac{\epsilon}{2}$ for all $m \geq N$. Then, by Fatou's lemma
\begin{align*}
\int_{\!-\infty}^\infty \abs{ (f - f_N)( \re{z} + \rmi\im{z} ) }^2 \,\rmd\im{z} &= \int_{\!-\infty}^\infty \lim_{m \to \infty} \abs{ (f_m - f_N)( \re{z} + \rmi\im{z} ) }^2 \,\rmd\im{z}\\
&\leq \lim_{m \to \infty} \int_{\!-\infty}^\infty \abs{ (f_m - f_N)( \re{z} + \rmi\im{z} ) }^2 \,\rmd\im{z}\\
&\leq \lim_{m \to \infty} \norm{f_m - f_N}_\omega^2 < \Big(\frac{\epsilon}{2}\Big)^2.
\end{align*}
Hence, $\norm{f - f_N}_\omega < \tfrac{\epsilon}{2}$ and so it follows that $\norm{f}_\omega < \infty$ and $\norm{f_m - f}_\omega \to 0$ as $m \to \infty$. Thus, the space
$\Scr{H}_\omega$ is complete under the norm $\norm{\cdot}_\omega$.
\end{proof}
Now let $\ell^\infty( L^2( \bb{R} ) ) \equiv \ell^\infty( \sZ, L^2( \bb{R} ) )$, that is
\begin{equation}
\ell^\infty( L^2( \bb{R} ) ) = \{g = \{g_m\} \colon g_m \in L^2( \bb{R} ), \ \forall \:m \in \sZ \text{ and } \norm{g}_{2, \infty} < \infty\},
\end{equation}
where
\begin{equation}
\norm{g}_{2, \infty} \= \sup_{m \in \sZ} \norm{g_m}_2 = \sup_{m \in \sZ} \bigg(\int_{\!-\infty}^\infty \abs{ g_m(x) }^2 \,\rmd x\bigg)^\frac{1}{2}.
\end{equation}
It is easily seen that $\norm{\cdot}_{2, \infty}$ defines a norm on $\ell^\infty( L^2( \bb{R} ) )$ under which it becomes a Banach space.

The following lemma is a Paley--Wiener type theorem for functions holomorphic in strips \cite {PW34} .
\begin{lemma}\label{L:Paley-Wiener}
Let $f \in \Scr{H}_\omega$. Then, there exists a measurable function $G$ such that
\[
\int_{\!-\infty}^\infty \abs{ G(\xi) }^2 \rme^{2\re{z}\xi} \,\rmd\xi < \infty,
\]
and
\begin{equation}
f(z) = \frac{1}{2\pi}\int_{\!-\infty}^\infty G(\xi)\rme^{z\xi} \,\rmd\xi,
\end{equation}
in the sense of $L^2$ convergence, for $z = \re{z} + \rmi\im{z} \in K \sse S_{a, b} \sse \bb{C}_\omega$ where $K$ is a compact set.
\end{lemma}
Let $f \in \Scr{H}_\omega$ and $n \in \sZ$. Suppose $x_k$ is a sequence such that $x_k \to \frac{\omega}{2}n^+$ (without loss of generality, let n be positive). Then, if $\delta > 0$,
there exists a positive integer $k_0$ such that $\frac{\omega}{2}n < x_k \leq \frac{\omega}{2}n + \delta < \frac{\omega}{2}(n + 1)$, for $k \geq k_0$. By lemma \ref{L:Paley-Wiener},
\[
f( x_k + \rmi\im{z} ) = \frac{1}{2\pi}\int_{\!-\infty}^\infty G(\xi)\rme^{ ( x_k + \rmi\im{z} )\xi } \,\rmd\xi,
\]
for some measurable function $G$ such that
\[
\int_{\!-\infty}^\infty \abs{ G(\xi) }^2 \rme^{2x_k \xi} \,\rmd\xi < \infty.
\]
Now, $G(\xi)\rme^{ ( x_k + \rmi\im{z} )\xi } \to G(\xi)\rme^{ ( \frac{\omega}{2}n + \rmi\im{z} )\xi }$ and $\abs{ G(\xi)\rme^{ ( x_k + \rmi\im{z} )\xi } } = \abs{ G(\xi) }\rme^{x_k \xi}\! \in \!L^1$ since
\begin{align*}
\int_{\!-\infty}^\infty \abs{ G(\xi) }\rme^{x_k \xi} \,\rmd\xi &= \int_{\!-\infty}^0 \abs{ G(\xi) }\rme^{ (x_k - \epsilon)\xi }\rme^{\epsilon\xi} \,\rmd\xi
+ \int_0^\infty \abs{ G(\xi) }\rme^{ (x_k + \epsilon)\xi }\rme^{-\epsilon\xi} \,\rmd\xi\\
&\leq \bigg(\int_{\!-\infty}^\infty \abs{ G(\xi) }^2 \rme^{2(x_k - \epsilon)\xi} \,\rmd\xi \int_{\!-\infty}^0 \rme^{2\epsilon\xi} \,\rmd\xi\bigg)^\frac{1}{2}\\
&{}+ \bigg(\int_{\!-\infty}^\infty \abs{ G(\xi) }^2 \rme^{2(x_k + \epsilon)\xi} \,\rmd\xi \int_0^\infty \rme^{-2\epsilon\xi} \,\rmd\xi\bigg)^\frac{1}{2} < \infty,
\end{align*}
for $\epsilon < \min\{x_k - \frac{\omega}{2}n, \,\frac{\omega}{2}n + \delta - x_k\}$. Thus, by dominated convergence, the limit $\lim\limits_{k \to \infty} f( x_k + \rmi\im{z} )$, i.e., the
limit $\lim\limits_{\re{z} \to \frac{\omega}{2}n^+} f( \re{z} + \rmi\im{z} ) \equiv f^+( \frac{\omega}{2}n + \rmi\im{z} )$ exists and also
\begin{equation}
f^+( \tfrac{\omega}{2}n + \rmi\im{z} ) = \frac{1}{2\pi}\int_{\!-\infty}^\infty G(\xi)\rme^{ ( \frac{\omega}{2}n + \rmi\im{z} )\xi } \,\rmd\xi,
\end{equation}
in the sense of $L^2$ convergence and pointwise almost everywhere convergence, since $f( \re{z} + \rmi\im{z} ) = \Scr{F}( G(\xi)\rme^{ \re{z}\xi } )( -\im{z} )$ and by Plancherel's theorem
$\Scr{F}( G(\xi)\rme^{ \re{z}\xi } )$ is square integrable. Similarly, we can establish the existence of the limit
\begin{equation}
f^-( \tfrac{\omega}{2}n + \rmi\im{z} ) \equiv \lim_{\re{z} \to \tfrac{\omega}{2}n^-} f( \re{z} + \rmi\im{z} )
\end{equation}
in the sense of $L^2$ and pointwise almost everywhere convergence.

From the above discussion, it is natural to define the bounded linear operator
$\op{J} \colon (\Scr{H}_\omega, \,\norm{\cdot}_\omega) \to ( \ell^\infty( L^2( \bb{R} ) ), \,\norm{\cdot}_{2, \infty} )$ by
\begin{equation}
(\op{J}f)_n(y) = \op{J}f(n, y) \= f^+(\tfrac{\omega}{2}n + \rmi y) - f^-(\tfrac{\omega}{2}n + \rmi y).
\end{equation}
We also define the linear operator $\op{l} \colon (\Scr{H}_\omega, \,\norm{\cdot}_\omega) \to ( \ell^\infty( L^2( \bb{R} ) ), \,\norm{\cdot}_{2, \infty} )$ with
\begin{equation}
(\op{l}f)_n(y) = \op{l}f(n, y) \= f^-(\!-\tfrac{\omega}{2}n + \rmi y).
\end{equation}
$\op{l}$ is bounded and $\norm{\op{l}f}_{2, \infty} \leq \norm{f}_\omega$.

Let us now state the following lemma, which expresses a simple identity.
\begin{lemma}\label{L:expintegral}
For $\tau \in \bb{R}$
\begin{equation}
\frac{1}{\zeta + \rmi\tau} = \begin{cases}
\displaystyle\int_{\!-\infty}^0 \rme^{ (\zeta + \rmi\tau)\xi } \,\rmd\xi, & \real \zeta > 0\\[15pt]
\displaystyle-\int_0^\infty \rme^{ (\zeta + \rmi\tau)\xi } \,\rmd\xi, & \real \zeta < 0.
\end{cases}
\end{equation}
\end{lemma}

Let $g = g_n(\tau) = g(\frac{\omega}{2}n + \rmi\tau) \in \ell^\infty( L^2( \bb{R} ) )$. Fix a positive integer $n$ (for $n$ negative the analysis is similar) and define the function
\begin{equation}
h_n(z) = \frac{1}{2\pi}\int_{\!-\infty}^\infty \frac{ g(\frac{\omega}{2}n + \rmi\tau) }{\frac{\omega}{2}n + \rmi\tau - z} \,\rmd\tau, \quad z \in S_{n - 1} \cup S_n.
\end{equation}
An application of H\"{o}lder's inequality shows that the function $h_n(z)$ is holomorphic. Assume that $z \in S_n$, i.e., $\frac{\omega}{2}n < \re{z}$.~Using lemma \ref{L:expintegral},
we can rewrite $h_n$ as
\begin{equation}
h_n(z) = \int_{\!-\infty}^\infty G_n(\xi)\Psi_c(\xi)\rme^{-\rmi\im{z}\xi} \,\rmd\xi = \Scr{F}( G_n(\xi)\Psi_c(\xi) )( \im{z} ),
\end{equation}
where $G_n(\xi) = \Scr{F}^{-1}( g(\frac{\omega}{2}n + \rmi\tau) )(\xi), \,c = \re{z} - \frac{\omega}{2}n$ and $\Psi_c(\xi) = -\rme^{-c\xi}$ for $\xi > 0$, $\Psi_c(\xi) = 0$ for $\xi < 0$.
Since $g_n \in L^2( \bb{R} )$, Plancherel's theorem yields
\begin{align*}
\int_{\!-\infty}^\infty \abs{ h_n( \re{z} + \rmi\im{z} ) }^2 \,\rmd\im{z} &= 2\pi\int_{\!-\infty}^\infty \abs{ G_n(\xi)\Psi_c(\xi) }^2 \,\rmd\xi\\
&< 2\pi\int_{\!-\infty}^\infty \abs{ G_n(\xi) }^2 \,\rmd\xi\\
&= \int_{\!-\infty}^\infty \abs{ g(\tfrac{\omega}{2}n + \rmi\tau) }^2 \,\rmd\tau = \norm{g_n}_2^2.
\end{align*}
Similarly, when $z \in S_{n - 1}$, i.e., $\frac{\omega}{2}n > \re{z}$,
\begin{equation}
h_n(z) = \Scr{F}( G_n(\xi)\Psi_c(\xi) )( \im{z} ),
\end{equation}
where this time $\Psi_c(\xi) = \rme^{-c\xi}$ for $\xi < 0$ and $\Psi_c(\xi) = 0$ for $\xi > 0$. Hence, once again by Plancherel's theorem
\[
\int_{\!-\infty}^\infty \abs{ h_n( \re{z} + \rmi\im{z} ) }^2 \,\rmd\im{z} < \norm{g_n}_2^2.
\]
\noi Therefore, since $n$ was arbitrary, we have that $h_n \in H(\wpln)$ for every nonzero integer $n$ and
\begin{equation}
\sup_{\re{z} \notin \frac{\omega}{2}\sZ} \bigg(\int_{\!-\infty}^\infty \abs{ h_n( \re{z} + \rmi\im{z} ) }^2 \,\rmd\im{z}\bigg)^\frac{1}{2} < \norm{g}_{2, \infty},
\end{equation}
in particular $h_n \in \Scr{H}_\omega$ for all $n \in \sZ$. A straightforward change of variables shows that
\begin{equation}
h_n(z) = \frac{1}{2\pi \rmi}\int_{\frac{\omega}{2}n + \rmi\infty}^{\frac{\omega}{2}n - \rmi\infty} \frac{ g(\zeta) }{\zeta - z} \,\rmd\zeta =
\frac{1}{2\pi \rmi}\int_{L_n} \frac{ g(\zeta) }{\zeta - z} \,\rmd\zeta.
\end{equation}
Applying the operator $\op{J}$ to the sequence $h_n$ and using the Plemelj--Sokhotski formulas for $L^2$ potentials, we have
\begin{equation}
(\op{J}h_n)_n(y) = h_n^+(\tfrac{\omega}{2}n + \rmi y) - h_n^-(\tfrac{\omega}{2}n + \rmi y) = g(\tfrac{\omega}{2}n + \rmi y).
\end{equation}
\begin{proposition}\label{Pr:lemma}
Let $a_n(\tau) = a(\frac{\omega}{2}n + \rmi\tau)$ be a sequence of complex functions such that
\begin{equation}
\sup_{ \tau \in \bb{R} } \abs{ a_n(\tau) } \leq \frac{c}{n^2}, \quad n \in \sZ,
\end{equation}
where $c$ is some constant, and let $g = g_n(\tau) = g(\frac{\omega}{2}n + \rmi\tau) \in \ell^\infty( L^2( \bb{R} ) )$. Then, the series
\[
\sideset{}{'}\sum_{n = -\infty}^\infty \frac{1}{2\pi}\int_{\!-\infty}^\infty \frac{ a(\frac{\omega}{2}n + \rmi\tau)g(\frac{\omega}{2}n + \rmi\tau) }{\frac{\omega}{2}n + \rmi\tau - z} \,\rmd\tau
\]
converges uniformly on compact subsets of $\wpln$ to a function $f \in \Scr{H}_\omega$ and
\begin{equation}
\norm{f}_\omega \leq c\frac{\pi^2}{3}\norm{g}_{2, \infty} \equiv \Gamma_c \norm{g}_{2, \infty},
\end{equation}
and
\begin{equation}
(\op{J}f)_n(y) = a(\tfrac{\omega}{2}n + \rmi y)g(\tfrac{\omega}{2}n + \rmi y).
\end{equation}
\end{proposition}
\begin{proof}
For every $n \in \sZ$
\[
\int_{\!-\infty}^\infty \abs{ a(\tfrac{\omega}{2}n + \rmi\tau)g(\tfrac{\omega}{2}n + \rmi\tau) }^2 \,\rmd\tau \leq
\frac{c^2}{n^4}\int_{\!-\infty}^\infty \abs{ g(\tfrac{\omega}{2}n + \rmi\tau) }^2 \,\rmd\tau = \frac{c^2}{n^4}\norm{g_n}_2^2,
\]
hence,
\[
\sup_{n \in \sZ} \norm{a_n g_n}_2 \leq c\norm{g}_{2, \infty} \sup_{n \in \sZ} \frac{1}{n^2} = c\norm{g}_{2, \infty} < \infty.
\]
Therefore,
\[
h_n(z) = \frac{1}{2\pi}\int_{\!-\infty}^\infty \frac{ a(\frac{\omega}{2}n + \rmi\tau)g(\frac{\omega}{2}n + \rmi\tau) }{\frac{\omega}{2}n + \rmi\tau - z} \,\rmd\tau \in \Scr{H}_\omega,
\]
and $\norm{h_n}_\omega \leq \frac{c}{n^2}\norm{g}_{2, \infty}$. Now, for $z \in \wpln$
\begin{align*}
\abs{ h_n(z) } &\leq \frac{1}{2\pi}\bigg(\int_{\!-\infty}^\infty \abs{ a(\tfrac{\omega}{2}n + \rmi\tau)g(\tfrac{\omega}{2}n + \rmi\tau) }^2 \,\rmd\tau\bigg)^\frac{1}{2}
\bigg( \int_{\!-\infty}^\infty \frac{\rmd\tau}{ \abs{\frac{\omega}{2}n + \rmi\tau - z}^2 } \bigg)^\frac{1}{2}\\
&\leq \frac{c}{2\pi n^2}\norm{g}_{2, \infty}\bigg( \int_{\!-\infty}^\infty \frac{\rmd\tau}{ \abs{\frac{\omega}{2}n + \rmi\tau - z}^2 } \bigg)^\frac{1}{2}\\
&= \frac{c}{ 2\sqrt{\pi} }\norm{g}_{2, \infty}\frac{1}{ n^2 \sqrt{ \abs{ \frac{\omega}{2}n - \re{z} } } }.
\end{align*}
An application of Weierstrass M-test shows that the series $\sideset{}{'}\sum\limits_{n = -\infty}^\infty h_n(z)$ converges uniformly on compact subsets of $\wpln$ to a function $f$
which is holomorphic in $\wpln$ since $h_n$ are also holomorphic. Moreover, since $h_n \in \Scr{H}_\omega$ we have
\[
\norm{f}_\omega \leq \sideset{}{'}\sum_{n = -\infty}^\infty \norm{ h_n(z) }_\omega \leq c\norm{g}_{2, \infty}\sideset{}{'}\sum_{n = -\infty}^\infty \frac{1}{n^2} < \infty,
\]
thus $f$ belongs to $\Scr{H}_\omega$.

Finally, fix a non-zero integer $m$. For $z \in S_{m - 1} \cup S_m$, the functions $h_n(z)$ are holomorphic when $n \neq m$. Also,
$(\op{J}h_m)_m(y) = a(\frac{\omega}{2}m + \rmi y)g(\frac{\omega}{2}m + \rmi y)$. Therefore,
\begin{align*}
(\op{J}f)_m(y) &= \bigg( \op{J}\sideset{}{'}\sum_{n = -\infty}^\infty h_n(z) \bigg)(m, y)\\
&= \op{J}\bigg( \sideset{}{'}\sum_{n \neq m} h_n(z) + h_m(z) \bigg)(m, y)\\
&= (\op{J}h_m)(m, y) = a(\tfrac{\omega}{2}m + \rmi y)g(\tfrac{\omega}{2}m + \rmi y).\qedhere
\end{align*}
\end{proof}
Introduce the operator $\op{S} \colon ( \ell^\infty( L^2( \bb{R} ) ), \,\norm{\cdot}_{2, \infty} ) \to (\Scr{H}_\omega, \,\norm{\cdot}_\omega)$ defined by
\begin{equation}
(\op{S}g)(x, y; z) \= \sideset{}{'}\sum_{n = -\infty}^\infty \frac{1}{2\pi}\int_{\!-\infty}^\infty
\frac{F(\frac{\omega}{2}n, \tau)g_n(\tau) }{\frac{\omega}{2}n + \rmi\tau - z}\rme^{ -\rmi\omega n(x - 2\tau y) } \,\rmd\tau,
\end{equation}
for all $(x, y) \in \Omega$ where the function $F$ has the following properties:
\begin{equation}\label{Eq:scatteringbound1}
\abs{ F(\tfrac{\omega}{2}n, \tau) } \leq \frac{c}{n^2}, \quad \forall \:\tau \in \bb{R}, \,n \in \sZ,
\end{equation}
for some constant $c$ and
\begin{equation}\label{Eq:scatteringbound2}
\int_{\!-\infty}^\infty \abs{ F(\tfrac{\omega}{2}n, \tau) }^2 \,\rmd\tau = O\bigg( \frac{1}{n^4} \bigg), \quad \forall \:n \in \sZ.
\end{equation}
By proposition \ref{Pr:lemma}, we see that $\op{S}$ is bounded linear, $\norm{\op{S}g}_\omega \leq \Gamma_c \norm{g}_{2, \infty}$ and that
$(\op{J}\op{S}g)_n(x, y; \tau) = F(\tfrac{\omega}{2}n, \tau)\rme^{ -\rmi\omega n(x - 2\tau y) }g_n(\tau)$.
\subsection[Recovering the Potential from the Spectral Data]{Recovering the Potential from the Spectral Data}\label{s:recover}
\separate

Returning now to the discussion of the inverse problem, consider the equation
\begin{equation}
\mu(x, y; z) = 1 + \frac{1}{2\pi \rmi}\sideset{}{'}\sum_{n = -\infty}^\infty \int_{L_n} \frac{ (\Scr{S}\mu)(x, y; \zeta) }{\zeta - z} \,\rmd\zeta,
\end{equation}
where $\Scr{S}$ is the spectral operator defined in \eqref{Eq:scatteringoperator} with the function $F$ satisfying \eqref{Eq:scatteringbound1} and 
\eqref{Eq:scatteringbound2}. This equation can be rewritten in the form
\begin{align}\label{Eq:inverse}
(\mu - 1)(x, y; z) &= \sideset{}{'}\sum_{n = -\infty}^\infty \frac{1}{2\pi}\int_{\!-\infty}^\infty
\frac{F(\frac{\omega}{2}n, \tau)\rme^{ -\rmi\omega n(x - 2\tau y) } }{\frac{\omega}{2}n + \rmi\tau - z} \,\rmd\tau\notag\\
&{}+ \op{S}\op{l}(\mu - 1)(x, y; z).
\end{align}
If the constant $c$ is such that $\Gamma_{c} < 1$, then the composition $\op{S}\op{l}$ is a contraction on $\Scr{H}_\omega$ since for
$f \in \Scr{H}_\omega$, $\norm{\op{S}\op{l}f}_\omega \leq \Gamma_{c} \norm{\op{l}f}_{2, \infty} \leq \Gamma_{c} \norm{f}_\omega$. Thus, Banach's 
fixed-point theorem implies that the above equation has a unique solution $\mu$ such that $\mu - 1 \in \Scr{H}_\omega$.

Recall equation~\eqref{Eq:shiftedeq}, $[ (\partial_x + \rmi z)^2 - (\partial_y - z^2) ]\mu = -u\mu$ which is written with the shifted derivatives
$D_1, D_2$, defined in~\eqref{Eq:shifted}, as
\begin{equation}
[D_1^2 - D_2 + u]\mu(x, y; z) = 0.
\end{equation}
The following two lemmas are useful for the proof of the Inverse Spectral theorem that follows.
\begin{lemma}\label{L:dressing}
If $\Scr{S}$ has the form
\begin{equation}
(\Scr{S}f)(x, y; z) = F(z)\rme^{ \rmi r_0(z) \cdot (\omega x, y) } f^-( x, y; -\bar{z} ),
\end{equation}
and the function $f(x, y; z)$ has one continuous derivative (in the $x$, $y$ variables), then $[D_1, \Scr{S}]f = [D_2, \Scr{S}]f = 0$ for $z$ on
$\bb{C} \setminus \wpln$.
\end{lemma}
\begin{proof}
A straightforward calculation.
\end{proof}
It is also useful to calculate the commutators of the translated derivatives $D_1, \,D_2$ with the operator $\mathcal{C}$. Since both $\partial_x$ and
$\partial_y$ commute with $\mathcal{C}$, it suffices to calculate $[z^m, \mathcal{C}]$. The following result holds.
\begin{lemma}\label{L:commutatorss}
Suppose $f(x, y; z)$ is a Schwartz function on $E$. Then,
\begin{gather}
[z^m, \mathcal{C}]f(x, y; z) = \frac{1}{2\pi \rmi}\sum_{k = 0}^{m - 1} z^k \sideset{}{'}\sum_{n = -\infty}^\infty \int_{L_n}
\zeta^{m - k - 1} f(x, y; \zeta) \,\rmd\zeta,\\
[D_1^m, \mathcal{C}]f(x, y; z) = \sum_{k = 0}^m \binom{m}{k}\rmi^k \partial_x^{m - k}[z^k, \mathcal{C}]f(x, y; z),
\end{gather}
for all $m \in \bb{N}$ and $z \in \wpln$.
\end{lemma}
\begin{proof}
The second identity follows immediately from the binomial theorem. The first, is a consequence of the identity
\[
z^m - \zeta^m = (z - \zeta)\sum_{k = 0}^{m - 1} z^k \zeta^{m - k - 1}.\qedhere
\]
\end{proof}
Notice that these lemmas imply lemma \ref{L:commutators}. Also Schwartz regularity is not necessary for the proof but makes the calculations easier. The 
results to follow holds with much weaker regularity conditions cf.~\eqref{Eq:scatteringoperator}.

We can now state and prove the following theorem.
\begin{theorem}[The Inverse Spectral Theorem]\label{Th:inversescattering}
Let $\Scr{S}$ be a spectral operator of the form \eqref{Eq:scatteringoperator} defined by a function $F(z)$ that is small in the sense that
\begin{equation}
\sup_{n \in \sZ} n^2 \sup_{ \tau \in \bb{R} } \abs{ F(\tfrac{\omega}{2}n, \tau) } < 1, \quad F(0, \tau) = 0, \,\forall \:\tau \in \bb{R},
\end{equation}
and
\begin{equation}
\bigg(\int_{\!-\infty}^\infty \abs{ F(\tfrac{\omega}{2}n, \tau) }^2 \,\rmd\tau\bigg)^\frac{1}{2} = O\bigg( \frac{1}{n^2} \bigg), \quad \forall \:n \in \sZ.
\end{equation}
Then, the equation
\begin{equation}\label{Eq:CSm}
\mu = 1 + \mathcal{C}\Scr{S}\mu,
\end{equation}
has a unique solution $\mu$ in $L^\infty(E)$, holomorphic in $\wpln$ with jump
\begin{equation}
\mu^+(x, y; z) - \mu^-(x, y; z) = F(z)\rme^{-\rmi( z + \bar{z} )x + (z^2 - \bar{z}^2)y} \mu^-( x, y; -\bar{z} )
\end{equation}
across the contours $L_n, \,n\in \sZ$ for all $(x, y) \in \Omega$. Moreover this function $\mu$ solves the perturbed heat equation
$P(\partial + w)\mu = -u\mu$ for a potential $u(x, y)$ which may be represented by the formula
\begin{equation}\label{Eq:solution}
u(x, y) = \frac{1}{\pi}\partial_x \sideset{}{'}\sum_{n = -\infty}^\infty \int_{L_n} (\Scr{S}\mu)(x, y; z) \,\rmd z.
\end{equation}
\end{theorem}
\begin{proof}
Proposition \ref{Pr:lemma} and the contraction mapping theorem guarantee a unique solution to~\eqref{Eq:inverse} in $\Scr{H}_\omega$ and hence, to
$\mu = 1 + \mathcal{C}\Scr{S}\mu$. The conditions on $F$ imply that $F$ is small in $L^2( \abs{\!\real z} ) \cap L^\infty(\Lambda)$. Thus, proposition 
\ref{Pr:important} yields that $\mu \in L^\infty(E)$. Also by the smallness assumption, $F \in \ell^\infty( L^2( \bb{R} ) )$. Hence, the term
\[
\sideset{}{'}\sum_{n = -\infty}^\infty \frac{1}{2\pi}\int_{\!-\infty}^\infty
\frac{F(\frac{\omega}{2}n, \tau)\rme^{ -\rmi\omega n(x - 2\tau y) } }{\frac{\omega}{2}n + \rmi\tau - z} \,\rmd\tau
\]
in~\eqref{Eq:inverse} represents a function holomorphic in $\wpln$ with respect to $z$. Thus, again by proposition \ref{Pr:lemma} we see that
\[
\op{J}(\mu - 1)_n(x, y; \tau) = F(\tfrac{\omega}{2}n, \tau)\rme^{ -\rmi\omega n(x - 2\tau y) } \mu^-(x, y; -\tfrac{\omega}{2}n + \rmi\tau),
\]
which implies that across the contour $\real z = \frac{\omega}{2}n, \,\imaginary z = \tau \in \bb{R}$,
\[
\mu^+(x, y; z) - \mu^-(x, y; z) = F(z)\rme^{-\rmi( z + \bar{z} )x + (z^2 - \bar{z}^2)y} \mu^-( x, y; -\bar{z} ).
\]

Now applying $P(\partial + w)$ to both sides of the equation $\mu = 1 + \mathcal{C}\Scr{S}\mu$ and using lemma \ref{L:commutatorss} yields
\[
P(\partial + w)\mu = \mathcal{C}\Scr{S}P(\partial + w)\mu - \frac{1}{\pi}\partial_x \sideset{}{'}\sum_{n = -\infty}^\infty \int_{L_n}
\Scr{S}\mu \,\rmd\zeta.
\]
This last term has no $z$-dependence. Call it $u(x, y)$. Then,
\[
( \op{Id} - \mathcal{C}\Scr{S} )P(\partial + w)\mu = -u(x, y),
\]
hence,
\[
P(\partial + w)\mu = -u(x, y)( \op{Id} - \mathcal{C}\Scr{S} )^{-1} 1 = -u(x, y)\mu.\qedhere
\]
\end{proof}
To successfully complete the inverse spectral transform, we should show that the spectral data associated to the potential $u$ defined in \eqref{Eq:solution} 
coincide with the function $F$ of theorem \ref{Th:inversescattering}. Thus, based on the analysis done in the direct problem, it is of great importance to 
show that $\wh{u}$ is small in $L^2 \cap L^\infty( \Scr{C} )$. A sufficient condition for the right hand side of \eqref{Eq:solution} to make sense is a $F$ to 
be a member of the Schwartz class. However, milder conditions can be found.

Taking the Fourier transform (in $(x, y)$) of equation \eqref{Eq:CSm} we obtain
\[
[\mu - 1]\mh(m, \xi; z) = [\mathcal{C}\Scr{S}1]\mh(m, \xi; z) + [ \mathcal{C}\Scr{S}(\mu - 1) ]\mh(m, \xi; z).
\]
First, we will show that this equation has a unique solution in $L^1 (\Scr{C} )$. 

Now, from theorem \ref{Th:inversescattering}, $\mu$ satisfies equation
$P(\partial + w)\mu = -u\mu$, and taking the Fourier transform once again and splitting the right hand side, we have (recall \eqref{Eq:contractionform})
\begin{equation}
P_z(m, \xi)( \wh{\mu - 1} )(m, \xi; z) = \wh{u\mu}(m, \xi; z) = \wh{u}(m, \xi) + \frac{1}{2\pi}\wh{u} \ast \wh{\mu - 1}(m, \xi; z),
\end{equation}
arriving at the following Fredholm integral equation for $\wh{u}$:
\begin{equation}\label{Eq:fouriertransofu}
\wh{u}(m, \xi) = -\frac{1}{2\pi}( \op{K}_{ [\mu - 1]\mh } \wh{u} )(m, \xi; z) + P_z(m, \xi)( \wh{\mu - 1} )(m, \xi; z),
\end{equation}
where $\op{K}_{ [\mu - 1]\mh }$ denotes the operator ``convolution in $(m, \xi)$ by $\wh{\mu - 1}(m, \xi; z)$'' on $L^p( \Scr{C} )$ for
$1 \leq p \leq \infty$. Since $\wh{\mu - 1}(m, \xi; z)$ is in $L^1( \Scr{C} )$, the first term of equation \eqref{Eq:fouriertransofu} has finite norm on
$L^2 \cap L^\infty( \Scr{C} )$. It will be shown that the other term satisfies the equation
\begin{equation}\label{Eq:RFm}
P_z(m, \xi)( \wh{\mu - 1} )(m, \xi; z) = F( \zeta(m, \xi) ) + P_z(m, \xi)\op{R}_F \wh{\mu - 1}(m, \xi; z),
\end{equation}
for some appropriate operator $\op{R}_F$ depending on the spectral data $F$. This dependence, will indicate appropriate conditions on $F$ for the
smallness of $P_z( \wh{\mu - 1} )$ on $L^2 \cap L^\infty( \Scr{C} )$ and the analysis of the inverse problem will be concluded.
\begin{proposition}
Under the assumptions of theorem \ref{Th:inversescattering}, the equation
\begin{equation}\label{Eq:fouriertransequation}
[\mu - 1]\mh(m, \xi; z) = [\mathcal{C}\Scr{S}1]\mh(m, \xi; z) + [ \mathcal{C}\Scr{S}(\mu - 1) ]\mh(m, \xi; z),
\end{equation}
has a unique solution in $L^1 (\Scr{C} )$, uniformly in $z \in \wpln$. Moreover we have the explicit estimate
\begin{equation}\label{Eq:Loneestimate}
\norm{ \wh{\mu - 1}(\cdot, \cdot; z) }_{L^1 ( \Scr{C} ) } \leq \frac{ \norm{F}_{\Lambda} }{ 1 - \norm{F}_{\Lambda} },
\end{equation}
for all $z \in \wpln$, where
\begin{equation}\label{Eq:scatteringnorm}
\norm{F}_{\Lambda} \= C\max\{ 2\norm{F}_{ L^2( \abs{\!\real z} ) }, \,\norm{F}_{ L^\infty(\Lambda) } \}.
\end{equation}
\end{proposition}
\begin{proof}
First, observe that
\begin{align*}
\mathcal{C}\Scr{S}1(x, y; z) &= \frac{1}{2\pi \rmi}\sideset{}{'}\sum_{n = -\infty}^\infty \int_{L_n} \frac{ \Scr{S}1(x, y; \zeta) }{\zeta - z} \,\rmd\zeta\\
&= \frac{1}{2\pi \rmi}\sideset{}{'}\sum_{n = -\infty}^\infty \int_{L_n} \frac{ F(\zeta)\rme^{ \rmi r_0(\zeta) \cdot (\omega x, y) } }{\zeta - z} \,\rmd\zeta\\
&= \frac{1}{2\pi \rmi}\sideset{}{'}\sum_{n = -\infty}^\infty \int_{\!-\infty}^\infty
\frac{ F( \zeta(n, \tau) )\rme^{ \rmi r_0( \zeta(n, \tau) ) \cdot (\omega x, y) } }{-\frac{\omega}{2}n - \rmi\frac{\tau}{2\omega n}- z}
\Big( \frac{-\rmi}{2\omega n} \Big) \,\rmd\tau\\
&= \frac{1}{2\pi}\sideset{}{'}\sum_{n = -\infty}^\infty \int_{\!-\infty}^\infty
\frac{ F( \zeta(n, \tau) )\rme^{ \rmi(n, \tau) \cdot (\omega x, y) } }{ (\omega n)^2 + 2\omega n z + \rmi\tau } \,\rmd\tau\\
&= \frac{1}{2\pi}\sideset{}{'}\sum_{n = -\infty}^\infty \int_{\!-\infty}^\infty
\frac{ F( \zeta(n, \tau) ) }{ P_z(n, \tau) }\rme^{\rmi\omega n x} \rme^{\rmi\tau y} \,\rmd\tau
= \bigg[ \frac{ F( \zeta(n, \tau) ) }{ P_z(n, \tau) } \bigg]\spcheck\!(x, y; z).
\end{align*}
Thus,
\begin{equation}\label{Eq:fouriertransCS1}
[\mathcal{C}\Scr{S}1]\mh(n, \tau; z) = \frac{ (F \circ \zeta)(n, \tau) }{ P_z(n, \tau) }.
\end{equation}
Furthermore, for a function $f$ in the Schwartz class such that its limits as $z$ approaches $L_n$ from the stripes $S_n$ exist we have
\begin{align}\label{Eq:fouriertransCSf}
[\mathcal{C}\Scr{S}f]\mh(m, \xi; z) &= \bigg[\frac{1}{2\pi \rmi}\sideset{}{'}\sum_{n = -\infty}^\infty \int_{L_n}
\frac{ F(\zeta)\rme^{ \rmi r_0(\zeta) \cdot (\omega x, y) } f^-( x, y; -\bar{\zeta} ) }{\zeta - z} \,\rmd\zeta\bigg]\mh(m, \xi; z) \notag\\
&= \frac{1}{2\ell}\int_{\!-\infty}^\infty \int_{\!-\ell}^\ell \bigg[\frac{1}{2\pi \rmi}\sideset{}{'}\sum_{n = -\infty}^\infty \int_{L_n}
\frac{ F(\zeta)\rme^{ \rmi r_0(\zeta) \cdot (\omega x, y) } f^-( x, y; -\bar{\zeta} ) }{\zeta - z} \,\rmd\zeta\bigg] \notag\\
&\ph{=x} \qquad \times \rme^{-\rmi\omega m x - \rmi\xi y} \,\rmd x\rmd y \notag\\
&= \frac{1}{2\pi \rmi}\sideset{}{'}\sum_{n = -\infty}^\infty \int_{L_n}
\frac{ F(\zeta) }{\zeta - z} \,\wh{f^-}( (m, \xi) - r_0(\zeta); -\bar{\zeta} ) \,\rmd\zeta\notag\\
&= \frac{1}{2\pi}\sideset{}{'}\sum_{n = -\infty}^\infty \int_{\!-\infty}^\infty
\frac{ F( \zeta(n,\tau) ) }{ P_z(n, \tau) } \,\wh{f^-}( m - n, \xi - \tau; -\bar{\zeta}(n, \tau) ) \,\rmd\tau \notag\\
&\equiv \op{R}_F \wh{f}(m, \xi; z).
\end{align}

By ``integrating''~\eqref{Eq:fouriertransCSf} first in $(m, \xi)$ and using Fubini's theorem and the \hyperlink{L:basiclemma}{Basic Lemma} we obtain that 
if the $L^1( \Scr{C} )$ norm of $\wh{f}(m, \xi; z)$ is a bounded function of $z$, then so is the $L^1( \Scr{C} )$ norm of $\op{R}_F \wh{f}(m, \xi; z)$. 
Assuming $F(z)$ to be small in $L^2( \abs{\!\real z} ) \cap L^\infty(\Lambda)$, the map $\wh{f} \mapsto \op{R}_F \wh{f}$ is a contraction of
$L^1( \Scr{C} )$ for every $z \in \wpln$. Likewise, $(F \circ \zeta)(m, \xi)/P_z(m, \xi)$ is in $L^1 (\Scr{C} )$ uniformly in $z$ by the 
\hyperlink{L:basiclemma}{Basic Lemma}. Thus, equation~\eqref{Eq:fouriertransequation} has a unique solution $\wh{\mu - 1}(m, \xi; z)$ in
$L^1 (\Scr{C} )$ uniformly in $z \in \wpln$.

Using equations \eqref{Eq:fouriertransCS1} and \eqref{Eq:fouriertransCSf}, one can see that the function $\wh{\mu - 1}$ satisfies the Fredholm integral
equation
\begin{equation}\label{Eq:newfredholm}
\wh{\mu - 1}(m, \xi; z) = \frac{ (F \circ \zeta)(m, \xi) }{ P_z(m, \xi) } + \op{R}_F \wh{\mu - 1}(m, \xi; z).
\end{equation}
Then, the \hyperlink{L:basiclemma}{Basic Lemma} yields
\begin{equation}
\Norm{ \frac{F \circ \zeta}{P_z} }_{L^1( \Scr{C} ) } \leq \norm{F}_{\Lambda}.
\end{equation}
Also,
\begin{equation}
\norm{ \op{R}_F \wh{\mu - 1} }_{L^1( \Scr{C} ) } \leq \norm{ \wh{\mu - 1} }_{L^1( \Scr{C} ) } \norm{F}_{\Lambda}.
\end{equation}
Thus, we get
\[
\norm{ \wh{\mu - 1}(\cdot, \cdot; z) }_{L^1 ( \Scr{C} ) } \leq \frac{ \norm{F}_{\Lambda} }{ 1 - \norm{F}_{\Lambda} },
\]
which is independent of $z$.
\end{proof}
\begin{remark}
The inverse Fourier transform $\mu(x, y; z)$ of the unique solution of equation \eqref{Eq:fouriertransequation} must be the unique solution to
$\mu = 1 +\mathcal{C}\Scr{S}\mu$. From this we get that for each $z \in \wpln$, $\mu(x, y; z) \to 1$ as
$\abs{y} \to \infty$, which is a direct consequence of the Riemann--Lebesgue lemma. 
\end{remark}

By Young's inequality and the estimate~\eqref{Eq:Loneestimate}, we have the estimate for the norm of the operator $\op{K}_{ [\mu - 1]\mh }$ :
\begin{equation}\label{Eq:young}
\norm{ \op{K}_{ [\mu - 1]\mh } }_\mrm{op} \leq \frac{ \norm{F}_{L_n} }{ 1 - \norm{F}_{L_n} },
\end{equation}
for $1 \leq p \leq \infty$, uniformly in $z \in \wpln$. Therefore, an estimation of the $L^2$ and $L^\infty$ norms of $\op{K}_{ [\mu - 1]\mh } \wh{u}$ is 
immediately provided. Now, multiplying equation \eqref{Eq:newfredholm} by $P_z$ we arrive at equation \eqref{Eq:RFm}. The following lemma which allows 
us to ``commute'' $P_z$ and $\op{R}_F$ shows that \eqref{Eq:RFm} can be written as a Fredholm integral equation:
\begin{equation}\label{Eq:Mfredholm}
P_z \wh{\mu - 1} = F \circ \zeta + \op{A}_F (P_z \wh{\mu - 1} ).
\end{equation}
\begin{lemma}
Suppose $f(m, \xi; z)$ is in the Schwartz class over $\Scr{C} \times \wpln$ and the limits $f^\pm( m, \xi; \zeta )$ as $z$ approaches $L_n$ from the strips 
$S_n$ exist. Then,
\begin{gather}
\op{R}_F f = \op{R}_{ (F/P_z) }(P_z f), \label{Eq:AFidentityone}\\[2pt]
P_z(m, \xi)\op{R}_{ (F/P_z) } f(m, \xi; z) = \op{R}_F f(m, \xi; z) - \op{R}_F f( m, \xi; \zeta(m, \xi) ) \equiv 
\op{A}_F f(m, \xi; z). \label{Eq:AFidentitytwo}
\end{gather}
\end{lemma}
\begin{proof}
For equation \eqref{Eq:AFidentityone}, observe that
\begin{align*}
\op{R}_{ (F/P_z) }(P_z f) &= \frac{1}{2\pi \rmi}\sideset{}{'}\sum_{n = -\infty}^\infty \int_{L_n} \frac{ F(\zeta) }{ P_\zeta(m, \xi)(\zeta - z) }\\
&\ph{=x} \qquad \times P_{ -\bar{\zeta} }( (m, \xi) - r_0(\zeta) )f^-( (m, \xi) - r_0(\zeta); -\bar{\zeta} ) \,\rmd\zeta.
\end{align*}
The result follows from the definition of $P_z$:
\begin{align*}
P_{ -\bar{\zeta} }( (m, \xi) - r_0(\zeta) ) &= P_{ -\bar{\zeta} }( (m, \xi) - [-( \zeta + \bar{\zeta} )/\omega, -\rmi(\zeta^2 - \bar{\zeta}^2) ] )\\
&= P_{ -\bar{\zeta} }( m + ( \zeta +\bar{\zeta} )/\omega, \xi + \rmi(\zeta^2 - \bar{\zeta}^2) )\\
&= -P(\rmi\omega m + \rmi( \zeta + \bar{\zeta} ) - \rmi\bar{\zeta}, \rmi\xi + \bar{\zeta}^2 - \zeta^2 - \bar{\zeta}^2)\\
&= -P(\rmi\omega m + \rmi\zeta, \rmi\xi - \zeta^2) = P_\zeta(m, \xi),
\end{align*}
hence the two polynomials cancel. Likewise, equation~\eqref{Eq:AFidentitytwo} is a consequence of the following observation:
\begin{align*}
\frac{ P_z(m, \xi) }{ (\zeta - z)P_\zeta(m, \xi) } &= \frac{ (\omega m)^2 + 2\omega m z + \rmi\xi }
{ (\zeta - z)( (\omega m)^2 + 2\omega m\zeta + \rmi\xi ) }\\
&= \frac{\frac{\omega m}{2} + \rmi\frac{\xi}{2\omega m} + z}{ (\zeta - z)(\frac{\omega m}{2} + \rmi\frac{\xi}{2\omega m} + \zeta) }
= \frac{-\zeta(m, \xi) + z}{ (\zeta - z)(-\zeta(m, \xi) + \zeta) }\\
&= \frac{1}{\zeta - z} - \frac{1}{ \zeta - \zeta(m, \xi) }.\qedhere
\end{align*}
\end{proof}

In order to ensure continuity of $\op{A}_F$ on $L^2 \cap L^\infty( \Scr{C} )$ we will restrict $F(z)$ to be a member of an appropriate
subspace of $L^2( \abs{\!\real z} ) \cap L^\infty(\Lambda)$, achieving the desired behaviour of $\op{A}_F$.
\begin{definition}
Let $k$ be a nonnegative integer and for $(a, b) \in \bb{C}^2$ set
\[
\langle a, b \rangle^k \equiv ( 1 + \abs{ (a, b) } )^k.
\]
We define the $k$\textsuperscript{th} weighted subspace of $L^2 \cap L^\infty( \Scr{C} )$ as
\[
W^k \equiv W^k ( L^2 \cap L^\infty( \Scr{C} ) ) \=
\{ f(q) \in L^2 \cap L^\infty( \Scr{C} ) \colon \langle q \rangle^k \!f(q) \in L^2 \cap L^\infty( \Scr{C} ) \}.
\]
\end{definition}
\noi This is a Banach space with the norm
\[
\norm{f}_{W^k} \= \sum_{j = 0}^k \binom{k}{j}\norm{ \langle q \rangle^j \!f(q) }_{ L^2 \cap L^\infty( \Scr{C} ) },
\]
where
\[
\norm{f}_{ L^2 \cap L^\infty( \Scr{C} ) } = \norm{f}_{ L^2( \Scr{C} ) } + \norm{f}_{ L^\infty( \Scr{C} ) }.
\]
\begin{definition}
The $k$\textsuperscript{th} weighted subspace of $L^2( \abs{\!\real z} ) \cap L^\infty(\Lambda)$ is denoted by $W_{\zeta}^k \equiv
W_{\zeta}^k ( L^2( \abs{\!\real z} ) \cap L^\infty(\Lambda) )$ and consists of those functions $f(z)$ for which $f \circ \zeta(q) \in W^k$.
This is a Banach space with the norm
\[
\norm{f}_{W_\zeta^k} \= \norm{f \circ \zeta}_{W^k}.
\]
Finally, if $f(q; z) \in L^\infty(\Scr{C} \times \wpln)$, introduce the function
\[
f^\star(q) \= \essup_{z \in \wpln} \abs{ f(q; z) },
\]
and define the $k$\textsuperscript{th} weighted max subspace
\begin{align*}
W_\infty^k = W_\infty^k ( L^\infty(\Scr{C} \times \wpln) ) &\= \{f \in L^\infty(\Scr{C} \times \wpln) \colon \langle q \rangle^j \!f^\star(q) \\
&\ph{\=} \quad \text{ is essentially bounded for all } 0 \leq j \leq k\}.
\end{align*}
\end{definition}
\noi Again, this is a Banach space with the norm
\[
\norm{f}_{W_\infty^k} \= \max_{0 \leq j \leq k} \norm{ \langle q \rangle^j \!f^\star(q) }_{ L^\infty( \Scr{C} ) }.
\]

The spaces $W^k$ and $W_\infty^k$ satisfy the following embedding properties.
\begin{proposition}\label{Pr:embeddingthm}
For every nonnegative integer $k$,
\begin{equation}
W^k \subset W_\infty^k \;\text{ and }\; \norm{f}_{W_\infty^k} \leq \norm{f}_{W^k}, \ \text{ for } f \in W^k.
\end{equation}
Moreover, $W_\infty^{k + 2} \subset W^k$ in the sense that if $f \in W_\infty^{k + 2}$, then $f^\star \in W^k$ and the embedding inequality
\begin{equation}\label{Eq:embedding}
\norm{f^\star}_{W^k} < 3 \cdot 2^k \norm{f}_{ W_\infty^{k + 2} }
\end{equation}
holds.
\end{proposition}
\begin{proof}
The first embedding follows immediately from the definitions of these spaces and $f^\star$ (notice that if $f \in W^k$, then $f^\star = \abs{f}$).
Likewise, since $1/\langle q \rangle^2$ belongs to $L^2(\Scr{C}^*)$ with norm less than 2 and $f \in W_\infty^{k + 2}$, then
$\langle q \rangle^j \!f^\star(q)$ is essentially bounded for all $0 \leq j \leq k + 2$, in particular $f^\star(q) \in L^\infty( \Scr{C} )$
and $\langle q \rangle^k \!f^\star(q) \in L^\infty( \Scr{C} )$. Also, for all $0 \leq r \leq k$
\begin{align*}
\sideset{}{'}\sum_{m = -\infty}^\infty \int_{\!\infty}^\infty \abs{ \langle m, \xi \rangle^r \!f^\star(m, \xi) }^2 \,\rmd\xi &=
\sideset{}{'}\sum_{m = -\infty}^\infty \int_{\!\infty}^\infty
\frac{\abs{ \langle m, \xi \rangle^{r + 2} \!f^\star(m, \xi) }^2}{\abs{ \langle m, \xi \rangle^2 }^2} \,\rmd\xi\\
&\leq \norm{ \langle q \rangle^{r + 2} \!f^\star(q) }_{ L^\infty( \Scr{C} ) }^2 \Norm{ \frac{1}{\langle q \rangle^2} }_{ L^2(\Scr{C}^*) }^2,
\end{align*}
hence, $\langle q \rangle^r \!f^\star(q) \in L^2( \Scr{C} )$ and so $f^\star(q) \in L^2( \Scr{C} )$ and
$\langle q \rangle^k \!f^\star(q) \in L^2( \Scr{C} )$ as well. Therefore, $f^\star \in W^k$. Comparing norms,
\begin{align*}
\norm{f^\star}_{W^k} &= \sum_{j = 0}^k \binom{k}{j}\norm{ \langle m, \xi \rangle^j \!f^\star(m, \xi) }_{ L^2 \cap L^\infty( \Scr{C} ) }\\
&= \sum_{j = 0}^k \binom{k}{j}
\big[ \norm{ \langle m, \xi \rangle^j \!f^\star(m, \xi) }_{L^2( \Scr{C} ) } + \norm{ \langle m, \xi \rangle^j \!f^\star(m, \xi) }_{L^\infty( \Scr{C} ) } \big]\\
&\leq \sum_{j = 0}^k \binom{k}{j} \bigg[ \norm{ \langle m, \xi \rangle^{j + 2} \!f^\star(m, \xi) }_{ L^\infty( \Scr{C} ) }
\Norm{ \frac{1}{\langle m, \xi \rangle^2} }_{ L^2( \Scr{C} ) } + \norm{f}_{ W_\infty^{k + 2} } \bigg]\\
&< \sum_{j = 0}^k \binom{k}{j} \big[ \,2\norm{f}_{ W_\infty^{k + 2} } + \norm{f}_{ W_\infty^{k + 2} } \big].
\end{align*}
Thus, we get the inequality \eqref{Eq:embedding}.
\end{proof}

We can now obtain a bound for $\op{A}_F$ on these subspaces.
\begin{lemma}
Let $F \in W_\zeta^k$ and $f \in W_\infty^k$. Then, for $k \in \bb{N}_0$
\[
\op{A}_F f \in W_\infty^k,
\]
and
\[
\norm{\op{A}_F}_\mrm{op} \leq \frac{C}{\pi} \norm{F}_{W_\zeta^k}.
\]
\end{lemma}
\begin{proof}
By the definition of $\op{A}_F$ and the triangle inequality
\[
\norm{\op{A}_F f}_{W_\infty^k} \leq 2\norm{\op{R}_F f}_{W_\infty^k}.
\]
Now
\begin{align*}
\langle m, \xi \rangle = 1 + \abs{ (m, \xi) } &= 1 + \abs{ (m - m', \xi - \xi') + (m', \xi') }\\
&\leq 1 + \abs{ (m - m', \xi - \xi') } + \abs{ (m', \xi') }\\
&< \langle m - m', \xi - \xi' \rangle + \langle m', \xi' \rangle.
\end{align*}
Thus, by the binomial theorem
\[
\langle m, \xi \rangle^k \leq \sum_{j = 0}^k \binom{k}{j} \langle m - m', \xi - \xi' \rangle^{k - j} \langle m', \xi' \rangle^j.
\]
Therefore,
\begin{align*}
\abs{ \langle m, \xi \rangle^k \op{R}_F f(m, \xi; z) }&\leq \frac{1}{2\pi}\sideset{}{'}\sum_{n = -\infty}^\infty \int_{\!-\infty}^\infty
\Abs{ \frac{ \langle m, \xi \rangle^k F( \zeta(n,\tau) ) }{ P_z(n, \tau) } }\\
&\ph{\leq x} \quad \times \abs{ f^-( m - n, \xi - \tau; -\bar{\zeta}(n, \tau) ) } \,\rmd\tau\\
&\leq \frac{1}{2\pi}\sum_{j = 0}^k \binom{k}{j} \sideset{}{'}\sum_{n = -\infty}^\infty \int_{\!-\infty}^\infty
\Abs{ \frac{ \langle n, \tau \rangle^j F( \zeta(n,\tau) ) }{ P_z(n, \tau) } }\\
&\ph{\leq x} \quad \times \abs{ \langle m - n, \xi - \tau \rangle^{k - j} f^-( m - n, \xi - \tau; -\bar{\zeta}(n, \tau) ) } \,\rmd\tau\\
&\leq \frac{1}{2\pi}\norm{f}_{W_\infty^k} \sum_{j = 0}^k \binom{k}{j}\sideset{}{'}\sum_{n = -\infty}^\infty \int_{\!-\infty}^\infty
\Abs{ \frac{ \langle n, \tau \rangle^j F \circ \zeta(n,\tau) }{ P_z(n, \tau) } } \,\rmd\tau\\
&\leq \frac{C}{2\pi}\norm{f}_{W_\infty^k}
\sum_{j = 0}^k \binom{k}{j}\norm{ \langle n, \tau \rangle^j F \circ \zeta(n,\tau) }_{ L^2 \cap L^\infty( \Scr{C} ) }\\
&= \frac{C}{2\pi}\norm{f}_{W_\infty^k} \norm{F}_{W_\zeta^k}.\qedhere
\end{align*}
\end{proof}

Finally we have the following theorem.
\begin{theorem}\label{Th:inversedecay}
Suppose that $(1 + (\real z)^2 + (\real z \imaginary z)^2)F(z)$ is sufficiently small in $L^2( \abs{\!\real z} ) \cap L^\infty(\Lambda)$. Then, there exists a 
function $u(x, y) \in L^2(\Omega)$ with bounded Fourier transform, such that $F(z)$ is the spectral data associated to $u$.
\end{theorem}
\begin{proof}
Using equation \eqref{Eq:Mfredholm} and the contraction mapping principle in $W_\infty^k$, we obtain the estimate
\[
\norm{P_z \wh{\mu - 1} }_{W_\infty^k} \leq \frac{ \norm{F \circ \zeta}_{W_\infty^k} }{ 1 - \frac{C}{\pi}\norm{F}_{W_\zeta^k} }.
\]
On the other hand, comparing norms with the aid of inequality \eqref{Eq:embedding}, and observing that $W^0 = L^2 \cap L^\infty( \Scr{C} )$, one obtains 
(for $k = 2$)
\[
\norm{P_z \wh{\mu - 1} }_{ L^2 \cap L^\infty( \Scr{C} ) } < \frac{ 3\norm{F \circ \zeta}_{W_\infty^2} }{ 1 - \frac{C}{\pi}\norm{F}_{W_\zeta^2} }
\leq \frac{ 3\norm{F}_{W_\zeta^2} }{ 1 - \frac{C}{\pi}\norm{F}_{W_\zeta^2} },
\]
this estimate being uniform in $z$. Combining it with \eqref{Eq:young} and \eqref{Eq:fouriertransofu} yields
\begin{align}\label{Eq:2inftyestimate}
\norm{ \wh{u} }_{ L^2 \cap L^\infty( \Scr{C} ) } &\leq \frac{ \norm{P_z \wh{\mu - 1} }_{ L^2 \cap L^\infty( \Scr{C} ) } }
{ 1 - \frac{1}{2\pi}\norm{ \op{K}_{ [\mu - 1]\mh } }_\mrm{op} } \notag\\
&\leq \bigg( \frac{ 1 - \norm{F}_{\Lambda} }{ 2\pi - (2\pi + 1)\norm{F}_{\Lambda} } \bigg)
\bigg( \frac{ 6\pi\norm{F}_{W_\zeta^2} }{ 1 - \frac{C}{\pi}\norm{F}_{W_\zeta^2} } \bigg).
\end{align}

$F$ is a member of $W_\zeta^2$ as well as $L^2( \abs{\!\real z} ) \cap L^\infty(\Lambda)$ because of the decay hypothesis on $F$. Let $\Scr{S}$ be
the spectral operator associated to $F$
and $\mu(x, y; z)$ the unique solution to $\mu = 1 + \mathcal{C}\Scr{S}\mu$ provided by the Inverse Spectral theorem. This $\mu$ is a solution to 
equation $P(\partial + w)\mu = -u\mu$ with $u(x, y)$ given by
\[
u(x, y) = \frac{1}{\pi}\partial_x \sideset{}{'}\sum_{n = -\infty}^\infty \int_{L_n} (\Scr{S}\mu)(x, y; z) \,\rmd z.
\]
By \eqref{Eq:2inftyestimate} and Plancherel's theorem, $u$ is small in $L^2 (\Omega)$ and has a small bounded Fourier transform. By the proof of theorem 
\ref{Th:exist_unique}, this guarantees that there is exactly one solution $\mu'$ to equation $P(\partial + w)\mu = -u\mu$ that has the property
$\wh{\mu' - 1} \in L^1( \Scr{C} )$. This fact allows the extension of the forward spectral transform to the ball
$\{ u \in L^2(\Omega) \colon \max\{\sqrt{\omega}\norm{u}_2, \,\norm{ \wh{u} }_\infty\} < \frac{2\pi}{C} \}$ by using
$(\Scr{S}'\mu')(x, y; z) = \mu'^+(x, y; z) - \mu'^-(x, y; z)$. This extension is well defined and agrees with the definition of the spectral data 
\eqref{Eq:scatteringoperator} on the
ball $\{ u \in L^1(\Omega) \cap L^2(\Omega) \colon \max\{\omega\norm{u}_1, \,\sqrt{\omega}\norm{u}_2\} < \frac{2\pi}{C} \}$.

Now $\mu$ solves $P(\partial + w)\mu = -u\mu$ and it also has the property $\wh{\mu - 1} \in L^1( \Scr{C} )$ by~\eqref{Eq:Loneestimate}. By 
uniqueness, $\mu = \mu'$, therefore, the spectral operator $\Scr{S}'$ associated to $u$ is the same as $\Scr{S}$.
\end{proof}
\section[Temporal Evolution of the Spectral Data]{Temporal Evolution of the Spectral Data}\label{S:time_evolution}
\separate

The machinery of the IST described in the previous sections suggests a method of solving the initial-value problem formulated by \eqref{Eq:cauchyproblem}
and \eqref{Eq:zeromass}. Considering $u$ to be a potential for the heat operator, if $u$ evolves like \eqref{Eq:cauchyproblem_a}, then $F = F(z, t)$
satisfies a linear evolution equation. Indeed, the \textsc{KP}II equation is the compatibility condition between the perturbed heat operator \eqref{Eq:Lpart}
$\op{L} = -\partial_y + \partial_x^2 + u$  and the evolution operator $\op{B}$ given by
\begin{equation}
\op{B} \= \frac{d}{dt} - \op{M},
\end{equation}
where the operator $\op{M}$ is defined in \eqref{Eq:Mpart}. These conditions are equivalent with the Lax equation
\[
\frac{d}{dt}\op{L} = [ \op{M}, \op{L} ].
\]

To see that the evolution of $u$ via the \textsc{KP}II equation corresponds to linear evolution of the spectral data, consider asymptotically exponential 
solutions $\psi$ to $\op{L}\psi = 0$ as $\abs{y} \to \infty$, with $u$ small. Then, one can write $\psi(x, y, t; z) = \mu(x, y, t; z)\rme^{\rmi z x - z^2 y}$ 
where
$\mu(x, y, t; z)$ is bounded in all variables. Suppose that $\psi$ evolves so as to satisfy $(d/dt)\psi = \op{M}\psi$, i.e., $\op{B}\psi = 0$. From the
asymptotic behaviour of $\mu$, namely $\mu \sim 1$ as $\abs{y} \to \infty$, we can determine that $\emph{\text{\textalpha}}(z) = 4\rmi z^3$.
Thus, $\psi$ satisfies the system
\begin{align}
&-\psi_y + \psi_{xx} + u\psi = 0, \notag\\
&\frac{d}{dt}\psi = 4\psi_{xxx}+ 6u\psi_x + \bigg(3u_x + 3\int_{\!-\ell}^x u_y \,\rmd s + 4\rmi z^3\bigg)\psi.
\end{align}
In this way the evolution of $u$ results in an evolution of $\psi$, hence, of the asymptotic behaviour of $\psi$.
\begin{lemma}\label{L:temporalevol}
Suppose $\psi \sim \rme^{\rmi z x - z^2 y}$ as $\abs{y} \to \infty$ and that $\psi \in \ker \op{B}(z)$. Then,
\begin{equation}\label{Eq:dataevolution}
\frac{d}{dt}F(z, t) = -4\rmi(z^3 + \bar{z}^3)F(z, t).
\end{equation}
\end{lemma}
\begin{proof}
First observe that for $z, k \in \bb{C}$, $\op{B}(z) - \op{B}(k) = \emph{\text{\textalpha}}(z) - \emph{\text{\textalpha}}(k) = 4\rmi(z^3 - k^3)$. Also,
$[ \op{B}, \op{J} ] = 0$ since $\emph{\text{\textalpha}}(z) \in C( \bb{C} )$. Now, equation \eqref{Eq:RiemannHilbert}
yields
\[
\op{J}\psi(x, y, t; z) = F(z, t)\psi^-( x, y, t; -\bar{z} ).
\]
Thus, since $\psi \in \ker \op{B}(z)$ we get
\begin{align*}
0 &= \op{B}(z)\op{J}\psi(x, y, t; z) = \op{B}(z)F(z, t)\psi^-( x, y, t; -\bar{z} )\\
&= \op{B}(k)F(z, t)\psi^-( x, y, t; -\bar{z} ) + 4\rmi(z^3 - k^3)F(z, t)\psi^-( x, y, t; -\bar{z} )\\
&= \psi^-( x, y, t; -\bar{z} )\Big[ \frac{d}{dt}F(z, t) + 4\rmi(z^3 - k^3)F(z, t) \Big] + F(z, t)\op{B}(z)\psi^-( x, y, t; -\bar{z} )
\end{align*}
Utilizing once more the continuity of $\emph{\text{\textalpha}}(z)$, yields $\op{B}(z)\psi^-( x, y, t; -\bar{z} ) = 0$. Setting $k = -\bar{z}$
completes the proof.
\end{proof}
Let us now calculate the evolution of $u$. For convenience, let $\dot{f}$ denote $df/dt$. 
\begin{lemma}\label{L:evolution}
The evolution of $u$ is given by
\begin{equation}
\dot{u}(x, y, t) = \frac{1}{\pi}\partial_x \sideset{}{'}\sum_{n = -\infty}^\infty \int_{L_n} \wt{\mu}(x, y; z) (\dot{ \Scr{S} }\mu)(x, y; z) \,\rmd z,
\end{equation}
where $P(-\partial + w)\wt{\mu} = -u\wt{\mu}$ or $\op{J}\wt{\mu} = \Scr{S}^\ast \wt{\mu}$.
\end{lemma}
\begin{proof}
From \eqref{Eq:solution} one has
\[
\pi\dot{u} = \partial_x \sideset{}{'}\sum_{n = -\infty}^\infty \int_{L_n} (\Scr{S}\mu)\spdot \,\rmd z.
\]
But $(\Scr{S}\mu)\spdot = (\Scr{S}1)\spdot + ( \Scr{S}(\mathcal{C}\Scr{S}\mu) )\spdot =
\dot{ \Scr{S} }1 + \dot{ \Scr{S} }(\mathcal{C}\Scr{S}\mu) + \Scr{S}\mathcal{C}(\Scr{S}\mu)\spdot$. Therefore,\\
$( \op{Id} - \Scr{S}\mathcal{C} )(\Scr{S}\mu)\spdot = \dot{ \Scr{S} }(1 + \mathcal{C}\Scr{S}\mu) = \dot{ \Scr{S} }\mu$. It follows that
\begin{align}
\pi\dot{u} &= \partial_x \sideset{}{'}\sum_{n = -\infty}^\infty \int_{L_n} ( \op{Id} - \Scr{S}\mathcal{C} )^{-1} \dot{ \Scr{S} }\mu \,\rmd z\nonumber\\
&= \partial_x \sideset{}{'}\sum_{n = -\infty}^\infty \int_{L_n} (\op{Id} - \mathcal{C}^t \Scr{S}^t)^{-1} \dot{ \Scr{S} }\mu \,\rmd z\nonumber\\
&= \partial_x \sideset{}{'}\sum_{n = -\infty}^\infty \int_{L_n} \wt{\mu}\dot{ \Scr{S} }\mu \,\rmd z,
\end{align}
where $\wt{\mu}$ satisfies the equation $\wt{\mu} = (\op{Id} - \mathcal{C}^t \Scr{S}^t)^{-1} 1 = (\op{Id} - \mathcal{C}\Scr{S}^\ast)^{-1} 1$
with $\Scr{S}^\ast = -\Scr{S}^t$. The superscript ${}^t$ denotes transposition with respect to the inner product on
$\Omega \times \bb{C} \setminus \wpln$ given by
\begin{equation}
\langle f, g \rangle \= \int_{\!-\infty}^\infty \int_{\!-\ell}^\ell \sideset{}{'}\sum_{n = -\infty}^\infty \int_{L_n} f(x, y; z)g(x, y; z) \,\rmd z\rmd x\rmd y.
\end{equation}
According to this inner product, $\partial^t = -\partial$ and $\mathcal{C}^t = -\mathcal{C}$. The function $\wt{\mu}(x, y; z)$ can be seen to satisfy the
transposed equation $P(-\partial + w)\wt{\mu} = -u\wt{\mu}$.

As a consequence of lemma \ref{L:commutators}, $[P(-\partial + w), \Scr{S}^\ast] = [ P(\partial + w), \Scr{S} ]^t = 0$ and
\[
[ P(-\partial + w), \mathcal{C} ]\Scr{S}^\ast \wt{\mu}(x, y) =
\frac{1}{\pi}\partial_x \sideset{}{'}\sum_{n = -\infty}^\infty \int_{L_n} (\Scr{S}^\ast \wt{\mu})(x, y; \zeta) \,\rmd z \equiv \wt{u}(x, y),
\]
hence, $P(-\partial + w)\wt{\mu} = (\op{Id} - \mathcal{C}\Scr{S}^\ast)^{-1} [ P(-\partial + w), \mathcal{C} ]\Scr{S}^\ast \wt{\mu} = \wt{u}\wt{\mu}$.
Meanwhile, $\wt{u}(x, y) = -u(x, y)$, since $\Scr{S}\mu = ( \op{Id} - \Scr{S}\mathcal{C} )^{-1} \Scr{S}1 = -\Scr{S}^\ast \wt{\mu}$.
\end{proof}

Based on lemma \ref{L:evolution}, we have the following global existence theorem which provides the solution of the
initial-value problem \eqref{Eq:cauchyproblem}\textendash\eqref{Eq:zeromass}.
\begin{theorem}\label{Th:solution}
Suppose the function $u_0(x, y)$ has small derivatives up to order $8$ in $L^1(\Omega) \cap L^2(\Omega)$. Then, the initial-value problem
\eqref{Eq:cauchyproblem}\textendash\eqref{Eq:zeromass} has a solution $u(x, y, t)$ for all $t \geq 0$, uniformly bounded for all $t$ in $L^2(\Omega)$
with bounded Fourier transform.
\end{theorem}
\begin{proof}
Equation \eqref{Eq:dataevolution} is a first order linear ordinary differential equation. Hence
\begin{equation}
F(z, t) = F(z, 0)\rme^{-4\rmi(z^3 + \bar{z}^3)t}.
\end{equation}
For all $z \in \bb{C}$, $\imaginary(z^3 + \bar{z}^3) = 0$. Thus, $F(z, t)$ remains bounded in each $W_\zeta^k$ for all $t$. In particular, setting
$F(z, 0) = F_0(z)$, $\norm{ F(\cdot, t) }_{W_\zeta^k} = \norm{F_0}_{W_\zeta^k}$. The initial value problem for the spectral data $F$
\begin{equation}
\frac{d}{dt}F(z, t) = -4\rmi(z^3 + \bar{z}^3)F(z, t), \qquad F(z, 0) = F_0(z),
\end{equation}
corresponds to the bounded evolution $\dot{ \Scr{S} } = [ \Scr{S}, \emph{\text{\textalpha}} ]$; here $\emph{\text{\textalpha}}$ stands for the operation
of multiplication by $\emph{\text{\textalpha}}(z)$. Thus, the bounded evolutions of $F$ correspond to bounded evolutions of $u$. This evolution is
\begin{align*}
\dot{u} &= \frac{1}{\pi}\partial_x \sideset{}{'}\sum_{n = -\infty}^\infty \int_{L_n} \wt{\mu}[ \Scr{S}, \emph{\text{\textalpha}} ]\mu \,\rmd z
= -\frac{1}{\pi}\partial_x \sideset{}{'}\sum_{n = -\infty}^\infty \int_{L_n}
\emph{\text{\textalpha}}( \wt{\mu}\Scr{S}\mu + \mu\Scr{S}^\ast \wt{\mu} ) \,\rmd z\\
&= -\frac{1}{\pi}\partial_x \sideset{}{'}\sum_{n = -\infty}^\infty \int_{L_n} \emph{\text{\textalpha}}( \wt{\mu}\op{J}\mu + \mu\op{J}\wt{\mu} ) \,\rmd z,
\end{align*}
thus,
\begin{equation}
\dot{u}(x, y, t) = \frac{4}{\pi \rmi}\partial_x \sideset{}{'}\sum_{n = -\infty}^\infty \int_{L_n}
z^3( \wt{\mu}\op{J}\mu + \mu\op{J}\wt{\mu} )(x, y; z) \,\rmd z.
\end{equation}
Observe now that $z^3$ is the coefficient of $1/s^4$ in the geometric series
\[
\sum\limits_{k = 1}^\infty \frac{ z^{k - 1} }{s^k} = \frac{1}{s - z}
\]
for $\abs{z} < \abs{s}$ and for any $s \in \wpln$. This observation suggests the introduction of the linear functional $\phi$ acting on the algebra of
formal power series in $s^{-1}$ with coefficients in some ring $R$ and defined by
\begin{equation}
\phi \bigg( \sum_{k = 0}^\infty \frac{a_k}{s^k} \bigg) \= a_4, \quad a_k \in R.
\end{equation}
Then, $z^3 = \phi(s - z)^{-1}$. Since $\phi$ is linear and commutes with all operations present, we have
\begin{align*}
\dot{u}(x, y, t) &= \phi\frac{4}{\pi \rmi}\partial_x \sideset{}{'}\sum_{n = -\infty}^\infty \int_{L_n}
\frac{ ( \wt{\mu}\op{J}\mu + \mu\op{J}\wt{\mu} )(x, y; z) }{s - z} \,\rmd z\\
&= -8\phi\frac{1}{2\pi \rmi}\partial_x \sideset{}{'}\sum_{n = -\infty}^\infty \int_{L_n}
\frac{ ( \wt{\mu}\op{J}\mu + \mu\op{J}\wt{\mu} )(x, y; z) }{z - s} \,\rmd z.
\end{align*}
Applying the Plemelj--Sokhotski formulae in the above equation yields
\[
\dot{u}(x, y, t) = -8\phi\frac{1}{2\pi \rmi}\partial_x \op{J}( \wt{\mu}\op{J}\mu + \mu\op{J}\wt{\mu} )(x, y; s) =
-16\phi\partial_x \op{J}\mu\op{J}\wt{\mu}(x, y; s),
\]
for $s \in L_{n}$ and $n \in \sZ$. Therefore, $u(x, y, t)$ solves the nonlinear system
\begin{gather}
\dot{u} = -16\phi\partial_x \op{J}\mu\op{J}\wt{\mu},\\
(-\partial_y + \partial_x^2 + 2\rmi s\partial_x)\mu = -u\mu,\\
(\partial_y + \partial_x^2 - 2\rmi s\partial_x)\wt{\mu} = -u\wt{\mu}.
\end{gather}
The asymptotic behaviour as $\abs{\!\imaginary s} \to \infty$ of the functions $\mu$ and $\wt{\mu}$, allows us to express the functions $\op{J}\mu$ and
$\op{J}\wt{\mu}$ as asymptotic series in $s^{-1}$ with coefficients in the ring of smooth functions up to order $3$ if $\Scr{S}$, i.e., $F$ has sufficient
decay. The coefficients can be determined recursively by the following relations
\begin{gather}
m_0(x, y) \equiv 1, \qquad 2i\partial_x m_{k +1}(x, y) = ( \partial_y - \partial_x^2 - u(x, y) )m_k(x, y),\\
\wt{m}_0(x, y) \equiv 1, \qquad 2i\partial_x \wt{m}_{k +1}(x, y) = ( \partial_y + \partial_x^2 + u(x, y) ) \wt{m}_k(x, y),
\end{gather}
where
\begin{equation}
\mu(x, y; s) = \sum_{k = 0}^\infty \frac{ m_k(x, y) }{s^k}, \qquad \wt{\mu}(x, y; s) = \sum_{k = 0}^\infty \frac{ \wt{m}_k(x, y) }{s^k}.
\end{equation}

For $u$ to be a solution of the \textsc{KP}II equation amounts to solving for $m_k$, $\wt{m}_k, \, k \leq 4$, multiply together the two series for $\mu$,
$\wt{\mu}$ and then picking out the coefficient of $s^{-4}$. We conclude the proof by noticing that the order of derivatives of $u$ is high enough to
provide continuity of the forward and inverse spectral transforms.
\end{proof}
\section[Conclusion]{Conclusion}\label{S:conc}
\medskip

In this paper we established that the initial-value problem for the \textsc{KP}II equation with small enough initial data periodic in the $x$ direction and decaying in 
the $y$ direction is rigorously solved by using the \textsc{IST} method via a Riemann--Hilbert problem with shift on the boundary of infinite strips of the complex
plane. Regarding the existence of special solutions for \textsc{KP}II equation, semi-periodic in $x$ or $y$, we know that real but singular solutions were obtained in
\cite{DT21} with the $\conj{\partial}$-dressing method of Zakharov and Manakov. We leave for future study the investigation of other semi-periodic problems for
the \text{KP} equations.

\addcontentsline{toc}{section}{References}
\end{document}